\documentclass[11pt]{amsart} 
\usepackage{amsmath, amsxtra, amsthm, amsfonts, amssymb}
\usepackage{mathrsfs}
\usepackage{graphicx}
\usepackage{url}
\usepackage{color}
\usepackage{bbm}
\usepackage{tikz-cd}

\usepackage[T1]{fontenc}
\usepackage{enumitem}
\usepackage{verbatim}
\usepackage{mathtools}

\usepackage[paperheight=11in, 
    paperwidth=8.5in,
    outer=1in,
    inner=1in,
    bottom=1in,
    top=1in]{geometry}

 \usepackage[hidelinks, backref=page]{hyperref} 
 \hypersetup{
    colorlinks,
    citecolor=magenta,
    filecolor=magenta,
    linkcolor=blue,
    urlcolor=black
}

\usepackage{cleveref}

\newcommand{\QQ}{\mathbb Q}
\newcommand{\RR}{\mathbb R}
\newcommand{\CC}{\mathbb C}
\newcommand{\PP}{\mathbb P}
\newcommand{\ZZ}{\mathbb Z}
\newcommand{\be}{\mathbf e}
\newcommand{\MM}{{\boldsymbol M}}
\newcommand{\bm}{\mathbf m}
\newcommand{\bT}{\mathbf T}
\newcommand{\bt}{\mathbf t}
\newcommand{\cS}{\mathcal S}
\newcommand{\cQ}{\mathcal Q}
\newcommand{\CCinv}{{\mathbb C}_{\textnormal{inv}}}
\newcommand{\uCCinv}{\underline{\mathbb C}_{\textnormal{inv}}}
\newcommand{\PPinv}{\PP_{\textnormal{inv}}}

\newcommand{\rk}{{\operatorname{rk}}}
\newcommand{\crk}{{\operatorname{crk}}}
\newcommand{\Gr}{{\operatorname{Gr}}}
\newcommand{\Fl}{{\operatorname{Fl}}}
\newcommand{\crem}{\operatorname{crem}}
\newcommand{\Td}{\operatorname{Td}}
\newcommand{\ch}{\operatorname{ch}}
\newcommand{\csm}{\operatorname{csm}}
\newcommand{\pt}{\textnormal{pt}}

\theoremstyle{theorem}
\newtheorem{thm}{Theorem}[section]
\newtheorem{cor}[thm]{Corollary}
\newtheorem{lem}[thm]{Lemma}
\newtheorem{prop}[thm]{Proposition}
\newtheorem{maintheorem}{Theorem}	
\theoremstyle{definition}
\newtheorem{defn}[thm]{Definition}
\newtheorem{eg}[thm]{Example}
\newtheorem{rem}[thm]{Remark}

\newtheorem{conjs}[thm]{Conjectures}
\newtheorem{ques}[thm]{Question}

\title{Tautological classes of matroids}
\author{Andrew Berget, Christopher Eur, Hunter Spink, Dennis Tseng}
\date{}

\address{Western Washington University. Bellingham, WA. USA.}
\email{andrew.berget@gmail.com}

\address{Stanford University. Stanford, CA. USA.}
\email{chrisweur@gmail.com}

\address{Stanford University. Stanford, CA. USA.}
\email{hunterspink@gmail.com}

\address{New York, NY. USA}
\email{dennisctseng@gmail.com}

\begin{document}

\maketitle

\begin{abstract}
We introduce certain torus-equivariant classes on permutohedral varieties which we call ``tautological classes of matroids'' as a new geometric framework for studying matroids.
Using this framework, we unify and extend many recent developments in matroid theory arising from its interaction with algebraic geometry.
We achieve this by establishing a Chow-theoretic description and a log-concavity property for a $4$-variable transformation of the Tutte polynomial, and by establishing an exceptional Hirzebruch-Riemann-Roch-type formula for permutohedral varieties that translates between $K$-theory and Chow theory.
\end{abstract}

\tableofcontents

\section{Introduction}
In Ardila's survey on the interaction between matroid theory and algebraic geometry \cite{Ard18}, recent developments are classified according to three geometric models for matroids.
However, developments from these different geometric models remained partially disjoint, as evidenced in \Cref{ques:intro1.1} below.
We introduce a new unifying framework we call ``tautological classes of matroids.'' An advantage of our approach is that we are able to exploit different techniques that were previously applicable in one model but not in others.  Two prominent such techniques are localization methods in torus-equivariant geometry and positivity properties in tropical Hodge theory.

\medskip
Let $E = \{0,1,\ldots, n\}$, and let $T$ be the algebraic torus $(\CC^*)^E$.  The standard action of $T$ on $\CC^E$ is $(t_0, \ldots, t_n)\cdot (x_0, \ldots, x_n) = (t_0x_0, \ldots, t_nx_n)$, which induces a $T$-action on the Grassmannian $\Gr(r;E)$ of $r$-dimensional subspaces of $\CC^E$.
Let $X_{E}$ be the $n$-dimensional permutohedral variety, which is the projective toric variety associated to the permutohedron $\Pi(E) = \operatorname{Conv}\big( \sigma\cdot (0,\ldots, n) \mid \sigma \text{ a permutation of $E$} \big) \subset \RR^E$.
We follow the conventions of \cite{Ful93,CLS11} for toric varieties and polyhedral geometry.

\medskip
Let $M$ be a matroid of rank $r$ with ground set $E$, and let $\operatorname{rk}_M\colon 2^E \to \ZZ$ be its rank function.
We refer to \cite{Wel76} or \cite{Oxl11} for a general background on matroids.
We always assume that a matroid has a nonempty ground set unless explicitly noted otherwise.
For a set $S$ and an element $i\in S$, as is customary in matroid theory, we often denote $S\setminus i = S\setminus \{i\}$.  For $i\in E$, we denote by $H_i$ the $i$-th coordinate hyperplane in $\CC^E$.
If $M$ is realizable over $\CC$, a realization of $M$ is an $r$-dimensional linear subspace $L\subseteq \CC^E$ such that the set of bases of $M$ equals the sub-collection $\{B \in \binom{E}{r} \mid  L \cap \bigcap_{i\in B} H_i = \{0\}\}$ of size $r$ subsets of $E$.

\subsection{Three previous geometric models of matroids}\label{subsec:overview}
For the reader's convenience, we include an overview of the previous three geometric models of matroids.  In each case, one begins with an algebro-geometric model defined for a realizable matroid, and then captures its combinatorial essence by a polyhedral model defined for an arbitrary (not necessarily realizable) matroid.

\begin{enumerate}[label=\textbf{Model (\arabic*)}, ref= (\arabic*), leftmargin=20pt, itemindent=30pt]

\item\label{model:base} (Base polytope and $K$-theory) In this model, one considers a realization $L \subseteq \CC^{E}$ as a point on $\Gr(r;E)$.
The resulting algebro-geometric model is the torus orbit closure $\overline {T\cdot L} \subset \Gr(r;E)$.
The polyhedral model is the base polytope $P(M) = \operatorname{Conv}(\sum_{i\in B} \be_i \mid B \text{ a basis of $M$})\subset\RR^{E}$, whose associated projective toric variety is isomorphic to $\overline{T\cdot L}$ when $L$ is a realization of $M$ \cite{GGMS87}.  This geometric approach to matroids led to:
\begin{itemize}
\item a matroid invariant called the $g$-polynomial \cite{Spe09}, 
\item a $K$-theoretic expression for the Tutte polynomial of $M$ \cite{FS12}, 
\item an Ehrhart-style lattice point counting expression for the Tutte polynomial of $M$ \cite{CF22}, and
\item a generalization of the Tutte polynomial of a matroid to flag matroids \cite{CDMS18, DES21}.
\end{itemize}

\smallskip
\item \label{model:Bergman} (Bergman fan and Chow ring) In this model, one considers a realization $L\subseteq \CC^{E}$ as a hyperplane arrangement $\{\mathcal H_i \subset L \mid \mathcal H_i = L \cap H_i\}_{i\in E}$ of the restrictions of $H_i\subset\CC^E$ to $L$.
The resulting algebro-geometric model is the wonderful compactification $W_L$ of the projective hyperplane arrangement complement $\PP L \setminus \bigcup \PP\mathcal H_i$, introduced in \cite{dCP95}.
It is a subvariety of the permutohedral variety $X_{E}$.
The polyhedral model is the Bergman fan $\Sigma_M$,
a subfan of the normal fan of the permutohedron, whose cones correspond to chains of flats of $M$ \cite{AK06}.  When considered as a Minkowski weight, it defines a homology class $\Delta_M$ in the Chow ring $A^\bullet(X_{E})$ that equals the class $[W_L]$ when $L$ is a realization of $M$ \cite{Stu02, KP11}.  This geometric approach to matroids led to:
\begin{itemize}
\item a notion of Chow rings of arbitrary (not necessarily realizable) matroids \cite{FY04},
\item a Chow-theoretic expression for the (reduced) characteristic polynomial of $M$ \cite{HuhKatz}, and
\item the development of the Hodge theory of matroids, and in particular a proof of the log-concavity of coefficients of the characteristic polynomial of a matroid \cite{AdiHuhKatz}, settling the long-standing conjectures of Heron, Rota, Mason, and Welsh \cite{Her72,Rot71,Ma72, Wel76}
\end{itemize}

\smallskip
\item \label{model:conormal} (Conormal fan and Chow ring) In this model, one considers a realization $L\subseteq \CC^E$ as a Lagrangian subvariety $L \times L^\perp$ of the cotangent space $\CC^E \times (\CC^E)^\vee$ of $\CC^E$, where $L^\perp=(\CC^E/L)^{\vee}$.  Projectivizing, one obtains the conormal space $\PP L \times \PP L^\perp$ of the linear subvariety $\PP L\subseteq\PP^n$.
The resulting algebro-geometric model is the critical set variety $\mathfrak X_L$ \cite{OT95,CDFV}, a variety birational to $\PP L \times \PP L^\perp$.
The polyhedral model is the conormal fan $\Sigma_{M,M^\perp}$, whose support equals the support of the product of two Bergman fans $\Sigma_M \times \Sigma_{M^\perp}$.  Here $M^\perp$ denotes the dual matroid of $M$, a matroid that is realized by $L^\perp$ when $M$ is realized by $L$.  This geometric approach to matroids led to:
\begin{itemize}
\item a notion of Chern-Schwartz-MacPherson (CSM) classes of matroids \cite{LdMRS20}, inspired by the related geometry of log-tangent sheaves \cite{DGS12, huh_2013},
\item a Chow-theoretic expression for the coefficients of $T_M(x,0)$, i.e.\ the Tutte polynomial $T_M(x,y)$ of $M$ evaluated at $y = 0$ \cite{LdMRS20}, and
\item a proof of the log-concavity of the coefficients of $T_M(x,0)$ \cite{H15, ADH}, settling  long-standing conjectures of Brylawski and Dawson \cite{Bry82,Dawson}.
\end{itemize}
\end{enumerate}

\medskip
In these models, the Tutte polynomial of a matroid and its specializations manifested geometrically in many different ways, but the connection between them has until now remained unclear. The following collects conjectures about such connections from several groups of authors. 

\begin{conjs}
\label{ques:intro1.1}
\
\begin{enumerate}[label=(\alph*)]
\item (within \ref{model:base}) The authors of \cite{CF22} conjectured there was a connection between their formulation of the Tutte polynomial and that of \cite{FS12} (see the discussion below \cite[Theorem 3.4]{CF22}).
\item (\ref{model:base} \& \ref{model:Bergman}) The authors of \cite{CDMS18} asked how the $K$-theoretic computations in \cite{FS12} relate to the Chow-theoretic computations made in \cite{HuhKatz} (see the discussion above \cite[\S1.1]{CDMS18}).  Moreover, they generalized the characteristic polynomial of a matroid to that of a flag matroid via $K$-theory, and conjectured that it also has log-concave coefficients \cite[Conjecture 9.4]{CDMS18}.
\item (\ref{model:base} \& \ref{model:conormal}) The authors of \cite{LdMRS20} conjectured a Chow-theoretic formula for the $g$-polynomial of matroids, originally defined in \cite{Spe09} via $K$-theory (see \cite[\S5.3]{LdMRS20}).
\end{enumerate}
\end{conjs}

We affirm all parts of \Cref{ques:intro1.1} (see the discussion at the end of \S\ref{sec:introKtoChow}).

\subsection{Tautological bundles and tautological classes}\label{subsec:tautdef}
We now introduce our new framework.  
Let $\CCinv^{E}$ denote the vector space $\CC^E$ with the \emph{inverse} action of $T$ where $(t_0, \ldots, t_n)\in T$ acts on $(x_0, \ldots, x_n) \in \uCCinv^E$ by $(t_0^{-1}x_0, \ldots, t_n^{-1}x_n)$.
Denote by $\uCCinv^E$ the $T$-equivariant vector bundle $X_{E}\times \CCinv^{E}$.  Finally, write $\mathbf 1$ for the identity point of the open torus $T/\CC^*\subset X_E$, where $\CC^*$ acts diagonally on $T$.

\begin{defn}\label{defn:tautbundle}
For an $r$-dimensional linear subspace $L\subseteq \CC^E$,
the \textbf{tautological subbundle} $\cS_L$ and the \textbf{tautological quotient bundle} $\cQ_L$ of $L$ are defined by
\[
\begin{split}
\cS_L \coloneqq \ &  \text{the unique  $T$-equivariant rank $r$ subbundle of $\uCCinv^{E}$ whose fiber at $\mathbf 1$ is $L$, \ \ and}\\
 \cQ_L \coloneqq \ & \text{the unique  $T$-equivariant rank $|E|-r$ quotient bundle of $\uCCinv^{E}$ whose fiber at $\mathbf 1$ is $\CC^E/L$}.
\end{split}
\]
\end{defn}

The uniqueness and existence of these bundles are verified in \Cref{prop:bundlebijection}, where we induce them from the tautological sub and quotient bundles of the Grassmannian $\Gr(r;E)$.
By construction, one has a short exact sequence $0\to \cS_L \to \uCCinv^{E} \to \cQ_L \to 0$ of vector bundles on $X_E$.

\medskip
When $L\subseteq \CC^E$ is a realization of a matroid $M$,
the $T$-equivariant $K$-classes $[\cS_L], [\cQ_L]\in K_0^T(X_{E})$ depend only on the matroid $M$ (\Cref{prop:tautdependsonmatroid}).
From this, we construct classes $[\cS_M], [\cQ_M]\in K_0^T(X_{E})$ on $X_{E}$ satisfying $[\cS_M] + [\cQ_M] = [\uCCinv^{E}]$ for arbitrary (not necessarily realizable) matroids $M$ with ground set $E$ (\Cref{defn:tautclasses}) which we call the \textbf{tautological $\boldsymbol K$-classes} of $M$.
We will write $[\cS_M^\vee], [\cQ_M^\vee]$ for their duals, and write $c_i(\mathcal{S}_M), c_i(\mathcal{Q}_M)\in A^i(X_{E})$ for their $i$-th non-equivariant Chern classes, which we call the \textbf{tautological Chern classes} of $M$.

\medskip
\subsection{Two fundamental properties}\label{sec:fundprops} The tautological $K$-classes and Chern classes of a matroid display useful features of both the $K$-theoretic approach to matroids via base polytopes (Model \ref{model:base} in \S\ref{subsec:overview}) and the Chow-theoretic approach via tropical geometry (Models \ref{model:Bergman} and \ref{model:conormal} in \S\ref{subsec:overview}).  More precisely:
\begin{enumerate}[label=(\Alph*)]
\item\label{fundprop:base}
The $T$-equivariant structure of tautological classes of matroids allows for the use of localization techniques in torus-equivariant geometry, reviewed in \S\ref{sec:equivariantreview}.  Using these localization techniques, we show in \S\ref{sec:unifyingTutte} that tautological classes of matroids satisfy a deletion-contraction property, and we show in \S\ref{sec:basepolytopeproperties} that they display the following properties often shared by matroid invariants derived from base polytopes of matroids:
\begin{itemize}
    \item A Hopf-algebraic structure reflecting the fact that a face of the base polytope $P(M)$ of a matroid $M$ is a product of the base polytopes of certain matroid minors of $M$ \cite{AA,Coalgebra,KRAJEWSKI2018271}.
    \item Valuativity, which means that the invariant satisfies an inclusion-exclusion property with respect to any subdivision of $P(M)$ into smaller matroid base polytopes \cite{ardila_fink_rincon_2010, derksen2010valuative, ardila2020valuations}.
    \item Well-behavedness under matroid duality and direct sums.
\end{itemize}

\smallskip
\item\label{fundprop:logconc} Several long-standing conjectures in matroid theory concerning the log-concavity of sequences were resolved by adapting positivity properties of nef line bundles in algebraic geometry to a tropical geometry setting \cite{AdiHuhKatz, ADH}.
In our case, if a matroid $M$ has a realization $L$, then $\cS_L^\vee$ and $\cQ_L$ are globally generated, and hence nef vector bundles.  Equivalently, the relative hyperplane classes of the bi-projective bundle $\PP(\cQ_L^{\vee}) \times_{X_{E}} \PP(\cS_L)$ are nef divisor classes, implying that the Chern classes of $\cS_L^{\vee}$ and $\cQ_L$ (as the Segre classes of $\cQ_L^{\vee}$ and $\cS_L$ respectively) have positivity and log-concavity properties.  In \S\ref{sec:logconcviataut}, we show that the Chern classes of $\cS_M^{\vee}$ and $\cQ_M$ for arbitrary (not necessarily realizable) matroids $M$ retain these same properties.  An essential tool for establishing these results is the tropical Hodge theory of Lefschetz fans developed in \cite[\S5]{ADH}.
\end{enumerate}

\subsection{A unifying formula and log-concavity for the Tutte polynomial}
An element $i\in E$ in a matroid $M$ is a loop (resp.~a coloop) if no basis of $M$ contains $i$ (resp.~every basis of $M$ contains $i$).  When $i$ is neither a loop nor a coloop in $M$, the deletion $M\setminus i$ and the contraction $M/i$ are matroids on $E\setminus i$ defined by
\begin{align*}
\text{the set of bases of $M\setminus i$} &= \{B \mid \text{$B$ a basis of $M$ such that } B\not\ni i\}, \quad\text{and}\\
\text{the set of bases of $M/i$} &= \{B \setminus i \mid \text{$B$ a basis of $M$ such that } B\ni i\}.
\end{align*}
When $i$ is a loop or a coloop in $M$, one writes $M\setminus i = M/i$ for the matroid whose set of bases equal the nonempty one among the two sets of bases above.
These notions give rise to the Tutte polynomial, the universal deletion-contraction invariant, defined for graphs in \cite{tutte}, and for matroids in \cite{crapoTutte}.

\begin{defn}
\label{defn:Tutte}
Let $M$ be a matroid with ground set $E$.
The \textbf{Tutte polynomial} $T_M(x,y)$ is the unique bivariate polynomial determined by the following two properties.
\begin{itemize}
    \item (Base case) If $|E| = 1$, then
    \[
    T_M(x,y) = \begin{cases}
    x & \text{if $M$ has rank 1 (i.e. $M$ is a coloop)}\\
    y & \text{if $M$ has rank 0 (i.e. $M$ is a loop).}
    \end{cases}
    \]
    \item (Deletion-contraction relation) If $|E| \geq 2$ and $i\in E$, then
    \[
    T_M(x,y) = \begin{cases}
    x\cdot T_{M/ i}(x,y) & \text{if $i\in E$ is a coloop in $M$}\\
    y\cdot T_{M\setminus i}(x,y) & \text{if $i\in E$ is a loop in $M$}\\
    T_{M\setminus i}(x,y) + T_{M/i}(x,y)& \text{if $i\in E$ is neither a loop nor a coloop in $M$}.\\
    \end{cases}
    \]
\end{itemize}
\end{defn}

We use the fundamental properties \ref{fundprop:base} and \ref{fundprop:logconc} of tautological classes of matroids to prove the following two theorems about the Tutte polynomial.
The first theorem is a Chow-theoretic expression for the Tutte polynomial that generalizes every previous expression for the Tutte polynomial and its specializations mentioned in 
\S\ref{subsec:overview}.  To state it, we recall two distinguished nef divisor classes on the permutohedral variety $X_{E}$ from \cite{HuhKatz,AdiHuhKatz}.  Writing $U_{1,E}$ and $U_{n,E}$ for the uniform matroids of rank 1 and corank 1 on $E$ (respectively), denote in the non-equivariant Chow group $A^1(X_E)$ the elements (see \Cref{eg:alphabeta})
\[
\alpha = c_1(\cQ_{U_{n,E}}) \quad\text{and}\quad \beta = c_1(\cS_{U_{1,E}}^\vee).
\]
Equivalently,  $X_{E}$ resolves the Cremona map $\PP^n \overset{\crem}\dashrightarrow \PP^n$ defined by $[x_0: \ldots: x_n] \mapsto [x_0^{-1}: \ldots: x_n^{-1}]$, and $\alpha,\beta$ are the pullbacks of the hyperplane classes from the domain and target respectively (see \S\ref{sec:CremonaInv}).

\begin{maintheorem}
\label{thm:4degintro}
Let $\int_{X_E}\colon A^\bullet(X_E) \to \ZZ$ be the degree map on $X_E$.  For a matroid $M$ of rank $r$ with ground set $E$, define a polynomial
\[
t_M(x,y,z,w) = (x+y)^{-1}(y+z)^{r}(x+w)^{|E|-r}T_M\Big(\frac{x+y}{y+z},\frac{x+y}{x+w}\Big).
\]
Then, we have an equality
\[
\sum_{i+j+k+\ell=n} \left(\int_{X_E}\alpha^i\beta^j  c_k(\mathcal{S}^{\vee}_M)c_\ell(\mathcal{Q}_M)\right)x^iy^jz^kw^\ell = t_M(x,y,z,w).
\]
\end{maintheorem}

We prove \Cref{thm:4degintro} in \S\ref{sec:unifyingTutte} by showing that the tautological Chern classes of matroids satisfy a deletion-contraction relation (\Cref{thm:delcont}).  We also present in \Cref{sec:convolution} a different proof obtained by establishing a  recursive convolution formula for both tautological Chern classes and Tutte polynomials of matroids, which may be of independent interest.

\medskip
The second theorem is a log-concavity property for the Tutte polynomial expression in  \Cref{thm:4degintro}, which generalizes every log-concavity result mentioned in \S\ref{subsec:overview}.  To state it, we recall that a nonnegative sequence $(a_0, a_1, \ldots, a_m)$ is \textbf{log-concave} if $a_k^2 \geq a_{k-1}a_{k+1}$ for all $1\leq k \leq m-1$, and has \textbf{no internal zeros} if $a_ia_j >0 \implies a_k>0$ for all $0\leq i \le k \le j \leq m$.  For a homogeneous polynomial $f\in \RR[x_1, \ldots, x_N]$ of degree $d$ with nonnegative coefficients, we say that its coefficients form a \textbf{log-concave unbroken array} if, for any $1\leq i < j \leq N$ and a monomial $x^{\bm}$ of degree $d'\leq d$, the coefficients of $\{x_i^kx_j^{d-d'-k}x^{\bm}\}_{0\leq k \leq d-d'}$ in $f$ form a log-concave sequence with no internal zeros.

\begin{maintheorem}\label{thm:masterlogconcintro}
For a matroid $M$ of rank $r$ with ground set $E$, the coefficients of the polynomial
\[
t_M(x,y,z,w) = (x+y)^{-1}(y+z)^{r}(x+w)^{|E|-r}T_M\Big(\frac{x+y}{y+z},\frac{x+y}{x+w}\Big)
\]
form a log-concave unbroken array.
\end{maintheorem}

We prove \Cref{thm:masterlogconcintro} in \S\ref{sec:logconcviataut}.  In fact, we establish a stronger statement (\Cref{thm:logconcMmany}) implying that the polynomial $t_M(x,y,z,w)$ is a ``denormalized Lorentzian polynomial'' in the sense of \cite{BH20}.

\medskip
By considering coefficients of $x^{r-1-i}y^iw^{|E|-r}$ in \Cref{thm:masterlogconcintro} (i.e.\ setting $z = 0$ and considering terms divisible by $w^{|E|-r}$), one recovers the log-concavity for the coefficients of the unsigned reduced characteristic polynomial $T_M(q+1,0)/(q+1)$ from \cite{AdiHuhKatz}.
By considering coefficients of $x^{r-1-i}z^iw^{|E|-r}$ (i.e.\ setting $y = 0$ and considering terms divisible by $w^{|E|-r}$), one recovers the log-concavity for the coefficients of $T_M(q,0)$, the $h$-vector of the broken circuit complex of $M$, from \cite{ADH}.
To deduce the log-concavity of the coefficients of  $T_M(q,1)$, the $h$-vector of the independence complex, \cite{ADH} appeals to a combinatorial property of the free coextension matroid of $M$ due to Brylawski \cite{Bry77}.  Here, one can deduce this result directly from \Cref{thm:masterlogconcintro} applied to $M$ by considering the coefficients of $x^{|E|-r-1+i}z^{r-i}$ (i.e.\ setting $y=w=0$).\footnote{A strengthening of the log-concavity of the $f$-vector of the independence complex to ultra-log-concavity, conjectured by Mason \cite{Ma72}, was established in \cite{ALOGV18} and \cite{BH20}.  A strengthening of the log-concavity of the $h$-vector is given in \cite{BST20}.  Neither strengthening is implied by \Cref{thm:masterlogconcintro}.} We note \cite{Dawson} showed that $h$-vector log-concavity always implies $f$-vector log-concavity.

\subsection{Minkowski weights associated to matroids}
We use \Cref{thm:4degintro} to relate the tautological Chern classes of a matroid to prior constructions such as the Bergman fan and the Chern-Schwartz-MacPherson classes of a matroid (\Cref{defn:Bergman} and \Cref{defn:CSM}, respectively).
See \S\ref{sec:minkowskiweights} for the definition of, and notations concerning, Minkowski weights used in the statements below.

\begin{maintheorem}\label{thm:previousMWs}

Let $M$ be a matroid of rank $r$ with ground set $E$.  Let $\Delta_M$ be its Bergman class.  For $0\leq k \leq r-1$, let $\operatorname{csm}_k(M)$ be its $k$-dimensional Chern-Schwartz-MacPherson class.  Then
\[
\Delta_M = c_{|E|-r}(\cQ_M)\cap \Delta_{\Sigma_E} \quad\text{and}\quad \operatorname{csm}_k(M) = c_{r-1-k}(\cS_M)c_{|E|-r}(\cQ_M) \cap \Delta_{\Sigma_E}.
\]
\end{maintheorem}

The result for $\Delta_M$ is proved in \Cref{thm:BergmanComb} and the result for $\operatorname{csm}_k(M)$ is proved in \Cref{thm:CSMComb}.
For realizable matroids, we establish stronger geometric statements in \Cref{thm:BergmanGeom} and \Cref{thm:CSMGeom}.
Using these theorems, in \S\ref{sec:Bergmanvia} and \S\ref{sec:CSMvia} we recover the properties of the Bergman fans and the CSM classes of matroids previously established in \cite{FS05,AK06} and \cite{LdMRS20}, respectively.
In light of \Cref{thm:previousMWs}, \Cref{thm:4degintro} generalizes \cite[Lemma 6.1]{HuhKatz}, which states that the intersection degrees $\alpha^{i}\beta^{r-1-i} \Delta_M$  equal the coefficients of the unsigned reduced characteristic polynomial $T_M(q+1,0)/(q+1)$, and generalizes \cite[Theorem 5.8]{LdMRS20}, which states that the intersection degrees $\alpha^{i}  \operatorname{csm}_{i}(M)$ equal the coefficients of $T_M(q,0)$.
These combinatorial interpretations of the tautological Chern classes, which are particular cases of Schur classes of $\mathcal{S}_M$ and $\mathcal{Q}_M$, motivate us to pose the following question.
 \begin{ques}
What  combinatorial interpretations do products of Schur classes of $\mathcal{S}_M^{\vee}$ and $\mathcal{Q}_M$ admit? Do they similarly satisfy positivity and log-concavity properties?
 \end{ques}
For instance, \cite{FulLaz,DPS94} establish positivity properties for Schur classes of nef vector bundles, which apply to the globally generated bundles $\cS_L^\vee$ and $\cQ_L$ if $L$ is a realization of $M$.

\subsection{A K-theory to Chow theory bridge}
\label{sec:introKtoChow}
We now turn to connecting \Cref{thm:4degintro}, a Chow-theoretic expression, to expressions for the Tutte polynomial obtained via $K$-theoretic tools.
One could try using the Hirzebruch-Riemann-Roch (HRR) theorem, which states that the Euler characteristic $\chi([\mathcal{E}])$ of a $K$-class $[\mathcal E]$ on a smooth projective variety $X$ satisfies
 \[
 \chi \big( [\mathcal{E}] \big)=\int_X \Td(X)\cdot \ch([\mathcal{E}]),
 \]
where $\Td(X)\in A^\bullet(X)_\QQ$ is the Todd class of $X$ and $\ch([\mathcal E])$ is the Chern character of $[\mathcal E]$ (by convention $\int_X\gamma=0$ if $\gamma\in A^i(X)$ for $i<\dim(X)$).
However, the Hirzebruch-Riemann-Roch theorem does not appear to be useful in our context (see \Cref{rem:ToddHard}).

\medskip
We construct an exceptional isomorphism $\zeta_{X_E}\colon K_0(X_{E})\overset\sim\to A^\bullet(X_{E})$, unrelated to the Chern character $\ch$, that translates between $K$-theoretic and Chow-theoretic computations using $1+\alpha+\ldots+\alpha^n$ in place of the Todd class $\Td(X_{E})$. This map behaves particularly well on a collection of $T$-equivariant $K$-classes that we say ``have simple Chern roots'' (\Cref{defn:simpleChernroots}), which includes $[\mathcal{S}_M^{\vee}]$ and $[\mathcal{Q}_M^{\vee}]$ for any matroid $M$.

\begin{maintheorem}
\label{thm:fakeHRRintro}
There exists a ring isomorphism
$
\zeta_{X_E} \colon  K_0(X_{E})\overset\sim\to A^\bullet(X_{E})
$ which satisfies
\[
\chi \big( [\mathcal{E}] \big)=\int_{X_E} (1+\alpha+\cdots + \alpha^n)\cdot \zeta_E([\mathcal{E}])
\]
for any $[\mathcal{E}]\in K_0(X_{E})$.
Denote by $\bigwedge^i$ for the $i$-th exterior power and $c(\mathcal E,u) \coloneqq \sum_{i\geq 0} c_i(\mathcal{E})u^i$ the Chern polynomial of $[\mathcal E]$. If $[\mathcal{E}]$ has simple Chern roots and rank $\operatorname{rk}(\mathcal E)$ then we have
\[
\begin{split}
&\sum_{i\geq 0} \zeta_{X_E}\big( [ \textstyle \bigwedge^i\mathcal{E}] \big) u^i=(u+1)^{\rk(\mathcal{E})}c(\mathcal{E},\frac{u}{u+1}), \quad\text{and}\\
&\displaystyle \sum_{i\geq 0}\zeta_{X_E} \big([ \textstyle \bigwedge^i \mathcal{E}^{\vee}] \big) u^i=(u+1)^{\rk(\mathcal{E})}c(\mathcal{E},1)^{-1}c(\mathcal{E},\frac{1}{u+1}).
\end{split}
\]
\end{maintheorem}

We prove the first part of \Cref{thm:fakeHRRintro} in \Cref{thm:fakeHRR}, and the second part in \Cref{prop:simpleChern}.
Applying \Cref{thm:fakeHRRintro} to \Cref{thm:4degintro}, we recover both the $K$-theoretic formula for the Tutte polynomial \cite[Theorem 5.1]{FS12} (see \Cref{thm:FSTutte}) and the lattice-point-counting formula for the Tutte polynomial \cite[Theorem 3.2]{CF22} (see \Cref{thm:CFTutte}), thereby answering \Cref{ques:intro1.1}.(a) and the first part of \Cref{ques:intro1.1}.(b).  We also use \Cref{thm:fakeHRRintro} to give a Chow-theoretic formula for Speyer's $g$-polynomial of a matroid (\Cref{thm:gPoly}) conjectured in \cite[Conjecture 1]{LdMRS20}, answering \Cref{ques:intro1.1}.(c).

\medskip
Finally, in \S\ref{sec:flagmatroids} we show that our methods generalize well to flag matroids, answering two conjectures concerning the characteristic polynomials of flag matroids that were defined and studied in \cite{CDMS18, DES21}.  In particular, we establish a log-concavity property answering the second part of \Cref{ques:intro1.1}.(b).

\medskip
Matroids and flag matroids are the ``type $A$'' examples of Coxeter matroids \cite{GS87b, BGW03}.  They are also examples of polymatroids (\Cref{rem:polymatroid}).  Motivated by these, we pose the following question.

\begin{ques}
How do the results here generalize to Coxeter matroids or polymatroids?
\end{ques}

For instance, we show in \S\ref{subsec:EhrVol} that \Cref{thm:fakeHRRintro} recovers Postnikov's result \cite[Theorem 11.3]{P09} relating Ehrhart and volume polynomials of generalized permutohedra, of which polymatroids are a subfamily.

\subsection*{Acknowledgements}
We would like to thank Alex Fink for helpful discussions on  the convolution formula for Tutte polynomials, and we would like to thank Eric Katz for helpful discussions and for sharing unpublished notes of a deletion-contraction proof of \cite[Proposition 5.2]{HuhKatz}. We would also like to thank the creators of Macaulay2 \cite{M2} for their helpful and free software, and Justin Chen for the Macaulay2 package on matroids \cite{Che19}, which was used extensively in the early stages of this project. We thank Graham Denham, Ahmed Ashref, and Avi Steiner for suggesting minor edits to an earlier draft of the paper. We thank the referee for a careful reading and helpful suggestions.  The second and fourth authors were partially supported by the US National Science Foundation (DMS-2001854 and DMS-2001712).


\section{Equivariant geometry of permutohedral varieties}\label{sec:equivariantreview}

We set notations and collect results relevant to the torus-equivariant $K$-theory and Chow theory of permutohedral varieties.
We also note some features that are special to permutohedral varieties not shared by arbitrary toric varieties.

\subsection{Equivariant K-ring and equivariant Chow ring}
Recall that $E = \{0,1, \ldots, n\}$, and $T = (\CC^*)^E$.  Let $\operatorname{Char}(T)$ be the character group of $T$.
For a smooth $T$-variety $X$, let $K^0_T(X)$ denote the $T$-equivariant Grothendieck $K$-ring of vector bundles on $X$, as defined in for example \cite{KnuRosu03,vezzosiVistoli}.
For $[\mathcal E]\in K_T^0(X)$, we write $[\mathcal E^\vee]$ for its dual class. 
Writing $T_0, \ldots, T_n$ for the standard characters of $T$, we identify the character ring $\ZZ[\operatorname{Char}(T)]$ with the Laurent polynomial ring $\ZZ[T_0^\pm, \ldots, T_n^\pm]$.
In particular, the $T$-equivariant $K$-ring $K^0_T(\pt)$ of a point is identified with $\ZZ[T_0^\pm, \ldots, T_n^\pm]$, where a $T$-representation corresponds to the sum of its characters.

We let $A_T^\bullet(X)$ denote the $T$-equivariant Chow ring of $X$ as defined in \cite{EdiGra98}, noting the identification of the equivariant Chow homology with the equivariant Chow cohomology occurring in \cite[Proposition~4]{EdiGra98} when $X$ is smooth. For a $T$-equivariant $K$-class $[\mathcal E]\in K_T^0(X)$, we write $c_i(\mathcal E)$ for its $i$-th \emph{non-equivariant} Chern class, reserving the notation $c_i^T(\mathcal{E})$ for the $i$-th  \emph{$T$-equivariant} Chern class.
We identify the symmetric algebra $\operatorname{Sym}^\bullet\operatorname{Char}(T)$, which is the $T$-equivariant Chow ring $A^\bullet_T(\pt)$ of a point, with the polynomial ring $\ZZ[t_0, \ldots, t_n]$.
Here we have used the lowercase $t$ for notational clarity: In the context of equivariant $K$-theory or Chow rings, the $T_i$ variables will denote elements in the Laurent polynomial ring, whereas the $t_i$ variables will denote elements in the polynomial ring.

\subsection{Grassmannians}\label{eg:Gr}
Let $\cS$ and $\cQ$ denote the tautological sub and quotient bundle on $\Gr(r;E)$, respectively.  For an $r$-element subset $I$ of $E$, let $p_I$ be the $T$-fixed point of $\Gr(r;E)$ corresponding to $\operatorname{span}(\be_i \mid i\in I)\subseteq \CC^E$.  
For a $T$-equivariant $K$-class $[\mathcal E]\in K_T^0(\Gr(r;E))$, write $[\mathcal E]_I$ for its image under the restriction map $K_T^0(\Gr(r;E)) \to K_T^0(p_I) = \ZZ[T_0^\pm, \ldots, T_n^\pm]$.
Then, we have
\[
[\cS]_I = \sum_{i\in I} T_i \quad\text{and}\quad [\cQ]_I = \sum_{j\in E\setminus I} T_j \quad\text{for any $I\in  \textstyle \binom{E}{r}$}.
\]
In particular, the ample line bundle $\mathcal O(1)$ on $\Gr(r;E)$, whose global sections give the Pl\"ucker embedding into $\PP(\CC^{\binom{E}{r}})$, satisfies $[\mathcal O(1)]_I = [\det \cS^\vee]_I = \prod_{i\in I} T_i^{-1}$ for every $I\in \binom{E}{r}$.
If $I$ and $J$ are $r$-element subsets of $E$ such that $J = I\setminus\{i\}\cup\{j\}$ for some $i\in I$ and $j\in J$,
then every $T$-equivariant $K$-class $[\mathcal E] \in K_T^0(\Gr(r;E))$ satisfies $[\mathcal E]_I \equiv [\mathcal E]_J \mod (1 - T_i/T_j)$. Conversely a collection $([\mathcal{E}]_I)_{I\in \binom{E}{r}}$ satisfying this congruence for all such $I,J$ determines a unique $[\mathcal{E}]$ (see for example the discussion in \cite[\S2.2]{FS12}).

\subsection{Conventions for permutations and cones}
\label{subsec:permutohedralcone}
We now specialize to permutohedral varieties.  We first set some notations and conventions.
Let $\langle \cdot, \cdot \rangle$ denote the standard inner product on $\RR^E$, and let $\mathfrak S_E$ be the set of permutations on $E$.
The normal fan $\widetilde \Sigma_{E}$ of the permutohedron $\Pi(E) = \operatorname{Conv}( \sigma\cdot(0,1, \ldots ,n) \mid \sigma\in \mathfrak S_E)$ is the fan in $\RR^E$ induced by the type $A_n$ hyperplane arrangement, which consists of the $\binom{n+1}{2}$ hyperplanes $\{x\in \RR^E \mid \langle x, \be_i - \be_j\rangle = 0\}_{0\le i < j \le n}$.
Every cone of $\widetilde\Sigma_{E}$ has 1-dimensional lineality space $\RR\mathbf 1$, where $\mathbf 1 = (1,1,\ldots, 1)$.  
Let $\Sigma_E$ be the quotient fan in $\RR^E/\RR\mathbf 1$.
It is a rational unimodular fan over the lattice $\ZZ^E/\ZZ\mathbf 1$, whose dual lattice is $\mathbf 1^\perp = \{x\in \ZZ^E \mid \langle x, \mathbf 1\rangle = 0\}$.

The cones of $\Sigma_E$ correspond to flags of nonempty proper subsets of $E$.  Such a flag
$\mathscr S\colon  \emptyset \subsetneq S_1 \subsetneq \cdots \subsetneq S_k \subsetneq E$ corresponds to $\operatorname{Cone}(\overline\be_{S_1},\ldots, \overline\be_{S_k})$,
where $\overline\be_S$ denotes the image of $\be_S\coloneqq\sum_{i\in S} \be_i$ under $\RR^E \to \RR^E/\RR\mathbf 1$.
The permutohedral variety $X_E$ is the toric variety of the fan $\Sigma_E$, whose dense open torus is $T/\CC^*$ with $\CC^*$ acting diagonally on $T$.  We often consider $X_E$ as a $T$-variety, where the diagonal $\CC^*$ acts trivially on $X_E$.  As the fan $\Sigma_E$ is unimodular, the variety $X_E$ is smooth.

\medskip
We set a bijection between the set $\mathfrak S_E$ of permutations on $E$ and the set of maximal cones of $\Sigma_E$ by
\[
\mathfrak S_E \ni \sigma \ \longleftrightarrow \ \begin{matrix}
\text{$\operatorname{Cone}(\overline\be_{S_1},\ldots, \overline\be_{S_{n}})$ corresponding to the chain $\mathscr S$} \\
\text{where }S_i = \{\sigma(0), \ldots, \sigma(i-1)\} \text{ for } 1\leq i\leq n
\end{matrix}.
\]
The elements in the interior of the cone corresponding to $\sigma$ are precisely those of the form $v_0\overline\be_{\sigma(0)} + v_1\overline\be_{\sigma(1)} + \cdots + v_n\overline\be_{\sigma(n)}$ for $v_0>\ldots >v_n$. This bijection naturally induces a bijection between $\mathfrak S_E$ and the set $X_E^T$ of $T$-fixed points of $X_E$.  For a permutation $\sigma$, we write $p_\sigma$ for the $T$-fixed point corresponding to $\sigma$, and write $U_\sigma \simeq \CC^n$ for the $T$-invariant affine chart around $p_\sigma$.  Since $\operatorname{Cone}(\be_{\sigma(0)} - \be_{\sigma(1)}, \ldots, \be_{\sigma(n-1)} -\be_{\sigma(n)})\subset \mathbf{1}^{\perp}$ is the dual cone of the maximal cone corresponding to $\sigma$,
the torus $T$ acts on $U_\sigma$, and hence in particular the tangent space to $p_{\sigma}$, with characters $T^{\ }_{\sigma(0)}T_{\sigma(1)}^{-1}, \ldots, T^{\ }_{\sigma(n-1)}T_{\sigma(n)}^{-1}$. 

\subsection{Localization theorems}
By the bijection between permutations $\mathfrak S_E$ and $T$-fixed points $X_E^T$, we identify $K_T^0(X_E^T)$ with $\prod_{\sigma\in \mathfrak S_E} \ZZ[T_0^\pm, \ldots, T_n^\pm]$, and identify $A^\bullet_T(X_E^T)$ with $\prod_{\sigma\in \mathfrak S_E} \ZZ[t_0, \ldots, t_n]$.  For $f$ in $\prod_{\sigma\in \mathfrak S_E} \ZZ[T_0^\pm, \ldots, T_n^\pm]$ (resp.\ $\prod_{\sigma\in \mathfrak S_E} \ZZ[t_0, \ldots, t_n]$), we write $f_\sigma$ for its projection to the copy of the factor $\ZZ[T_0^\pm, \ldots, T_n^\pm]$ (resp.\ $\ZZ[t_0, \ldots, t_n]$) indexed by the permutation $\sigma$.

\begin{thm}\label{thm:localization}
Let $X_E$ be the permutohedral variety as above.
\begin{enumerate}[label = (\alph*)]

\item \label{localization:K}
The restriction map $K_T^0(X_E) \to K_T^0(X_E^T)$ from the $T$-equivariant $K$-ring of $X_E$ to that of its $T$-fixed points is injective, and its image is the subring of $K_T^0(X_E^T)$ given by
\[
\left\{ f\in \prod_{\sigma\in \mathfrak S_E} \ZZ[T_0^\pm, \ldots, T_n^\pm] \ \middle| \ \begin{matrix} f_\sigma \equiv f_{\sigma'}\mod (1- \frac{T_{\sigma(i+1)}}{T_{\sigma(i)}}) \\ \textnormal{whenever $\sigma' = \sigma  \circ (i, i+1)$ for a transposition $(i,i+1)$}\end{matrix}\right\}.
\]
In particular, the ring $K_T^0(X_E)$ can be identified with the ring $PLaur(\widetilde \Sigma_{E})$ of piecewise Laurent polynomials  on the fan $\widetilde \Sigma_E$ in $\RR^E$, and the non-equivariant $K$-ring $K^0(X_E)$ is isomorphic to the quotient of $PLaur(\widetilde\Sigma_E)$ by the ideal generated by $f(T_0, \ldots, T_n) - f(1,\ldots,1)$ for each global Laurent polynomial $f$ on $\RR^E$.

\item \label{localization:Chow} The restriction map $A^\bullet_T(X_E) \to A^\bullet_T(X_E^T)$ from the $T$-equivariant Chow ring of $X_E$ to that of its $T$-fixed points is injective, and its image is the subring of $A^\bullet_T(X_E^T)$ given by
\[
\left\{ \varphi \in \prod_{\sigma\in \mathfrak S_E} \ZZ[t_0, \ldots, t_n] \ \middle| \ \begin{matrix} \varphi_\sigma \equiv \varphi_{\sigma'}\mod (t_{\sigma(i)} - t_{\sigma(i+1)}) \\ \textnormal{whenever $\sigma' = \sigma \circ (i, i+1) $ for a transposition $(i,i+1)$}\end{matrix}\right\}.
\]
In particular, the ring $A^\bullet_T(X_E)$ can be identified with the ring $PPoly(\widetilde \Sigma_{E})$ of piecewise polynomial functions on the fan $\widetilde \Sigma_E$ in $\RR^E$, and the non-equivariant Chow ring $A^\bullet(X_E)$ is isomorphic to the quotient of $PPoly(\widetilde\Sigma_E)$ by the ideal generated by $\varphi(t_0, \ldots, t_n) - \varphi(0,\ldots,0)$ for each global polynomial $\varphi$ on $\RR^E$.
\end{enumerate}
\end{thm}

In light of the above theorem, for $[\mathcal E]\in K_T^0(X_E)$ we also write $[\mathcal E]$ for its image in $K_T^0(X_E^T)$, and $[\mathcal E]_\sigma \in \ZZ[T_0^\pm, \ldots, T_n^\pm]$ for the restriction of $[\mathcal E]$ to the $T$-fixed point $p_\sigma$.  We notate similarly for $\xi\in A^\bullet_T(X_E)$.

\medskip
With one exception, \Cref{thm:localization} collects standard results in equivariant geometry that hold, e.g., for smooth proper toric varieties.
\Cref{thm:localization}.\ref{localization:K} follows from \cite[Corollary~5.12, Theorem~5.19]{vezzosiVistoli} while \ref{localization:Chow} follows from \cite[Corollary~2.3, Theorem~3.4]{Bri97}.
The non-standard exception above is the the identification of $K_T^0(X_E)$ with a ring of piecewise Laurent polynomials on a fan.  This result is special to the permutohedral variety, and fails for arbitrary toric varieties.  
The validity of the identification follows from two straightforward observations: No codimension 1 cone of $\widetilde \Sigma_E$, whose linear span is the hyperplane normal to $\be_{\sigma(i)} -\be_{\sigma(i+1)}$ for some $\sigma$ and $i$, is contained in a coordinate hyperplane of $\RR^E$, and two Laurent polynomials $f_\sigma$ and $ f_{\sigma'}$ satisfy $f_\sigma \equiv f_{\sigma'}\mod (1- \frac{T_{\sigma(i+1)}}{T_{\sigma(i)}})$ if and only if $f_\sigma \equiv f_{\sigma'}\mod (T_{\sigma(i)}- T_{\sigma(i+1)})$.

\subsection{Duality, rank, exterior powers, and Chern classes}\label{subsec:prelimChernRoots}
For $\bm = (m_0, \ldots, m_n) \in \ZZ^E$, write $\bT^{\mathbf m} = T_0^{m_0}\cdots T_n^{m_n}$, and write $\bm \cdot \bt = m_0t_0 + \cdots + m_n t_n$.  
Let $[\mathcal E]\in K_T^0(X_E)$.  Then, for each $\sigma\in \mathfrak S_E$ we have
\[
[\mathcal E]_\sigma  = \sum_{i=1}^{k_\sigma} a_{\sigma,i} {\bT}^{\bm_{\sigma,i}}
\]
for some integer $k_\sigma \geq 0$, signs $a_1, \ldots, a_{k_\sigma}\in \{-1,1\}$, and $\bm_{\sigma,1}, \ldots, \bm_{\sigma,k_\sigma}\in \ZZ^E$. The dual class $[\mathcal{E}^{\vee}]$ is defined by saying $$[\mathcal{E}^{\vee}]_{\sigma}= \sum_{i=1}^{k_\sigma} a_{\sigma,i} {\bT}^{-\bm_{\sigma,i}}.$$ The map which takes a vector bundle $\mathcal{E}$  to its rank is additive, and hence extends to a map $\rk\colon  K_0^T(X_E)\to \mathbb{Z}$. We have $\rk(\mathcal{E})=a_{\sigma,1} + \cdots + a_{\sigma, k_\sigma}$ for any $\sigma$, since when $\mathcal{E}$ is a vector bundle the right hand side is the rank of the pullback of the bundle to the torus fixed point $p_{\sigma}$. The $j$-th exterior power, denoted $[ \textstyle\bigwedge^j \mathcal E]$, and the $j$-th $T$-equivariant Chern class, are given by equating 
$$\sum_{j=0}^{\infty}[ {\textstyle\bigwedge}^j \mathcal{E}]_{\sigma}u^j=\prod_{i=1}^{k_{\sigma}} (1+\bT^{\bm_{\sigma,i}}u)^{a_{\sigma,i}} \quad \text{and}\quad 
c^T(\mathcal E, u) = \displaystyle \sum_{j=0}^{\infty}c_j^T(\mathcal{E})_{\sigma}u^j=\prod_{i=1}^{k_{\sigma}}(1+\bm_{\sigma,i}\cdot\bt u)^{a_{\sigma,i}},$$
where $u$ is a formal variable. We note that this implies $c^T(\mathcal{E}^{\vee},u)=c^T(\mathcal{E},-u)$. When $a_{\sigma,i} = 1$ for all $\sigma$ and $i$, in which case $k_\sigma = \operatorname{rk}(\mathcal E)$ for all permutations $\sigma$, we have for $0\le j \le \rk(\mathcal{E})$ that
\[
[\textstyle\bigwedge^j \mathcal E]_\sigma = \operatorname{Elem}_j(\bT^{\bm_{\sigma,1}}, \ldots,  \bT^{\bm_{\sigma,\operatorname{rk}(\mathcal E)}}) \quad\text{and}\quad c_j^T(\mathcal E) = \operatorname{Elem}_j(\bm_{\sigma,1}\cdot  \bt, \ldots, \bm_{\sigma,\operatorname{rk}(\mathcal E)}\cdot  \bt),
\]
where $\operatorname{Elem}_j$ denotes the $j$-th elementary symmetric polynomial.
More generally, when $a_{\sigma, i} = 1$ for all $\sigma$ and $i$, given an element $\lambda(x)\in \Lambda\subset \mathbb{Z}[[x_1,x_2,\ldots]]$ in the ring $\Lambda$ of symmetric functions \cite[Section I.2]{Mac15}, we define the $T$-equivariant $K$-class $[\mathsf S^\lambda \mathcal E]$ by
\[
[\mathsf S^\lambda \mathcal E]_{\sigma} = \lambda(\bT^{\bm_{\sigma,1}}, \dots,  \bT^{\bm_{\sigma,\operatorname{rk}(\mathcal E)}}, 0, 0, \dots) \quad\text{for all $\sigma\in \mathfrak S_E$},
\]
and define the $T$-equivariant Chow class $\mathsf s_\lambda^T(\mathcal E)$ by
\[
\mathsf s_\lambda^T(\mathcal E)_\sigma = \lambda(\bm_{\sigma,1}\cdot  \bt, \dots, \bm_{\sigma,\operatorname{rk}(\mathcal E)}\cdot  \bt, 0,0,\dots) \quad\text{for all $\sigma\in \mathfrak S_E$}.
\]

\subsection{Cremona involution}
\label{sec:CremonaInv}
The projective space $\PP^n$ is a $T$-variety by the standard action of $T$ on $\CC^E$.  The fan of $\PP^n$ as a toric variety is then the coarsening of $\Sigma_E$ with rays $\operatorname{Cone}(\overline \be_i)$ for each $i\in E$, and we have the naturally induced birational map $\pi_E\colon  X_E \to \PP^n$ that is the identity on the common dense open torus $T/\CC^*$.
Let $\crem\colon  \PP^n \dashrightarrow \PP^n$ be the Cremona transformation defined by $[t_0: \cdots t_n] \mapsto [t_0^{-1} : \cdots : t_n^{-1}]$ on the open dense torus $T/\CC^*$ of $\PP^n$.
Relatedly, the map $\RR^E/\RR\mathbf 1 \to \RR^E/\RR\mathbf 1$ given by $x\mapsto -x$ defines an involution of the fan $\Sigma_E$, and hence induces the Cremona involution $\crem\colon  X_E \to X_E$ of the permutohedral variety $X_E$, fitting into the diagram
\[
\begin{tikzcd}
X_E \ar[r, "\crem"] \ar[d, swap, "\pi_E"]&X_E \ar[d, "\pi_E"]\\
\PP^n \ar[r, dashed, "\crem"] &\PP^n.
\end{tikzcd}
\]
Because $\operatorname{crem}$ is an involution, we have $\operatorname{crem}_*=\operatorname{crem}^*$ on $K_0^T(X_E)$ and $A^\bullet_T(X_E)$, so we simply write $\operatorname{crem}$ when acting on $K$-classes or Chow classes.

\begin{rem}\label{rem:crem}
The map $\crem\colon  X_E \to X_E$ is not $T$-equivariant, but is a toric morphism, where the map of tori $T\to T$ is given by $t\mapsto t^{-1}$.  For a permutation $\sigma\in \mathfrak  S_E$, note that the $T$-fixed point $p_\sigma$ maps under $\crem$ to the point $p_{\overline\sigma}$, where
\[
\overline\sigma \in \mathfrak S_E \quad\text{is defined by}\quad \overline\sigma(i) = \sigma(n-i) \quad\text{for $i=0,\ldots, n$.}
\]
As a result, if $[\mathcal E]\in K_T^0(X_E)$ with $[\mathcal E]_\sigma = \sum_{i=1}^{k_\sigma} a_{\sigma,i} {\bT}^{\bm_{\sigma,i}}$ for $\sigma\in \mathfrak S_E$, then $\crem [\mathcal E] \in K_T^0(X_E)$ satisfies
\[
(\crem [\mathcal E])_{\sigma} = [\mathcal E]_{\overline{\sigma}}(T_0^{-1}, \ldots, T_n^{-1}) = \sum_{i=1}^{k_{\overline\sigma}} a_{{\overline\sigma},i} {\bT}^{-\bm_{{\overline\sigma},i}}.
\]
\end{rem}

The Cremona involution gives rise to two distinguished divisor classes on $X_E$ as follows.

\begin{defn}
\label{defn:alphabeta}
Writing $h = c_1(\mathcal O(1)) \in A^1(\PP^n)$ for the hyperplane class in $\PP^n$, define the divisor classes $\alpha_E = \pi_E^*h$ and $\beta_E = \crem \alpha_E = (\pi_E \circ \crem)^* h$ in $A^1(X_E)$.
\end{defn}

We omit the subscript $E$ from $\alpha_E$ and $\beta_E$ when the ground set $E$ is clear.  The line bundles of the divisor classes $\alpha$ and $\beta$ are given $T$-linearizations in \Cref{rem:alphabeta}.

\subsection{Generalized permutohedra and $T$-linearized line bundles}\label{subsec:Tdiv}
Let $P \subset\RR^E$ be a lattice polytope contained in a translate of the sublattice $\mathbf 1^\perp = \{\bm \in \ZZ^E \mid  m_0  + \cdots + m_n = 0\}$.
Let $h_P\colon  \RR^E \to \RR$ be its support function defined by $h_P(x) = \max_{\bm\in P}\langle \bm, x \rangle$.
It was shown in several places \cite[\S II]{Edm70}, \cite[Ch.~4]{Mur03}, \cite[Theorem 12.3 \& references therein]{AA} that the following two statements are equivalent:
\begin{itemize}
\item The polytope $P$ is a \textbf{generalized permutohedron}, i.e.\ its normal fan coarsens the fan $\widetilde\Sigma_E$.
\item The function $\rk_P\colon  2^E \to \RR$ defined by $\rk_P(S) = h_P(\be_S)$ for a subset $S\subseteq E$ is \textbf{submodular}, i.e.\
\[
\rk_P(S) + \rk_P(S') \geq \rk_P(S\cup S') + \rk_P(S\cap S') \quad\text{for all subsets $S,S'\subseteq E$},
\]
and $P = \{ x \in \RR^E \mid \langle x, \mathbf 1 \rangle = \rk_P(E) \text{ and } \langle x, \be_S \rangle \leq \rk_P(S) \text{ for all nonempty } S\subseteq E\}$.
\end{itemize}
Let $P$ now be a generalized permutohedron. The negated polytope $-P$ has support function $h_{-P}(v)=h_P(-v)$, and is also a generalized permutohedron as  $\rk_{-P}(S)=\rk_P(E\setminus S) - \rk_P(E)$ is submodular. We associate to $P$ a $T$-linearized line bundle on $X_E$ as follows.
Since $\mathbf 1^\perp$ is the dual lattice of $\ZZ^E/\ZZ\mathbf 1$,
for every translate $P'$ of $P$ such that $P'\subset\mathbf 1^\perp$, there is the associated $(T/\mathbb{C}^*)$-invariant divisor $D_{P'}$ on $X_E$ (see \cite[Theorem 4.2.12 \& Proposition 4.2.14]{CLS11}) given by 
\[
D_{P'} = \sum_{\emptyset\subsetneq S \subsetneq E} -\min_{\mathbf m \in P'} \langle \mathbf m, \overline\be_S\rangle Z_S=\sum_{\emptyset \subsetneq S \subsetneq E}\max_{\bm \in -P'} \langle \bm, \overline\be_S\rangle Z_S = \sum_{\emptyset \subsetneq S \subsetneq E} \rk_{-P'}(S) Z_S
\]
where $Z_S$ denotes the $(T/\mathbb{C}^*)$-invariant divisor in $X_E$ corresponding to the ray $\operatorname{Cone}(\be_S)$ in the fan $\Sigma_E$. We note here that the divisor class $[D_{P'}]\in A^1(X_E)$ in the non-equivariant Chow ring of $X_E$ is independent of the translation $P'$, so we may write $[D_P]$ for this divisor class. The resulting
line bundle $\mathcal O(D_P)$ admits a
$T$-linearization given by
\[
[\mathcal O(D_P)]_\sigma = \bT^{-\bm_\sigma} \in \ZZ[T_0^\pm, \ldots, T_n^\pm] \quad\text{for any $\sigma \in \mathfrak S_E$},
\]
where $\bm_\sigma$ is the vertex of $P$ minimizing the pairing with any vector in the interior of the cone corresponding to the permutation $\sigma$.  Concretely, the lattice point $\bm_\sigma$ is the vertex of $P$ achieving for any sequence $v_0>\ldots>v_n$ the minimum in
$
\min_{\bm\in P} \langle \bm, v_0\be_{\sigma(0)} + v_1\be_{\sigma(1)} + \cdots + v_n\be_{\sigma(n)}\rangle.
$
In particular, we have
\[
[\mathcal O(D_{-P})]_\sigma = \bT^{\widetilde \bm_\sigma} \in \ZZ[T_0^\pm, \ldots, T_n^\pm] \quad\text{for any $\sigma\in \mathfrak S_E$},
\]
where the lattice point $\widetilde\bm_\sigma$ is the vertex of $P$ achieving for any sequence $v_0 > \cdots > v_n$ the maximum in $\max_{\bm \in P} \langle \bm,  v_0\be_{\sigma(0)} + v_1\be_{\sigma(1)} + \cdots + v_n\be_{\sigma(n)}\rangle$.

\medskip
\noindent\textbf{Notation.} For a generalized permutohedron $P$, we write $\mathcal O(D_P)$ for the corresponding line bundle on $X_E$ with the $T$-linearization given as above.

\begin{rem}\label{rem:alphabeta}
Let $\Delta = \operatorname{Conv}(\be_i \mid i\in E)$ be the standard simplex in $\RR^E$, and let $-\Delta =\operatorname{Conv}(-\be_i \mid i\in E)$ be the negative standard simplex.  
Since the fan of $\PP^n$ as a toric variety is the normal fan of $\Delta$, the $T$-equivariant line bundles $\mathcal O(D_\Delta)$ and $\mathcal O(D_{-\Delta})$ on $X_E$ are non-equivariantly isomorphic to $\mathcal O(\alpha)$ and $\mathcal O(\beta)$, respectively.  Moreover, by the discussion above they satisfy
\[
[\mathcal O(D_{\Delta})]_\sigma = T_{\sigma(n)}^{-1} \quad\text{and}\quad  [\mathcal O(D_{-\Delta})]_\sigma = T_{\sigma(0)}
\]
for every permutation $\sigma\in \mathfrak S_E$.  Furthermore, fixing an element $i\in E$, we can consider the translate $\Delta - \be_i$, which is contained in the sublattice $\mathbf 1^\perp$, so that we have the equality of divisor classes
\[
\alpha = [D_{\Delta}] = \sum_{\emptyset\subsetneq S \subsetneq E} \rk_{ \be_i - \Delta}(S) [Z_S] = \sum_{i\in S \subsetneq E} [Z_S]
\]
in $A^1(X_E)$, and by a similar computation 
$\beta = \sum_{\emptyset \subsetneq S \subseteq E\setminus \{i\}} [Z_S]$.
These last definitions of $\alpha$ and $\beta$ are the same as in \cite{HuhKatz, AdiHuhKatz, ADH}.
\end{rem}

\section{Tautological classes of matroids}\label{sec:definingtaut}

In \S\ref{subsec:welldefined}, we define the tautological bundles of a realizations of matroids, and define the tautological $K$-classes and Chern classes of matroids, which we collectively refer to as ``tautological classes of matroids''.  In \S\ref{subsec:welldefinedgeneral}, we provide a slight generalization that we will not need until \S\ref{sec:KtoChow}.

\subsection{Well-definedness}\label{subsec:welldefined}
We prepare by recalling some properties of the base polytope of a matroid.  Introduced in \cite{GGMS87}, the \textbf{base polytope} $P(M)$ of a matroid $M$ with ground set $E$ is defined as
\[
P(M) = \operatorname{Conv}\Big( \be_B : B \text{ a basis of $M$}\Big) \subset \RR^E,
\]
where $\be_S =\sum_{i\in S} \be_i$ for a subset $S\subseteq E$.
We recall two well-known facts about the base polytope $P(M)$.

\begin{prop}\label{prop:basepolytope}
Let $M$ be matroid of rank $r$ with ground set $E$.
\begin{enumerate}[label = (\alph*)]
\item \label{basepolytope:coarsen} The base polytope $P(M)$ is a generalized permutohedron.  In other words, the normal fan of $P(M)$, as a fan in $\RR^E/\RR\mathbf 1$, coarsens the permutohedral fan $\Sigma_E$.
\item \label{basepolytope:orbit} Let $L\subseteq \CC^E$ be a realization of $M$, considered as a point $L\in \Gr(r;E)$.  Then, the torus-orbit-closure $\overline{T\cdot L} \subseteq \Gr(r;E)$ is torus-equivariantly isomorphic to the normal projective toric variety $X_{P(M)}$ associated to the polytope $P(M)$.  For a basis $B\subseteq E$ of $M$, the $T$-fixed point in $X_{P(M)}$ corresponding to the vertex $\be_B$ of $P(M)$ maps under the isomorphism to the $T$-fixed point $p_B$ in $\Gr(r;E)$, as denoted in \S\ref{eg:Gr}.
\end{enumerate}
\end{prop}

Part \ref{basepolytope:coarsen} is classical, tracing back to \cite{Edm70}; see \cite[\S4.4]{ACEP20} for a proof and a generalization to Coxeter matroids.  Part \ref{basepolytope:orbit} follows from combining \cite[Corollary 2.4]{GGMS87} and \cite{Whi77}; see \cite[\S5]{CDMS18} for a proof.  \Cref{prop:basepolytope}.\ref{basepolytope:coarsen} allows us to make the following definition.

\begin{defn}
For a permutation $\sigma\in \mathfrak S_E$, the \textbf{lex-first-basis} of a matroid $M$ with respect to $\sigma$, denoted by $B_{\sigma}(M)$, is the unique basis of $M$ such that the vertex $\be_{B_{\sigma}(M)}$ of $P(M)$ achieves for any sequence $v_0>\cdots > v_n$ the {maximum} in
\[
\max_{\bm \in P(M)}\langle \bm, v_0\be_{\sigma(0)} + v_1\be_{\sigma(1)} + \cdots + v_n\be_{\sigma(n)}\rangle.
\]
Writing $\overline\sigma\in \mathfrak S_E$ for the permutation defined by $\overline\sigma(i) = \sigma(n-i)$, define the basis $B_{\overline\sigma}(M)$ to be the \textbf{reverse-lex-basis} of $M$ with respect to $\sigma$.  Equivalently, the basis $B_{\overline\sigma}(M)$ is the basis of $M$ such that the vertex $\be_{B_{\overline\sigma}(M)}$ of $P(M)$ achieves for any sequence $v_0>\cdots > v_n$ the {minimum} in
\[
\min_{\bm \in P(M)}\langle \bm, v_0\be_{\sigma(0)} + v_1\be_{\sigma(1)} + \cdots + v_{n}\be_{\sigma(n)}\rangle.
\]
We simply write $B_{\sigma}$ or $B_{\overline\sigma}$ if the matroid in question is clear.
\end{defn}

That the respective maximum or minimum is achieved uniquely at a vertex of $P(M)$, independently of all $v_0 > \cdots > v_n$, follows from \Cref{prop:basepolytope}.\ref{basepolytope:coarsen}.
The equivalence of the two definitions for $B_{\overline\sigma}(M)$ follows from noting that if $w_i=-v_{n-i}$ then $v_0 > \cdots >v_n$ is equivalent to $w_0 > \cdots > w_n$, and
$v_0\be_{\sigma(0)} + v_1\be_{\sigma(1)} + \cdots + v_{n}\be_{\sigma(n)}=-(w_0\be_{\overline{\sigma}(0)} + w_1\be_{\overline{\sigma}(1)} + \cdots + w_n\be_{\overline{\sigma}(n)}).$

\begin{rem}
\label{rem:minconv}
Under the min-convention for polyhedral geometry, the cone in the normal fan of $P(M)$ corresponding to the vertex $\be_{B}$ consists of $x\in \RR^E/\RR\mathbf 1$ such that $\be_{B}$ achieves the minimum in $\min_{\bm\in P(M)} \langle \bm, x \rangle$.  In particular, for a permutation $\sigma\in \mathfrak S_E$, the definition of $\be_{B_{\overline\sigma(M)}}$ implies that the cone of $\Sigma_E$ corresponding to $\sigma$ is contained in the cone in the normal fan of $P(M)$ corresponding to the vertex $\be_{B_{\overline\sigma}(M)}$.
\end{rem}

\begin{rem}\label{rem:lexfirstbasis}
Choosing $v_0\gg \dots \gg v_n$ justifies the terminology ``lex-first-basis'' because it implies  $B_{\sigma}(M)$ is the first basis in the lexicographic ordering when the ground set has the linear order $\sigma(0)\prec \dots \prec \sigma(n)$.
\end{rem}

Combining the two parts of \Cref{prop:basepolytope}, we have the following.

\begin{lem}\label{lem:keymap}
If $L\subseteq \CC^E$ is a realization of a matroid $M$ of rank $r$, then one has a $T$-equivariant map
\begin{equation}\label{eq:keymap}\tag{$\dagger$}
\varphi_L\colon  X_E \to X_{P(M)} \overset\sim\to \overline{T\cdot L} \subset \Gr(r;E)
\end{equation}
which sends the identity point $\mathbf 1$ of the open torus $T/\CC^*$ of $X_E$ to the point $L$ in $\Gr(r;E)$. For each permutation $\sigma\in \mathfrak S_E$, the map $\varphi_L$ sends the $T$-fixed point $p_\sigma$ in $X_E$ to the $T$-fixed point $p_{B_{\overline\sigma}(M)}$ in $\Gr(r;E)$.
\end{lem}

\begin{proof}
\Cref{prop:basepolytope}.\ref{basepolytope:coarsen} implies that we have a $T$-equivariant map $X_E \to X_{P(M)}$ induced by a coarsening of fans.  By \Cref{rem:minconv}, for a permutation $\sigma\in \mathfrak S_E$, the cone of $\Sigma_E$ corresponding to $\sigma$ is contained in the cone in the normal fan of $P(M)$ corresponding to the vertex $\be_{B_{\overline\sigma}(M)}$.  The rest of the lemma now follows from \Cref{prop:basepolytope}.\ref{basepolytope:orbit}.
\end{proof}

We thus have the following proposition.  Recall that $\uCCinv^{E}=X_{E}\times \CCinv^{E}$, where $\CCinv^{E}$ denotes the vector space $\CC^E$ with the inverse action of $T$.

\begin{prop}
\label{prop:bundlebijection}
The assignment $L \mapsto \crem \varphi_L^*\cS$ (resp.\ $L\mapsto \crem \varphi_L^*\cQ$), where $\cS$ (resp.\ $\cQ$) is the tautological sub (resp.\ quotient) bundle of $\Gr(\dim L; E)$, is a bijection between subspaces $L\subseteq \CC^E$ and $T$-equivariant sub (resp.\ quotient) bundles of $\uCCinv^{E}$ whose fiber at the identity $\mathbf{1}\in T/\CC^*$ is $L$ (resp.\ $\CC^E/L$).
\end{prop}

\begin{proof}
Since $\cS$ (resp. $\cQ$) is a sub (resp. quotient) bundle of the trivial bundle $\Gr(\dim L; E) \times \CC^E$, where $T$ acts on $\CC^E$ by the standard action $t \cdot (x_0,\dots,x_n) = (t_0x_0,\dots,t_nx_n)$, the pullback $\varphi_L^*\cS$ (resp.\ $\varphi_L^*\cQ$) is a $T$-equivariant sub (resp.\ quotient) bundle of $X_E \times \CC^E$.  Applying the Cremona involution, we thus have that $\crem\varphi_L^*\cS$ (resp.\ $\crem\varphi_L^*\cQ$) is a $T$-equivariant sub (resp.\ quotient) bundle of $\uCCinv^{E} = X_E \times \CCinv^E$ whose fiber over $\mathbf{1}$ is $L$ (resp. $\CC^E/L$). Since such a sub (resp.\ quotient) bundle is uniquely determined by its fibers over the dense torus of $X_{E}$, which by torus-equivariance is uniquely determined by its fiber at the identity, the assignment is a bijection.
\end{proof}

We thus find that the notions in \Cref{defn:tautbundle}, reproduced below, are well-defined.

\theoremstyle{definition}
\newtheorem*{defn:tautbundle}{\Cref{defn:tautbundle}}
\begin{defn:tautbundle}
For an $r$-dimensional linear subspace $L\subseteq \CC^E$, the \textbf{tautological subbundle} $\cS_L$ and the \textbf{tautological quotient bundle} $\cQ_L$ of $L$ are defined by
\begin{align*}
\cS_L \coloneqq \ &  \text{the unique $T$-equivariant rank $r$ subbundle of $\uCCinv^{E}$ whose fiber at $\mathbf 1$ is $L$, \ \ and}\\
 \cQ_L \coloneqq \ & \text{the unique $T$-equivariant rank $|E|-r$ quotient bundle of $\uCCinv^{E}$ whose fiber at $\mathbf 1$ is $\CC^E/L$}.
\end{align*}
Equivalently, $\cS_L$ and $\cQ_L$ are defined as
$$\cS_L\coloneqq \crem\varphi_L^*\mathcal{S}\qquad\text{and}\qquad\cQ_L\coloneqq \crem\varphi_L^*\mathcal{Q}.$$
\end{defn:tautbundle}

In other words, the fiber of $\cS_L$ over a point $\overline t$ in the open torus $T/\CC^*$ of $X_E$ is identified with the subspace $t^{-1}L$ of $\CC^E$.  \Cref{rem:whyinverse} explains the inverse $t^{-1}$.
We now identify the localizations of $[\mathcal{S}_L]$ and $[\mathcal{Q}_L]$ at the torus fixed points of $X_E$, proving that the $K$-classes depend only on the matroid $M$ of the realization $L$.

\begin{prop}
\label{prop:tautdependsonmatroid}
For any realization $L\subseteq \CC^E$ of a matroid $M$ with ground set $E$, the $T$-equivariant $K$-classes $[\mathcal{S}_L]$ and $[\mathcal{Q}_L]$ only depend on the matroid $M$, and satisfy
\[
[\mathcal{S}_L]_\sigma=\sum_{i\in B_{\sigma}(M)} T_i^{-1}\quad\text{and}\quad [\mathcal{Q}_L]_\sigma=\sum_{i \in E\setminus B_{\sigma}(M)} T_i^{-1} \quad\text{for every permutation $\sigma\in \mathfrak S_E$}.
\]
\end{prop}

\begin{proof}
Let $r = \dim L$, and let $\varphi_L\colon  X_E \to X_{P(M)}\overset\sim \to \overline{T\cdot L} \subset \Gr(r;E)$ be as above.  For a permutation $\sigma\in \mathfrak S_E$, the map $\varphi_L\colon X_E\to \overline{T\cdot L}$ sends $p_{\sigma}$ to the fixed point of $X_E$ corresponding to $B_{\overline{\sigma}}(M)$. Therefore $\varphi_L(p_\sigma)=p_{B_{\overline\sigma}}\in \Gr(r;E)$, so we have
\[
[\varphi_L^* \cS]_{\sigma} = [\cS]_{B_{\overline\sigma}(M)} = \sum_{i\in B_{\overline\sigma}(M)} T_i \quad\text{and}\quad [\varphi_L^* \cQ]_{\sigma} = [\cQ]_{B_{\overline\sigma}(M)}= \sum_{i\in E\setminus B_{\overline\sigma}(M)} T_i
\]
by \S\ref{eg:Gr}.  Applying the Cremona involution (\Cref{rem:crem}) then yields the desired statement.
\end{proof}

This description of the $T$-equivariant $K$-classes of tautological bundles of realizations of matroids extend to arbitrary (not necessarily realizable) matroids.

\begin{prop}
\label{prop:welldefined}
For any matroid $M$ (not necessarily realizable) on ground set $E$, the two $\mathfrak S_E$-tuples $[\cS_M]$ and $[\cQ_M]$ of Laurent polynomials defined by
\[
[\cS_M]_\sigma = \sum_{i \in B_{\sigma}}T_i^{-1} \quad\text{and}\quad
[\cQ_M]_\sigma = \sum_{i\not\in B_{\sigma}}T_i^{-1} \quad\text{for $\sigma\in \mathfrak S_E$}
\]
are well-defined $T$-equivariant $K$-classes on $X_E$ satisfying $[\cS_M] + [\cQ_M] = [\uCCinv^{E}]$.
\end{prop}

\begin{proof}
Let $\sigma$ and $\sigma'$ be permutations such that $\sigma' = \sigma \circ (i,i+1)$ for some $i\in E$.  Note that $\overline{\sigma'}=\overline{\sigma}\circ (n-i,n-i-1)$. So, the two maximal cones in $\Sigma_E$ corresponding to $\overline{\sigma}$ and $\overline{\sigma'}$ intersect in a codimension 1 cone whose linear span is the hyperplane normal to $\be_{\overline{\sigma}(n-i)} - \be_{\overline{\sigma}(n-i-1)}$.  Since the normal fan of the base polytope $P(M)$ coarsens $\Sigma_E$, the two vertices $\be_{B_{\sigma}}$ and $\be_{B_{\sigma'}}$ are either identical, or their difference is equal to $\be_{\overline{\sigma}(n-i)} - \be_{\overline{\sigma}(n-i-1)}=\be_{\sigma(i)}-\be_{\sigma(i+1)}$. In other words, the lex-first-bases $B_\sigma$ and $B_{\sigma'}$ are thus identical or have symmetric difference $\{\sigma(i), \sigma(i+1)\}$.  Hence, the two $\mathfrak S_E$-tuples satisfy the condition in \Cref{thm:localization}.\ref{localization:K}.  Their sum is the $\mathfrak S_E$-tuple such that we have $\sum_{i\in E}T_i^{-1}$ for all $\sigma\in \mathfrak S_E$, which defines the class of $[\uCCinv^{E}]$.
\end{proof}

\begin{defn}
\label{defn:tautclasses}
For a matroid $M$ with ground set $E$, we define the \textbf{tautological sub} (resp.\ \textbf{quotient}) \textbf{$\boldsymbol K$-class} of $M$ to be the $T$-equivariant $K$-class $[\mathcal{S}_M]$ (resp.\ $[\mathcal{Q}_M]$) in $K_0^T(X_{E})$ defined by
$$[\mathcal{S}_M]_\sigma=\sum_{i\in B_{\sigma}}T_i^{-1} \quad\text{resp.}\quad [\mathcal{Q}_M]_\sigma=\sum_{i\not\in B_{\sigma}}T_i^{-1} \quad\text{for } \sigma\in \mathfrak S_E.$$
We define the \textbf{tautological sub} (resp. \textbf{quotient}) \textbf{equivariant Chern classes} of $M$ to be associated equivariant Chow classes $c^T_j(\mathcal{S}_M)$ (resp. $c^T_j(\mathcal{Q}_M)$) in $A_T^\bullet(X_E)$, which by definition are given by
$$c^T_j(\mathcal{S}_M)_\sigma=\operatorname{Elem}_j(\{-t_i\}_{i\in B_{\sigma}}) \quad\text{resp.}\quad c^T_j(\mathcal{Q}_M)_\sigma=\operatorname{Elem}_j(\{-t_i\}_{i\not \in B_{\sigma}}) \quad\text{for } \sigma\in \mathfrak S_E.$$
\end{defn}

\begin{eg}\label{eg:alphabeta}
Recall the two distinguished divisor classes $\alpha$ and $\beta$ on $X_E$, and denote by $U_{r,E}$ the uniform matroid on $E$ of rank $r$.  Note that the base polytope $P(U_{1,E})$ is the standard simplex $\Delta = \operatorname{Conv}(\be_i \mid i\in E)$.  Since $E\setminus B_\sigma(U_{n,E}) = \{\sigma(n)\}$ and $B_\sigma(U_{1,E}) = \{\sigma(0)\}$, by comparing the localizations at $p_{\sigma}$ for $\sigma\in \mathfrak{S}_E$, it follows from \Cref{rem:alphabeta} that we have
$[\cQ_{U_{n,E}}] =[\mathcal O(D_{P(U_{1,E})})]$ and $[\cS_{U_{1,E}}^\vee] = [\mathcal O(D_{-P(U_{1,E})})]$ as classes in $K_T^0(X_E)$, and thus
\[
[\cQ_{U_{n,E}}] = [\mathcal O(\alpha)] \quad\text{and}\quad [\cS_{U_{1,E}}^\vee] = [\mathcal O(\beta)]
\]
as non-equivariant $K$-classes.  Moreover, since $[\cS_M] + [\cQ_M] = [\uCCinv^E]$ for a matroid $M$, we have
\[
c(\cS^\vee_{U_{n,E}}) = c(\cQ^\vee_{U_{n,E}})^{-1} = 1 + \alpha + \cdots +\alpha^n \quad\text{and}\quad c(\cQ_{U_{1,E}}) = c(\cS_{U_{1,E}})^{-1} = 1 + \beta + \cdots +\beta^n
\]
as non-equivariant Chow classes in $A^\bullet(X_E)$.
\end{eg}

\begin{eg}
\label{eg:minusPM}
We note $[\mathcal{O}(D_{-P(M)})]= [\det \mathcal{S}_M^{\vee}]$. Indeed, by the discussion in \S\ref{subsec:Tdiv} above \Cref{rem:alphabeta},
$$[\mathcal{O}(D_{-P(M)})]_{\sigma}=\bT^{\be_{B_{\sigma}(M)}} =\prod_{i\in B_{\sigma}(M)} T_i = [\det \mathcal{S}_M^{\vee}]_{\sigma} \quad\text{for any $\sigma\in \mathfrak S_E$}.$$
\end{eg}

\begin{eg}
Let $M=U_{1,\{0,1\}}\oplus U_{1,\{2\}}$.  The following figure illustrates the fan $\Sigma_{\{0,1,2\}}$ and the class $[\mathcal{S}_M]$ represented by assignments of Laurent polynomials to the maximal cones.
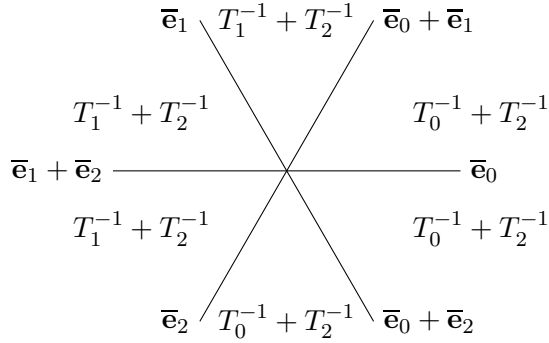
\begin{center}
    \begin{tikzpicture}[scale=0.77]
    \draw (1.5,2.6)node[anchor=west]{$\overline\be_0+\overline\be_1$}--(-1.5,-2.6)node[anchor=east]{$\overline\be_2$};
    \draw (-3,0)node[anchor=east]{$\overline\be_1+\overline\be_2$}--(3,0) node[anchor=west]{$\overline\be_0$};
    \draw(1.5,-2.6)node[anchor=west]{$\overline\be_0+\overline\be_2$}--(-1.5,2.6)node[anchor=east]{$\overline\be_1$};
    \draw (2,1)node[anchor=west]{$T_0^{-1}+T_2^{-1}$};
     \draw (0,2.6)node[]{$T_1^{-1}+T_2^{-1}$};
       \draw (-2.5,1)node[]{$T_1^{-1}+T_2^{-1}$};
         \draw (-2.5,-1)node[]{$T_1^{-1}+T_2^{-1}$};
         \draw (0,-2.6)node[]{$T_0^{-1}+T_2^{-1}$};
         \draw (2,-1)node[anchor=west]{$T_0^{-1}+T_2^{-1}$};
    \end{tikzpicture}
\end{center}
This matroid $M$ has a realization $L\subseteq \CC^{\{0,1,2\}}$ where $L$ is the row-span of the matrix $\begin{bmatrix} 1 & 1 & 0 \\ 0 & 0 & 1\end{bmatrix}$.  Consider a permutation $\sigma$ defined by $(\sigma(0), \sigma(1),\sigma(2)) = (2, 0, 1)$.  Note that $\overline\be_0 + 2\overline\be_2$ lies in the interior of $\operatorname{Cone}(\overline\be_{\sigma(0)}, \overline\be_{\sigma(0)}+\overline\be_{\sigma(1)})$, the cone corresponding to $\sigma$.  Hence, the map $\lambda_\sigma\colon  \CC^* \to (\CC^*)^{\{0,1,2\}}/\CC^*$ defined by $s\mapsto [s : 1 : s^2]$ limits as $s\to 0$ to the $T$-fixed point $p_\sigma$ in $X_{\{0,1,2\}}$.
The limit as $s\to 0$ of $\lambda_\sigma(s)^{-1}L = \operatorname{row-span} \begin{bmatrix} s^{-1} & 1 & 0 \\  0 & 0 & s^{-2}\end{bmatrix}$ in $\Gr(2;\{0,1,2\})$  is
$\operatorname{row-span} \begin{bmatrix} 1 & 0 & 0 \\  0 & 0 & 1\end{bmatrix}$. This verifies as expected that $[\cS_L]_\sigma = T_0^{-1} + T_2^{-1}$.
\end{eg}

\subsection{Matroid analogues of Grassmannian $\boldsymbol K$-classes}\label{subsec:welldefinedgeneral}
Here we generalize the association of $[\mathcal{S}_M]$ and $[\mathcal{Q}_M]$ from the vector bundles $\mathcal{S}$ and $\mathcal{Q}$ on $\Gr(r;E)$, to arbitrary $K$-classes on $\Gr(r;E)$.  We will not need this until  \S\ref{sec:KtoChow}, where we relate our computations to the computations of \cite{FS12}. One may consider it as a combinatorial abstraction for arbitrary matroids of a pullback map from the $K$-ring of $\Gr(r;E)$ to the $K$-ring of the permutohedral variety.

\begin{prop}\label{prop:welldefinedgeneral}
Let $M$ be a matroid of rank $r$ with ground set $E$.  For any class $[\mathcal E]\in K_T^0(\Gr(r;E))$, 
an element $[\mathcal E_M]$ defined by
\[
[\mathcal E_M]_\sigma = [\mathcal E]_{B_{\sigma}(M)}(T_0^{-1}, \ldots, T_n^{-1}) \in \ZZ[T_0^\pm, \ldots, T_n^\pm] \quad\text{for each }\sigma\in \mathfrak S_E
\]
is a well-defined element in $K_T^0(X_E)$ such that if $L\subseteq \CC^E$ is any realization of $M$, then $[\mathcal E_M] = \crem \varphi_L^*[\mathcal E]$ where $\varphi_L\colon  X_E \to \Gr(r;E)$ is the map \eqref{eq:keymap} in \Cref{lem:keymap}.  Moreover, the assignment $[\mathcal E] \mapsto [\mathcal E_M]$ is a ring homomorphism $K_T^0(\Gr(r;E)) \to K_T^0(X_E)$ that respects exterior powers.
\end{prop}

The notation of \Cref{prop:welldefinedgeneral} is consistent with our definition of the tautological $K$-classes $[\cS_M]$ and $[\cQ_M]$ of a matroid $M$.  In particular, \Cref{prop:welldefinedgeneral} implies that the notation $[\bigwedge^i\cS_M]$ is unambiguous, since $[(\bigwedge^i\cS)_M] = [\bigwedge^i (\cS_M)]$, and likewise for exterior powers of $\cQ_M$ and the duals $\cS_M^\vee, \cQ_M^\vee$.

\begin{proof}
By the property of the Cremona involution (\Cref{rem:crem}), it suffices to show that an element $[\mathcal E_M']$ defined by $[\mathcal E_M']_\sigma = [\mathcal E]_{B_{\overline\sigma}(M)} \in \ZZ[T_0^\pm, \ldots, T_n^\pm]$ for $\sigma\in \mathfrak S_E$ is a well-defined element in $K_T^0(X_E)$ such that $[\mathcal E_M'] = \varphi_L^*[\mathcal E]$ for a realization $L$.
To see well-definedness, we check that $[\mathcal E_M']$ satisfies the condition in \Cref{thm:localization}.\ref{localization:K}.  Suppose that $\sigma$ and $\sigma'$ are maximal cones in $\Sigma_E$ sharing a codimension 1 face, whose linear span is $\{x\in \RR^E/\RR\mathbf 1 \mid x_i = x_j\}$ for some $i\neq j \in E$.  Since the subsets $B_{\overline\sigma}(M)$ and $B_{\overline{\sigma'}}(M)$ are either identical or have symmetric difference $\{i,j\}$, the condition for $[\mathcal E]\in K_T^0(\Gr(r;E))$ as noted in \S\ref{eg:Gr} implies that $[\mathcal E'_M]_{B_{\overline\sigma}(M)} \equiv [\mathcal E'_M]_{B_{\overline{\sigma'}}(M)} \mod (T_i - T_j)$, as desired.
That $[\mathcal E_M'] = \varphi_L^*[\mathcal E]$ for a realization $L$ follows from the fact that $\varphi_L$ maps the point $p_\sigma$ in $X_E$ to $p_{B_{\overline\sigma}(M)}$ in $\Gr(r;E)$ for any permutation $\sigma\in \mathfrak S_E$ by \Cref{lem:keymap}.  That the assignment $[\mathcal E]\mapsto [\mathcal E_M]$ is a ring homomorphism respecting exterior powers is straightforward to check from the defining formula $[\mathcal E_M]_\sigma = [\mathcal E]_{B_{\sigma}(M)}(T_0^{-1}, \ldots, T_n^{-1})$ for each $\sigma\in \mathfrak S_E$.
\end{proof}

\section{A unifying Tutte polynomial formula}\label{sec:unifyingTutte}

In this section, we prove \Cref{thm:4degintro}, reproduced below, by establishing a deletion-contraction relation for the tautological Chern classes of matroids.

\theoremstyle{theorem}
\newtheorem*{thm:4degintro}{\Cref{thm:4degintro}}
\begin{thm:4degintro}
Let $\int_{X_E}\colon  A^\bullet(X_E) \to \ZZ$ be the degree map on $X_E$.  For a matroid $M$ of rank $r$ with ground set $E$, define a polynomial
\[
t_M(x,y,z,w) = (x+y)^{-1}(y+z)^{r}(x+w)^{|E|-r}T_M\Big(\frac{x+y}{y+z},\frac{x+y}{x+w}\Big).
\]
Then, we have an equality
\[
\sum_{i+j+k+\ell=n} \left(\int_{X_E}\alpha^i\beta^j  c_k(\mathcal{S}^{\vee}_M)c_\ell(\mathcal{Q}_M)\right)x^iy^jz^kw^\ell = t_M(x,y,z,w).
\]
\end{thm:4degintro}

To do so, we will use the following property of Tutte polynomials.
See \cite[Section 6.2]{BO92}, specifically Exercise 6.5(b),  for details and further references.


\begin{prop}\label{prop:TGinv}
There is a unique $5$-variable polynomial $G_M(u,v,a,b,\gamma)\in \mathbb{Z}[u,v,a,b,\gamma]$ associated to a matroid $M$,  called a \textbf{generalized Tutte-Grothendieck invariant} of $M$, satisfying the following properties:
    \begin{enumerate}
        \item (Base case) If $|E| = 1$, then
    \[
    G_M(u,v,a,b,\gamma) = \begin{cases}
    u & \text{if $M$ has rank 1 (i.e. $M$ is a coloop)}\\
    v & \text{if $M$ has rank 0 (i.e. $M$ is a loop).}
    \end{cases}\]
    \item  (Deletion-contraction relation) If $|E| \geq 2$ and $i\in E$, then
    \[
    G_M(u,v,a,b,\gamma) = \begin{cases}
    \gamma uG_{M/ i}(u,v,a,b,\gamma)& \text{if $i\in E$ is a coloop in $M$}\\
    \gamma vG_{M\setminus i}(u,v,a,b,\gamma)& \text{if $i\in E$ is a loop in $M$}\\
    a G_{M\setminus i}(u,v,a,b,\gamma) + b G_{M/i}(u,v,a,b,\gamma)& \text{if $i\in E$ is neither a loop nor a coloop in $M$}.\end{cases}\]
    \end{enumerate}
    For a matroid $M$ with ground set $E$ and of rank $r$, this polynomial is given by
    $$G_M(u,v,a,b,\gamma)  = \gamma^{-1}b^{ r}a^{|E|-r} T_M\big(\frac{u\gamma}{b}, \frac{v\gamma}{a}\Big).$$
\end{prop}
 In particular, for the polynomial appearing in \Cref{thm:4degintro} we have
$$G_M(1,1,x+w,y+z,x+y)=(x+y)^{-1}(y+z)^r(x+w)^{|E|-r}T_M(\frac{x+y}{y+z},\frac{x+y}{x+w}).$$
Let us now restate \Cref{thm:4degintro} as follows.
Denote an element $\xi_M \in A^\bullet(X_E)[x,y,z,w]$ by
\begin{align*}
\xi_M &= (1+ \alpha x + \cdots + \alpha^nx^n) (1+\beta y + \cdots + \beta^n y^n) c(\mathcal{S}_M^{\vee},z)c(\mathcal{Q}_M,w) \\
&= c(\cS_{U_{n,E}}^\vee,x) c(\cQ_{U_{1,E}},y) c(\mathcal{S}_M^{\vee},z) c(\mathcal{Q}_M,w),
\end{align*}
where the second equality follows from \Cref{eg:alphabeta}.
We show that $\int_{X_E}\xi_M$ 
is the generalized Tutte-Grothendieck invariant in \Cref{prop:TGinv} with $u=v=1$, $a=x+w$, $b=y+z$, and $\gamma = x+y$.



For the base case, note that if $|E| = 1$ then $\xi_M = 1$ because $X_E$ is 0-dimensional. It follows that $\int_{X_E} \xi_M = 1$. 

If $|E|\geq 2$, we will show that $\int_{X_E}\xi_M$ satisfies the deletion-contraction relation in \Cref{prop:TGinv} in two steps (for concreteness we will take $i=n$ in our arugments).
First, in \S\ref{subsec:projectionmap}, we consider a surjective map $f\colon  X_E \to X_{E\setminus n}$ of permutohedral varieties defined by deleting $n\in E$, and study the behavior of the tautological Chern classes of $M$ under the pushforward map $f_*\colon  A^\bullet(X_E) \to A^{\bullet-1}(X_{E\setminus n})$.  Then, in \S\ref{subsec:delcont}, we use these observations to show \Cref{thm:delcont}, which states that $\xi_M$ satisfies
\[
f_*\xi_M = \begin{cases}
(x+y)\xi_{M\setminus n}& \text{if $n$ is a loop }\\
(x+y) \xi_{M/n} & \text{if $n$ is a coloop}\\
(x+w) \xi_{M\setminus n}+ (y+z)\xi_{M/n} &\text{if $n$ is neither a loop nor a coloop}.
\end{cases}
\]
Noting that $ \int_{X_E} \xi_M=\int_{X_{E\setminus n}} f_* \xi_M$ by the functoriality of pushforward maps, we conclude that $\int_{X_E} \xi_M$ satisfies the deletion-contraction relation of  \Cref{prop:TGinv}, and \Cref{thm:4degintro} follows.

\subsection{A projection map of permutohedral varieties}\label{subsec:projectionmap}
As before, we let $E = \{0,1, \ldots, n\}$, and assume throughout that $|E| \geq 2$.  Denote $T' = (\CC^*)^{E\setminus n}$.
The projection map $\RR^E/\RR\mathbf 1 \to \RR^{(E\setminus n)}/\RR\mathbf 1$ induces a map of fans $\Sigma_E \to \Sigma_{E\setminus n}$ because the cone of $\Sigma_E$ corresponding to a permutation $\widetilde\sigma\in \mathfrak S_E$ maps to the cone of $\Sigma_{E\setminus n}$ corresponding to the  permutation $\sigma\in \mathfrak S_{E\setminus n}$ where the sequence $(\sigma(0), \ldots, \sigma(n-1))$ is obtained from $(\widetilde\sigma(0), \ldots, \widetilde\sigma(n))$ by omitting the entry $n=\widetilde\sigma (\widetilde\sigma^{-1}(n))$.

\begin{defn}\label{defn:projmap}
Let $f\colon  X_E\to X_{E\setminus n}$ be the toric morphism of permutohedral varieties induced by the projection map $\RR^E/\RR\mathbf 1 \to \RR^{(E\setminus n)}/\RR\mathbf 1$, where $X_E$ and $X_{E\setminus n}$ are considered as a $T$-variety and a $T'$-variety, respectively.  The underlying map of tori $T/\mathbb{C}^{\times}\to T'/\mathbb{C}^{\times}$ corresponding to the toric morphism $f$ is induced by the map $T\to T'$ given by projection onto the first $n$ coordinates.

Given a permutation $\sigma \in \mathfrak S_{E\setminus n}$, define for each $i\in E$ a permutation $\sigma^i \in \mathfrak S_{E}$ by
\[
\sigma^i(j) = \begin{cases}
\sigma(j) & j < i\\
n & j = i \\
\sigma(j-1) & j > i
\end{cases}\qquad\qquad \text{for $0\leq j \leq n$,}
\]
so that the preimage $f^{-1}(p_\sigma)$ of the $T'$-fixed point $p_\sigma$ of $X_{E\setminus n}$ consists of $T$-fixed points $\{p_{\sigma^{i}}\}_{i\in E}$ of $X_E$.
\end{defn}

One may consider the permutation $\sigma^i \in \mathfrak S_E$ as the permutation obtained from $\sigma\in \mathfrak S_{E\setminus n}$ by inserting $n$ right after $\sigma(i-1)$ in the linear order $\sigma(0)\prec \ldots \prec \sigma(n-1)$ of $\sigma$.  In other words, the linear order defining the permutation $\sigma^i$ is given by $\sigma(0) \prec \cdots \prec \sigma(i-1) \prec n \prec \sigma(i) \prec \cdots \prec \sigma(n-1)$.

\medskip
In order to push forward $\xi_M$ under $f_{*}\colon  A^{\bullet}(X_E)\to A^{\bullet-1}(X_{E\setminus n})$, our strategy will be to take an equivariant version $\xi_M^T$ of $\xi_M$ and compute the image under the composite $A^{\bullet}_T(X_{E})\xrightarrow{f_{*}} A^{\bullet-1}_{T}(X_{E\backslash\{n\}})\to A^{\bullet-1}_{T'}(X_{E\backslash\{n\}})$, where the first map is given by equivariant pushforward and the second map is induced by the map $T'\to T$ given by inclusion into the first $n$ coordinates. This would recover the pushforward $f_{*}(\xi_M)$ by the commutativity of the diagram below.
\begin{center}
\begin{tikzcd}
\xi_M^T\in A^{\bullet}_T(X_{E})\hspace{1em} \ar[d,"f_{*}"]\ar[rr] & & \hspace{1em} A^{\bullet}(X_{E})\ni\xi_M \ar[d,"f_{*}"] \\
 A^{\bullet-1}_{T}(X_{E\backslash\{n\}})\ar[r]& A^{\bullet-1}_{T'}(X_{E\backslash\{n\}}) \ar[r] &  A^{\bullet-1}(X_{E\backslash\{n\}})
\end{tikzcd}
\end{center}
To compute the image of $\xi_M^T$ under the composite $A^{\bullet}_T(X_{E})\xrightarrow{f_{*}} A^{\bullet-1}_{T}(X_{E\backslash\{n\}})\to A^{\bullet-1}_{T'}(X_{E\backslash\{n\}})$, we compute the localizations of the image at each of the torus-fixed points. Our basic tool for doing so  is \Cref{lem:telescope} below.

\begin{lem}\label{lem:telescope}
For a $T$-equivariant Chow class $\xi^T \in A^\bullet_T(X_E)$, the pushforward map $f_*\colon  A_T^{\bullet}(X_E) \to A_{T}^{\bullet-1}(X_{E\setminus n})$ satisfies
\[
(f_*\xi^T)_{\sigma}|_{t_n=0}= \sum_{i=0}^{n-1}t_{\sigma(i)}^{-1}(\xi^T_{\sigma^{i+1}}\bigr|_{t_n=0}-\xi^T_{\sigma^i}|_{t_n=0}\bigr)\in A^{\bullet-1}_{T'}(\pt) \qquad\text{for any permutation $\sigma \in \mathfrak S_{E\setminus n}$,}
\]
where the right-hand-side always simplifies to a polynomial in $A^{\bullet-1}_{T'}(\pt)=\ZZ[t_0, \ldots, t_{n-1}]$.
\end{lem}

\begin{proof} We apply the localization formula \cite[Corollary~4.2]{Bri97} with the identification of the torus action on the tangent spaces to the torus-fixed points of $X_E$  at the end of  \S\ref{subsec:permutohedralcone}, and write 
\begin{align*}
    \xi^T = \sum_{\tau \in \mathfrak{S}_E} \frac{\xi^T_\tau}{\prod_{\ell=1}^n( t_{\tau(\ell-1)} - t_{\tau(\ell)})} [p_\tau],
\end{align*}
where $[p_\tau]$ is the class of the $T$-fixed point $p_\tau \in X_E$, and its coefficient is the equivariant multiplicity of $\xi^T$ at $p_\tau$ by \cite[Theorem 5.4]{Bri97}. Applying $f_*$ to this equation and regrouping gives,
\begin{align*}
    f_*\xi^T = \sum_{\sigma \in \mathfrak{S}_{E \setminus n}} \sum_{i=0}^n \frac{\xi^T_{\sigma^i}}{\prod_{\ell=1}^n( t_{\sigma^i(\ell-1)} - t_{\sigma^i(\ell)})} [p_\sigma],
\end{align*}
Here, we are treating $X_{E\setminus n}$ as a $T$-variety via the map $T\to T'$, and the class $f_*\xi^T$ as a $T$-equivariant Chow class in $A^\bullet_T(X_{E\setminus n})$.
Applying the equivariant multiplicity map of \cite[Theorem 4.2]{Bri97} we obtain
\begin{align*}
    \frac{(f_*\xi^T)_\sigma }{\prod_{j=1}^{n-1} (t_{\sigma(j-1)}-t_{\sigma(j)}) } = \sum_{i=0}^n \frac{\xi^T_{\sigma^i}}{\prod_{\ell=1}^n( t_{\sigma^i(\ell-1)} - t_{\sigma^i(\ell)})},
\end{align*}
which yields
\begin{align*}(f_*\xi^T)_{\sigma}&=\sum_{i=0}^{n} \frac{\prod_{j=1}^{n-1} (t_{\sigma(j-1)}-t_{\sigma(j)})}{\prod_{\ell=1}^n (t_{\sigma^i(\ell-1)}-t_{\sigma^i(\ell)})}\xi^T_{\sigma^{i}}\\&=-\frac{1}{t_{\sigma(0)}-t_n}\xi^T_{\sigma^0}+\frac{1}{t_{\sigma(n-1)}-t_n}\xi^T_{\sigma^n}+\sum_{i=1}^{n-1} \frac{t_{\sigma(i-1)}-t_{\sigma(i)}}{(t_{\sigma(i-1)}-t_n)(t_n-t_{\sigma(i)})}\xi^T_{\sigma^i}\\
&=-\frac{1}{t_{\sigma(0)}-t_n}\xi^T_{\sigma^0}+\frac{1}{t_{\sigma(n-1)}-t_n}\xi^T_{\sigma^n}+\sum_{i=1}^{n-1} (\frac{1}{t_{\sigma(i-1)}-t_n}-\frac{1}{t_{\sigma(i)}-t_n})\xi^T_{\sigma^i}.\end{align*}
Reordering the terms and setting $t_n=0$ then yields the desired result.
\end{proof}

We now consider how the $T$-equivariant tautological Chern classes of matroids behave with respect to the pushforward map $f_*$.  Let us prepare with the following.

\begin{lem}\label{lem:delcontcomb}
Let $M$ be a matroid on ground set $E$, and let $\sigma \in \mathfrak S_{E\setminus n}$. For all $0\leq i \leq n$, if $n$ is a loop then $B_{\sigma^i}(M) = B_\sigma(M\setminus n)$, and if $n$ is a coloop then $B_{\sigma^i}(M) = B_\sigma(M/n) \sqcup \{n\}$.  If $n\in E$ is neither a loop nor a coloop, there is a $0\le k\le n-1$ such that
\begin{itemize}
    \item $B_{\sigma^0}(M)=B_{\sigma^1}(M)=\cdots=B_{\sigma^k}(M)=B_{\sigma}(M/n)\sqcup \{n\}$,
    \item $ B_{\sigma^{k+1}}(M)=\cdots=B_{\sigma^n}(M)=B_{\sigma}(M\setminus n)$, \text{ and}
    \item $B_{\sigma}(M/n)\sqcup \{\sigma(k)\}=B_{\sigma}(M\setminus n)$.
\end{itemize}
\end{lem}

\begin{proof}
When $n$ is not a coloop, the set of bases of $M\setminus n$ is $\{B \mid \text{$B$ a basis of $M$ such that } B\not\ni n\}$, and when $n$ is not a loop, the set of bases of $M/n$ is $\{B \setminus n \mid \text{$B$ a basis of $M$ such that } B\ni n\}$.
Thus, for all $0\leq i \leq n$, if $n$ is a loop then $B_{\sigma^i}(M) = B_\sigma(M\setminus n)$, and if $n$ is a coloop then $B_{\sigma^i}(M) = B_\sigma(M/n) \sqcup \{n\}$.

For all $0\leq i \leq n$, the definition of $\sigma^i\in \mathfrak S_E$ and the ``lex-first'' property (\Cref{rem:lexfirstbasis}) imply that if $n\in B_{\sigma^i}(M)$ then $n\in B_{\sigma^{j}}(M)$ for all $0\leq j \leq i$, and if $n\not\in B_{\sigma^i}(M)$ then $n\not\in B_{\sigma^j}(M)$ for all $i \leq j \leq n$.
If $n$ is neither a loop nor a coloop,
then it is contained in some basis of $M$, and also avoids some other basis of $M$.
Thus, the lex-first-basis $B_{\sigma^0}(M)$ of $M$ with respect to the linear ordering $n \prec \sigma(0) \prec \cdots \prec \sigma(n-1)$ contains $n$, and the lex-first-basis $B_{\sigma^n}(M)$ of $M$ with respect to the linear ordering $\sigma(0) \prec \cdots \prec \sigma(n-1)\prec n$ does not contain $n$.  Hence, the maximum $\max\{0\leq i \leq n-1 \mid n \in B_{\sigma^i}(M)\}$ is well-defined, and setting $k$ to be this maximum, we see that the first two bullet points follow from the lex-first property.

For the third, we first note that the two cones in $\Sigma_E$ corresponding to permutations $\sigma^k$ and $\sigma^{k+1}$ share the codimension 1 face whose linear span is the hyperplane normal to $\be_{\sigma(k)} - \be_{n}$.  Hence, since the normal fan of $P(M)$ coarsens $\Sigma_E$ (\Cref{prop:basepolytope}.\ref{basepolytope:coarsen}), and since $B_{\sigma^k}(M) \neq B_{\sigma^{k+1}}(M)$ by definition of $k$, then as in the proof of \Cref{prop:welldefined} we have the symmetric difference of $B_{\sigma^k}(M)$ and $B_{\sigma^{k+1}}(M)$ is $\{n, \sigma(k)\}$.
\end{proof}

\begin{defn}
Let $M$ be a matroid on ground set $E$, and let $\sigma\in \mathfrak S_{E\setminus n}$.  Then define $k_\sigma(M)\in \mathbb{Z}$ by
\[
k_\sigma(M) = \begin{cases}
-1 & \text{if $n\in E$ is loop in $M$}\\
n & \text{if $n\in E$ is a coloop in $M$}\\
\max \{ 0 \leq i \leq n-1 \mid n\in B_{\sigma^{i}}(M)\} &\text{if $n\in E$ is neither a loop nor a coloop}.
\end{cases}
\]
\end{defn}

The following lemma records the key behavior of tautological Chern classes of matroids that we will need to establish a deletion-contraction relation (\Cref{thm:delcont}) in the next subsection.

\begin{lem}\label{lem:delcont}
Let $M$ be a matroid on ground set $E$, and $\sigma\in \mathfrak S_{E\setminus n}$. 
\begin{enumerate}[label = (\alph*)]
\item \label{delcont:classes}For any $0\le i\le n$ we have
\begin{align*}
c^T(\cS_M^\vee,u)_{\sigma^i}|_{t_n=0}=& \begin{cases}
c^{T'}(\cS_{M/n}^\vee, u)_\sigma&\text{if $i\leq k_\sigma(M)$}\\[2mm]
c^{T'}(\cS_{M\setminus n}^\vee, u)_\sigma &\text{if $i> k_\sigma(M)$,\quad and}
\end{cases}
\\
c^T(\cQ_M,u)_{\sigma^i}|_{t_n=0}=& \begin{cases}
c^{T'}(\cQ_{M/ n}, u)_\sigma &\text{if $i\leq k_\sigma(M)$}\\[2mm]
c^{T'}(\cQ_{M\setminus n}, u)_\sigma &\text{if $i> k_\sigma(M)$}.
\end{cases}
\end{align*}
\item  \label{delcont:factor} If $n\in E$ is neither a loop nor a coloop, then writing $k=k_{\sigma(M)}$, we have
\begin{align*}
(1+t_{\sigma(k)}u) c^{T'}(\cS^\vee_{M/ n}, u)_\sigma &= c^{T'}(\cS^\vee_{M\setminus n}, u)_\sigma \qquad\text{and}\\
(1-t_{\sigma(k)}u) c^{T'}(\cQ_{M\setminus n}, u)_\sigma &= c^{T'}(\cQ_{M/n},u)_\sigma.
\end{align*}
\end{enumerate}
\end{lem}

\begin{proof}
Recalling that for any permutation $\widetilde\sigma \in \mathfrak{S}_E$ we have
$$c^T(\mathcal{S}_M^{\vee},u)_{\widetilde\sigma}=\prod_{j\in B_{\widetilde\sigma}(M)}(1+t_ju)\quad\text{and}\quad c^T(\mathcal{Q}_M,u)_{\widetilde\sigma}=\prod_{j\not\in B_{\widetilde\sigma}(M)}(1-t_ju),$$
this is a direct reformulation of \Cref{lem:delcontcomb} after noting $(1\pm t_n u)|_{t_n=0}=1$.
\end{proof}

\subsection{A deletion-contraction relation}\label{subsec:delcont}
With these preparations in place, we are now ready to prove the deletion-contraction relation. Throughout, assume that $|E|\geq 2$.  As before, let $T' = (\CC^*)^{E\setminus n}$.
For notational clarity, let us define
$$a_E(x)=c(\mathcal{S}^{\vee}_{U_{n,E}},x)\qquad\text{and}\qquad b_E(y)=c(\mathcal{Q}_{U_{1,E}},y),$$
and their $T$-equivariant counterparts 
$a_E^T(x)=c^T(\mathcal{S}^{\vee}_{U_{n,E}},x)$ and $b_E^T(y)=c^T(\mathcal{Q}_{U_{1,E}},y)$.
In particular, for a matroid $M$ on ground set $E$, we have
\[
\xi_M=a_E(x)b_E(y)c(\mathcal{S}_M^{\vee},z)c(\mathcal{Q}_M,w).
\]
The following corollary of \Cref{lem:delcont} will be useful in our computations.

\begin{cor}\label{cor:abdelcont}
The $T$-equivariant classes $a_E^T(x)=c^T(\mathcal{S}^{\vee}_{U_{n,E}},x)$ and $b_E^T(y)=c^T(\mathcal{Q}_{U_{1,E}},y)$ satisfy
\begin{align*}
&a_E^T(x)_{\sigma^i}|_{t_n=0} = a_{E\setminus n}^{T'}(x)_{\sigma} \quad \text{for}\quad 0\leq i < n \quad\text{and}\quad a_E^T(x)_{\sigma^n}|_{t_n=0} = (1+t_{\sigma(n-1)}x)a_{E\setminus n}^{T'}(x)_\sigma, \text{ and}\\[2mm]
&b_E^T(y)_{\sigma^i}|_{t_n=0} = b_{E\setminus n}^{T'}(y)_{\sigma} \quad \text{for}\quad 0< i \leq n \quad \text{and}\quad b_E^T(y)_{\sigma^0}|_{t_n=0}= (1-t_{\sigma(0)}y)b_{E\setminus n}^{T'}(y)_\sigma
\end{align*}
for any permutation $\sigma\in \mathfrak S_{E\setminus n}$.
\end{cor}

\begin{proof}
For the part concerning $a_E^T$, apply \Cref{lem:delcont} to the matroid $U_{n,E}$, noting that $(U_{n,E})/n = U_{n-1, E\setminus n}$ and that $k_\sigma(U_{n,E}) = n-1$ for any $\sigma\in \mathfrak S_{E\setminus n}$.  Likewise, for $b_E^T$, apply \Cref{lem:delcont} to the matroid $U_{1,E}$, noting that $(U_{1,E})\setminus n = U_{1,E\setminus n}$ and that $k_\sigma(U_{1,E}) = 0$ for any $\sigma\in \mathfrak S_{E\setminus n}$.
\end{proof}

\begin{thm}\label{thm:delcont}
Let $M$ be a matroid on ground set $E$ with $|E|\geq 2$.  Let $f_*\colon  A^\bullet(X_E) \to A^{\bullet-1}(X_{E\setminus n})$ be the pushforward map of the toric map $f\colon  X_E \to X_{E\setminus n}$ in \Cref{defn:projmap}.
Then, we have
$$f_*\xi_M=\begin{cases}
(x+y)\xi_{M\setminus n}&\text{if $n\in E$ is a loop in $M$}\\
(x+y)\xi_{M/ n}&\text{if $n\in E$ is a coloop in $M$}\\
(x+w)\xi_{M\setminus n}+(y+z)\xi_{M/n}&\text{otherwise.}
\end{cases}$$
\end{thm}

\begin{proof}
We compute using the $T$-equivariant classes.  That is, denote by
\[
\xi_M^T = a^T_E(x) b^T_E(y) c^T(\mathcal{S}_M^{\vee},z)c^T(\mathcal{Q}_M,w)
\]
an element in $A_T^\bullet(X_E)[x,y,z,w]$, which maps to the non-equivariant class $\xi_M$.
We wish to show that for any permutation $\sigma \in \mathfrak S_{E\setminus n}$
\begin{equation}\label{eq:Tdelcont}\tag{\#}
(f_*\xi_M^T)_\sigma|_{t_n=0} = \begin{cases}
(x+y) (\xi^{T'}_{M\setminus n})_\sigma&\text{if $k_{\sigma(M)}=-1$}\\
(x+w) (\xi^{T'}_{M\setminus n} )_\sigma+(y+z) (\xi^{T'}_{M/n})_\sigma &\text{if $0\le k_{\sigma(M)}\le n-1$}\\
(x+y)(\xi_{M/ n}^{T'})_\sigma &\text{if $k_{\sigma(M)}=n$.}
\end{cases}
\end{equation}
Let us fix an arbitrary permutation $\sigma\in \mathfrak S_{E\setminus n}$.
By \Cref{lem:telescope}, we have
\[
(f_*\xi_M^T)_{\sigma}|_{t_n=0}=\sum_{i=0}^{n-1} t_{\sigma(i)}^{-1}\big( (\xi^T_M)_{\sigma^{i+1}}\bigr|_{t_n=0}-(\xi^T_M)_{\sigma^i}\bigr|_{t_n=0} \big) .
\]

Noting $k_{\sigma(M)}=-1$ if $n$ is a loop, $k_{\sigma(M)}=n$ if $n$ is a loop, and $0\le k_{\sigma(M)}\le n-1$ otherwise, the desired equality \eqref{eq:Tdelcont} is implied by the following claim consisting of three cases.

\medskip
\noindent\textbf{Claim:} For $i=0$ we have
\begin{align*}t_{\sigma(0)}^{-1}\big( (\xi^T_M)_{\sigma^{1}}\bigr|_{t_n=0}-(\xi^T_M)_{\sigma^0}\bigr|_{t_n=0} \big)=\begin{cases}
y (\xi_{M\setminus n}^{T'})_\sigma &\text{if $k_{\sigma}(M)=-1$}\\
w(\xi_{M\setminus n}^{T'})_\sigma + (y+z)(\xi_{M/ n}^{T'})_\sigma &\text{if $k_{\sigma}(M) = 0$}\\
y (\xi_{M/ n}^{T'})_\sigma &\text{if $k_{\sigma}(M)>0$}.\end{cases}\end{align*}
For $0<i<n-1$ we have
\begin{align*}
t_{\sigma(i)}^{-1}\big( (\xi^T_M)_{\sigma^{i+1}}\bigr|_{t_n=0}-(\xi^T_M)_{\sigma^i}\bigr|_{t_n=0} \big)=\begin{cases}0&\text{if $k_\sigma(M)\ne i$}\\w(\xi_{M\setminus n}^{T'})_\sigma + z(\xi_{M/ n}^{T'})_\sigma &\text{if $k_\sigma(M)=i$.}\end{cases}
\end{align*}
Finally, for $i=n-1$ we have
\begin{align*}
    t_{\sigma(n-1)}^{-1}\big( (\xi^T_M)_{\sigma^{n}}\bigr|_{t_n=0}-(\xi^T_M)_{\sigma^{n-1}}\bigr|_{t_n=0} \big)=\begin{cases}x (\xi_{M\setminus n}^{T'})_\sigma &\text{if $k_{\sigma}(M) < n-1$}\\(x+w)(\xi_{M\setminus n}^{T'})_\sigma + z(\xi_{M/ n}^{T'})_\sigma &\text{if $k_{\sigma}(M) = n-1$.}\\
    x (\xi_{M/n}^{T'})_\sigma &\text{if $k_{\sigma}(M) =n$}\end{cases}
\end{align*}
The proofs for $i = 0$ and $i = n-1$ are nearly identical so we only show the former. Also,  \Cref{lem:delcont}.\ref{delcont:classes} and \Cref{cor:abdelcont} together imply that the difference $(\xi_M^T)_{\sigma^{i+1}}|_{t_n=0}-(\xi_M^T)_{\sigma^i}|_{t_n=0}$ is zero when $0<i<n-1$ and $i = k_{\sigma}(M)$, so when $0<i<n-1$ we only need to establish the case $k_{\sigma(M)}=i$.

\medskip
\noindent\underline{Case $i = 0$ and $k_{\sigma}(M)\neq 0$.}

Write $M'=M\setminus n$ if $k_{\sigma(M)}=-1$ and $M'=M/n$ if $k_{\sigma(M)}>0$. Since $k_{\sigma}(M)\neq 0$, \Cref{lem:delcont}.\ref{delcont:classes} implies that
 \begin{align*}
c^T(\cS_M^\vee, z)_{\sigma^0}|_{t_n=0}&=  c^T(\cS_{M}^\vee,z)_{\sigma^1}|_{t_n=0}=c^{T'}(\cS_{M'}^{\vee},z)_\sigma, \text{ and} \\
c^T(\cQ_M, w)_{\sigma^0}|_{t_n=0}&=  c^T(\cQ_{M},w)_{\sigma^1}|_{t_n=0}=c^{T'}(\cQ_{M'},w)_\sigma.
 \end{align*}
By \Cref{cor:abdelcont}, we also have
\begin{align*}
a^T_E(x)_{\sigma^0}|_{t_n=0} &=  a^T_E(x)_{\sigma^1}|_{t_n=0}=a^{T'}_{E\setminus n}(x)_{\sigma},\\
b_E^T(y)_{\sigma^0}|_{t_n=0} &= (1- t_{\sigma(0)}y) b_{E\setminus n}^{T'}(y)_{\sigma}, \quad\text{and}\\
 b^T_E(y)_{\sigma^1}|_{t_n=0}&=b^{T'}_{E\setminus n}(y)_\sigma.
\end{align*}
We conclude that
\begin{align*}
&t_{\sigma(0)}^{-1}\big( (\xi^T_M)_{\sigma^{1}}\bigr|_{t_n=0}-(\xi^T_M)_{\sigma^0}\bigr|_{t_n=0} \big) \\
&= t_{\sigma(0)}^{-1} a^{T'}_{E\setminus n}(x)_\sigma \big(1-(1-t_{\sigma(0)}y)\big) b_{E\setminus n}^{T'}(y)_\sigma c^{T'}(\mathcal{S}_{M'},z)_\sigma c^{T'}(\mathcal{Q}_{M'},w)_\sigma\\& = y (\xi_{M'}^{T'})_\sigma.
\end{align*}

\medskip
\noindent\underline{Case $i = 0$ and $k_{\sigma}(M)= 0$.}

Since $k_{\sigma}(M) = 0$, \Cref{lem:delcont} implies that
\begin{align*}
& c^T(\cS_{M}^\vee,z)_{\sigma^{0}} |_{t_n=0}= c^{T'}(\cS_{M/n}^\vee,z)_\sigma, \\
&c^T(\cS_{M}^\vee,z)_{\sigma^1}|_{t_n=0}= c^{T'}(\cS_{M\setminus n}^\vee,z)_\sigma = (1 + t_{\sigma(0)}z) c^{T'}(\cS_{M/ n}^\vee,z)_\sigma,\\
&c^T(\cQ_{M},w)_{\sigma^{0}}|_{t_n=0}= c^{T'}(\cQ_{M/n},w)_\sigma = (1- t_{\sigma(0)}w) c^{T'}(\cQ_{M\setminus n},w)_\sigma, \quad\text{and}\\
&c^T(\cQ_M,w)_{\sigma^1}|_{t_n=0} = c^{T'}(\cQ_{M\setminus n},w)_\sigma.
\end{align*}
Similarly to the previous case, by \Cref{cor:abdelcont}, we also have
\begin{align*}
a^T_E(x)_{\sigma^0}|_{t_n=0} &=  a^T_E(x)_{\sigma^1}|_{t_n=0}=a^{T'}_{E\setminus n}(x)_{\sigma},\\
b_E^T(y)_{\sigma^0}|_{t_n=0} &= (1- t_{\sigma(0)}y) b_{E\setminus n}^{T'}(y)_{\sigma}, \quad\text{and}\\
 b^T_E(y)_{\sigma^1}|_{t_n=0}&=b^{T'}_{E\setminus n}(y)_\sigma.
\end{align*}
Thus, we conclude that
\begin{align*}
&t_{\sigma(0)}^{-1}\big( (\xi^T_M)_{\sigma^{1}}\bigr|_{t_n=0}-(\xi^T_M)_{\sigma^0}\bigr|_{t_n=0} \big)  \\
&=t_{\sigma(0)}^{-1}a^{T'}_{E\setminus n}(x)_\sigma\Big((1+t_{\sigma(0)}z)-(1-t_{\sigma(0)}y)(1-t_{\sigma(0)}w)\Big)b^{T'}_{E\setminus n}(y)_{\sigma}c^{T'}(\cS_{M/n}^\vee,z)_\sigma c^{T'}(\cQ_{M\setminus n},w)_\sigma\\
&= a^{T'}_{E\setminus n}(x)_\sigma b^{T'}_{E\setminus n}(y)_{\sigma}\Big( z + y + w - t_{\sigma(0)}yw \Big)c^{T'}(\cS_{M/n}^\vee,z)_\sigma c^{T'}(\cQ_{M\setminus n},w)_\sigma\\
&=a^{T'}_{E\setminus n}(x)_{\sigma}b^{T'}_{E\setminus n}(x)_{\sigma}\Big((w(1+t_{\sigma(0)}z)+(y+z)(1-t_{\sigma(0)}w)\Big)c^{T'}(\mathcal{S}_{M/n}^{\vee},z)_{\sigma}c^{T'}(\mathcal{Q}_{M\setminus n},w)_{\sigma}\\
&=a_{E\setminus n}^{T'}(x)_{\sigma}b_{E\setminus n}^{T'}(x)_{\sigma}\left(wc^{T'}(\mathcal{S}_{M\setminus n}^{\vee},z)_{\sigma}c^{T'}(\mathcal{Q}_{M\setminus n},z)_{\sigma}+(y+z)c^{T'}(\mathcal{S}_{M/ n}^{\vee},z)_{\sigma}c^{T'}(\mathcal{Q}_{M/n},z)_{\sigma}\right)\\
& = w(\xi_{M\setminus n}^{T'})_\sigma + (y+z)(\xi_{M/n}^{T'})_\sigma,
\end{align*}
where the second last equality follows from \Cref{lem:delcont}.\ref{delcont:factor}.

\medskip
\noindent\underline{Case $0 < i <n-1$ and $k_{\sigma}(M)=i$.}

Applying \Cref{lem:delcont} to $M$ with $i =  k_\sigma(M)$ implies that
\begin{align*}
& c^T(\cS_{M}^\vee,z)_{\sigma^{i}}|_{t_n=0} = c^{T'}(\cS_{M/n}^\vee,z)_\sigma, \quad\text{and}\\
&c^T(\cS_{M}^\vee,z)_{\sigma^{i+1}} |_{t_n=0}= c^{T'}(\cS_{M\setminus n}^\vee,z)_\sigma = (1 + t_{\sigma(i)}z) c^{T'}(\cS_{M/ n}^\vee,z)_\sigma,
\end{align*}
and moreover that
\begin{align*}
&c^T(\cQ_{M},w)_{\sigma^{i}}|_{t_n=0}= c^{T'}(\cQ_{M/n},w)_\sigma = (1- t_{\sigma(i)}w) c^{T'}(\cQ_{M\setminus n},w)_\sigma, \quad\text{and}\\
&c^T(\cQ_M,w)_{\sigma^{i+1}}|_{t_n=0}= c^{T'}(\cQ_{M\setminus n},w)_\sigma.
\end{align*}
Since $0 < i < n-1$, by \Cref{cor:abdelcont} we also have
\begin{align*}
&a_E^T(x)_{\sigma^i}|_{t_n=0} = a_E^T(x)_{\sigma^{i+1}}|_{t_n=0} = a_{E\setminus n}^{T'}(x)_\sigma , \quad\text{and}\\
&b_E^T(y)_{\sigma^i}|_{t_n=0}= b_E^T(y)_{\sigma^{i+1}}|_{t_n=0} = b_{E\setminus n}^{T'}(y)_\sigma .
\end{align*}
Thus, we conclude that
\begin{align*}
&t_{\sigma(i)}^{-1}\big( (\xi^T_M)_{\sigma^{i+1}}\bigr|_{t_n=0}-(\xi^T_M)_{\sigma^i}\bigr|_{t_n=0} \big) \\
&= t_{\sigma(i)}^{-1} a_{E\setminus n}^{T'}(x)_\sigma b_{E\setminus n}^{T'}(y)_\sigma \Big( (1+t_{\sigma(i)}z) - (1-t_{\sigma(i)}w) \Big) c^{T'}(\cS_{M/n}^\vee,z)_\sigma c^{T'}(\cQ_{M\setminus n},w)_\sigma\\
&= a_{E\setminus n}^{T'}(x)_\sigma b_{E\setminus n}^{T'}(y)_\sigma\Big( z + w \Big) c^{T'}(\cS_{M/n}^\vee,z)_\sigma c^{T'}(\cQ_{M\setminus n},w)_\sigma\\
&= a_{E\setminus n}^{T'}(x)_\sigma b_{E\setminus n}^{T'}(y)_\sigma\Big( w(1+ t_{\sigma(i)}z) + z(1-t_{\sigma(i)}w) \Big)  c^{T'}(\cS_{M/n}^\vee,z)_\sigma c^{T'}(\cQ_{M\setminus n},w)_\sigma\\
&= a_{E\setminus n}^{T'}(x)_\sigma b_{E\setminus n}^{T'}(y)_\sigma\Big( wc^{T'}(\cS_{M \setminus n}^\vee,z)_\sigma c^{T'}(\cQ_{M\setminus n},w)_\sigma + z   c^{T'}(\cS_{M/n}^\vee,z)_\sigma c^{T'}(\cQ_{M/ n},w)_\sigma\Big)\\
&= w(\xi^{T'}_{M\setminus n})_\sigma + z(\xi^{T'}_{M/ n})_\sigma,
\end{align*}
where the second last equality follows from \Cref{lem:delcont}.\ref{delcont:factor}.

This concludes our proof of the claim, and thereby that of \Cref{thm:delcont}.
\end{proof}
From this, we conclude the proof of \Cref{thm:4degintro}.
\begin{proof}[Proof of \Cref{thm:4degintro}]
\Cref{thm:delcont} shows that the pushforward $f_*\xi_M$ satisfies the same deletion-contraction relation as  $t_M(x,y,z,w)$ does. Noting that $ \int_{X_E} \xi_M=\int_{X_{E\setminus n}} f_* \xi_M$ by the functoriality of pushforward maps, we conclude \Cref{thm:4degintro} by induction on the cardinality of $E$.
\end{proof}

\section{Base polytope properties}\label{sec:basepolytopeproperties}

We establish the base polytope properties of tautological classes of matroids and their Chern classes listed in \S\ref{sec:fundprops}\ref{fundprop:base}---matroid minors decomposition, valuativity, and well-behavedness under duality and direct sum.

\subsection{Matroid minors decomposition}
For a matroid $M$ on $E$ and subset $S\subseteq E$, recall that the \textbf{restriction} $M|S$ is a matroid on $S$ with rank function $\operatorname{rk}_{M|S}(\cdot) = \operatorname{rk}_M(\cdot)$, and that the \textbf{contraction} $M/S$ is a matroid on $E\setminus S$ with rank function $\operatorname{rk}_{M/S}(\cdot) = \operatorname{rk}_M(\cdot \cup S) - \operatorname{rk}_M(S)$.
A \textbf{minor} of $M$ is a matroid $M|S/S'$ on $S\setminus S'$ for some $S' \subseteq S \subseteq E$.  One can verify that $M|S/S' = (M/S')|S$.

\medskip
Let $\mathscr S\colon  \emptyset \subsetneq S_1 \subsetneq \cdots \subsetneq S_k \subsetneq E$ be a chain of nonempty proper subsets of $E$.  We always denote by convention $S_0 = \emptyset$ and $S_{k+1} = E$ for such a chain.  Faces of the base polytope of a matroid have the following decomposition property. 

\begin{prop}\label{prop:greedy}\cite[Proposition 2]{AK06}
Let $M$ be a matroid with ground set $E$.
The face of the base polytope $P(M)$ maximizing the pairing $\langle \cdot, \be_{S_1} + \cdots + \be_{S_k}\rangle$ is equal to the product
\[P(M|{S_1})\times P(M|{S_2}/S_1)\times \ldots \times P(M|{S_k}/S_{k-1})\times P(M/S_k).
\]
\end{prop}

Many combinatorial invariants on matroids respect this matroid minors decomposition behavior, which underlies the Hopf algebra structure on matroids studied in \cite{JR79, Sch87, AA,KRAJEWSKI2018271,Coalgebra}.
We show that the tautological classes of matroids also display such behavior.  We prepare with the following, which is a geometric restatement of the fact that faces of permutohedra are products of smaller permutohedra (see for example \cite[\S4.1]{AA} and references therein).

\begin{prop}
\label{prop:torusrest}
Let $\mathscr S\colon  \emptyset \subsetneq S_1 \subsetneq \cdots \subsetneq S_k \subsetneq E$ be a chain of nonempty proper subsets of $E$.  Then, the torus-orbit closure $Z_{\mathscr S} \subset X_E$ corresponding to $\operatorname{Cone}(\overline \be_{S_1}, \ldots, \overline\be_{S_k}) \in \Sigma_E$ has a natural $T$-equivariant isomorphism
\[
Z_{\mathscr S} \simeq X_{S_1} \times X_{S_2\setminus S_1} \times \cdots \times X_{E\setminus S_k},
\]
where the torus $T = (\CC^*)^E$ acts on the torus $T_{S_{i+1}\setminus S_{i}} = (\CC^*)^{S_{i+1}\setminus S_i}$ via the obvious projection for each $i = 0, \ldots, k$.  In particular, we have
\[
A^\bullet_T(Z_{\mathscr S}) \simeq \bigotimes_{i=0}^k A^\bullet_{T_{S_{i+1}\setminus S_{i}}}(X_{{S_{i+1}\setminus S_{i}}}) \quad\text{and}\quad K_T^0(Z_{\mathscr S}) \simeq \bigotimes_{i=0}^{k} K_{T_{S_{i+1}\setminus S_{i}}}^0(X_{{S_{i+1}\setminus S_{i}}}).
\]
Under the isomorphism, for a $(k+1)$-tuple $(\sigma_0, \ldots, \sigma_k)$ of permutations where $\sigma_i \in \mathfrak S_{S_{i+1}\setminus S_i}$, 
the inclusion $Z_{\mathscr S} \hookrightarrow X_E$ maps the $T$-fixed point $p_{\sigma_0} \times \cdots \times p_{\sigma_k}$ of $Z_{\mathscr S}$ to the point $p_{\sigma}$  of $X_E$ where $\sigma$ is the permutation $\sigma_0 \circ \cdots \circ \sigma_k$ on $E = \bigsqcup_{i=0}^k (S_{i+1}\setminus S_i)$.
\end{prop}

\begin{prop}
\label{prop:tautrest}
Let $\mathscr{S}\colon  \emptyset \subsetneq S_1\subsetneq \cdots \subsetneq S_k\subsetneq E$ be a chain of nonempty proper subsets of $E$, and let $M$ be a matroid on $E$. Then, under the isomorphism in \Cref{prop:torusrest}, one has
\[
[\mathcal{S}_M]|_{Z_{\mathscr S}}=\sum_{i=0}^{k}  1^{\otimes i}\otimes [\mathcal{S}_{M|S_{i+1}/S_{i}}]\otimes 1^{\otimes(k-i)} \quad\text{and}\quad [\mathcal{Q}_M]|_{Z_{\mathscr S}}=\sum_{i=0}^{k} 1^{\otimes i}\otimes [\mathcal{Q}_{M|S_{i+1}/S_{i}}]\otimes 1^{\otimes(k-i)}
\]
as elements in $K_T^0(Z_{\mathscr S})$. In particular, with a formal variable $u$, one has
\[
c(\mathcal{S}_M,u)|_{Z_{\mathscr S}}=\bigotimes_{i=0}^k c(\mathcal{S}_{M|S_{i+1}/S_{i}},u) \quad\text{and}\quad c(\mathcal{Q}_M,u)|_{Z_{\mathscr S}}=\bigotimes_{i=0}^k c(\mathcal{Q}_{M|S_{i+1}/S_{i}},u)
\]
as elements in $\bigotimes_{i=0}^k A^\bullet(X_{S_{i+1}\setminus S_{i}}) \simeq A^\bullet(Z_{\mathscr S})$.
\end{prop}

\begin{proof}
Let $\sigma\in \mathfrak S_E$ be the composition of permutations $\sigma_0, \ldots, \sigma_k$ where $\sigma_i \in \mathfrak S_{S_{i+1}\setminus S_i}$ for $i = 0, \ldots, k$.  Since $B_{\sigma}(M) = \bigsqcup_{i=0}^k B_{\sigma_i}(M|S_{i+1}/S_{i})$ by \Cref{prop:greedy}, the restrictions of the $K$-classes $[\mathcal{S}_M]|_{Z_{\mathscr S}}$ and $\sum_{i=0}^{k}  1^{\otimes i}\otimes [\mathcal{S}_{M|S_{i+1}/S_{i}}]\otimes 1^{\otimes(k-i)}$ to the point $p_\sigma$ both give the same Laurent polynomial $\sum_{j\in B_\sigma(M)} T_j^{-1}$.  The statement for $[\cQ_M]$ is proved similarly.
\end{proof}

\begin{cor}\label{cor:alphabetarest}
Let $\mathscr{S}\colon  \emptyset \subsetneq S_1\subsetneq \cdots \subsetneq S_k\subsetneq E$ be a chain of nonempty proper subsets of $E$, and let $M$ be a matroid on $E$.  The divisor classes $\alpha_E$ and $\beta_E$ on $X_E$ defined in \Cref{defn:alphabeta} satisfy
\[
\alpha|_{Z_{\mathscr S}} = 1^{\otimes k} \otimes \alpha_{E\setminus S_k} \quad\text{and}\quad \beta|_{Z_{\mathscr S}} = \beta_{S_1} \otimes 1^{\otimes k}
\]
as elements in  $\bigotimes_{i=0}^k A^\bullet(X_{S_{i+1}\setminus S_{i}}) \simeq A^\bullet(Z_{\mathscr S})$.
\end{cor}

\begin{proof}
 In \Cref{eg:alphabeta}, we showed that the divisor classes $\alpha_E$ and $\beta_E$ on $X_E$ satisfy $[\cQ_{U_{n,E}}] = [\mathcal O(\alpha)]$ and $[\cS_{U_{1,E}}^\vee] = [\mathcal O(\beta)]$ as non-equivariant $K$-classes.  For any subsets $\emptyset \subseteq S'\subseteq S \subseteq E$, the matroid minor $U_{n,E}|S/S'$ is a uniform matroid of corank 1 if $S=E$ and $S'\subsetneq E$, and is corank 0 otherwise.  Likewise, the matroid minor $U_{1,E}|S/S'$ is a uniform matroid of rank 1 if $\emptyset\subsetneq S$ and $S' = \emptyset$, and is rank 0 otherwise.  Thus, applying \Cref{prop:tautrest} with $M= U_{n,E}$ and $M = U_{1,E}$ implies the desired result.
\end{proof}

\subsection{Valuativity}
For a subset $P \subseteq \RR^E$, let $1_P\colon  \RR^E \to \ZZ$ denote its indicator function defined by $1_P(x) = 1$ if $x\in P$ and 0 otherwise.
An important tool for extending algebraic constructions from realizable matroids to arbitrary matroids is the following notion of ``valuativity''.

\begin{defn}
\cite[Definition~3.5]{derksen2010valuative}
A function $\phi$ from the set of matroids with ground set $E$ to an abelian group $A$ is \textbf{valuative} if, for any matroids $M_1, \ldots, M_\ell$ on $E$ and integers $a_1, \ldots, a_\ell$ such that $\sum_{i=1}^\ell a_i 1_{P(M_i)} = 0$, the function $\phi$ satisfies $\sum_{i=1}^\ell a_i \phi(M_i) = 0$.
\end{defn}

In other words, the map $\phi$ is valuative if it factors through the map sending $M$ to the indicator function $1_{P(M)}$ of its base polytope $P(M)$.  Many invariants of matroids satisfy valuativity, including the Tutte polynomial and its specializations \cite{ardila_fink_rincon_2010, derksen2010valuative}. For a more comprehensive list see \cite{ardila2020valuations}, and see \cite{ESS21} for a study of valuativity for Coxeter matroids.

\smallskip
We show that a wide range of classes associated to the tautological $K$-classes of matroids are also valuative.
These will include any polynomial expression in exterior powers or Chern classes of $[\mathcal{S}_M]$, $[\mathcal{Q}_M]$, and their dual $K$-classes $[\cS_M^\vee]$ and $[\cQ_M^\vee]$.  For example, one may consider assigning to a matroid $M$ the class $[\bigwedge^2 \cS_M][\bigwedge^3 \cQ_M] +4 [\bigwedge^5 \cS_M^\vee]$, or the class $c_1(\cS_M)^2c_{2}(\cQ_M) - c_4(\cS_M)^3$.

More precisely, recall from \Cref{subsec:prelimChernRoots} that for an element $\lambda(x) \in \Lambda\subset  \ZZ[[x_1,x_2,\dots]]$ in the ring of symmetric functions,
 we may construct the classes $[\mathsf S^{\lambda}\mathcal E] \in K_T(X_E)$ and $\mathsf s_{\lambda}^T(\mathcal E) \in A^\bullet_T(X_E)$ when $\mathcal E = [\cS_M]$, $[\cQ_M]$, $[\cS_M^\vee]$, or $[\cQ_M^\vee]$.
For a polynomial $f(a,b,c,d)\in \ZZ[a,b,c,d]$ and a sequence $\boldsymbol \lambda = (\lambda_a(x),\lambda_b(x), \lambda_c(x), \lambda_d(x))$ of symmetric functions, we define
\begin{align*}
\phi_{f,\boldsymbol \lambda}\colon  \{\text{Matroids on $E$}\} \to K_T(X_E) &\quad\text{by}\quad M \mapsto f([\mathsf S^{\lambda_a}\cS_M], [\mathsf S^{\lambda_b}\cQ_M], [\mathsf S^{\lambda_c}\cS_M^\vee], \mathsf [S^{\lambda_d}\cQ_M^\vee]), \quad\text{and}\\
\psi_{f,\boldsymbol \lambda}\colon  \{\text{Matroids on $E$}\} \to A^\bullet_T(X_E) &\quad\text{by}\quad M \mapsto f(\mathsf s_{\lambda_a}^T(\cS_M), \mathsf s_{\lambda_b}^T(\cQ_M),\mathsf s_{\lambda_c}^T(\cS_M^\vee),\mathsf s_{\lambda_d}^T(\cQ_M^\vee)).
\end{align*}
For instance, by taking $\boldsymbol \lambda$ to be appropriate elementary symmetric functions and $f$ certain polynomials, one obtains the two aforementioned examples.

\begin{prop}
\label{prop:valpolys}
The maps $\phi_{f,\boldsymbol\lambda}$ and $\psi_{f,\boldsymbol\lambda}$ defined above are valuative.
\end{prop}

We prepare the proof with the following lemma.

\begin{lem}
\label{lem:Bsigmaval}
Let $\ZZ^{\binom{E}{r}}$ be the free abelian group on the set of $r$-subsets of $E$, with its standard basis denoted $\{\langle B \rangle \mid B\in \binom{E}{r}\}$.  For any fixed permutation $\sigma \in \mathfrak S_E$, the assignment $M \mapsto \langle B_{\sigma}(M) \rangle\in \mathbb{Z}^{\binom{E}{r}}$ on matroids of rank $r$ with ground set $E$ is valuative.
\end{lem}
\begin{proof}
For a total order $<$ on the ground set $E$, recall that an element $i$ in a basis $B$ is internally active if there is no $j\in E$ such that $j<i$ and $(B\setminus i)\cup j$ is a basis.
When the total order $<$ is given by the permutation $\sigma$, the lex-first basis $B_\sigma(M)$ is the unique basis with $r$ internally active elements (\Cref{rem:lexfirstbasis}).
The lemma is now a special case of \cite[Theorem 5.4]{ardila_fink_rincon_2010}.
\end{proof}

\begin{proof}[Proof of \Cref{prop:valpolys}]
For a matroid $M$ and a permutation $\sigma$, by the construction of $[\mathsf S^{\lambda}\mathcal E]$ in \Cref{subsec:prelimChernRoots}, the restriction of the equivariant $K$-class $\phi_{f,\boldsymbol \lambda}(M)$ to the $T$-fixed point $p_\sigma$ is a Laurent polynomial $\phi_{f,\boldsymbol \lambda}(M)_\sigma \in \ZZ[T_0^\pm, \ldots, T_n^\pm]$ determined completely by $B_\sigma(M)$.  Similarly,  the polynomial $\psi_{f,\boldsymbol \lambda}(M)_\sigma$ in $\ZZ[t_0, \ldots, t_n]$ representing the restriction to $p_\sigma$ of the equivariant Chow class $\psi_{f,\boldsymbol \lambda}(M)$ is completely determined by $B_\sigma(M)$.  
In other words, the map $M\mapsto \phi_{f,\boldsymbol \lambda}(M)_\sigma$ factors through the free abelian group $\ZZ^{\binom{E}{r}}$ by $M \mapsto \langle B_\sigma(M) \rangle \in \ZZ^{\binom{E}{r}}$, and similarly for $\psi_{f,\boldsymbol \lambda}$.
Thus, the method of localization \Cref{thm:localization} implies that both $\phi_{f,\boldsymbol \lambda}$ and $\psi_{f,\boldsymbol \lambda}$ factor through the map $M\mapsto ( \langle B_\sigma(M) \rangle)_{\sigma\in \mathfrak S_E} \in \bigoplus_{\sigma\in \mathfrak S_E} \ZZ^{\binom{E}{r}}$, which is valuative by \Cref{lem:Bsigmaval}, and the result follows.
\end{proof}

We note for future use in \S\ref{sec:KtoChow} the following generalization of \Cref{prop:valpolys} concerning the valuativity of classes defined in \Cref{prop:welldefinedgeneral}.

\begin{prop}\label{prop:valpolysgeneral}
For a fixed set of classes $[\mathcal E^{(1)}], \ldots, [\mathcal E^{(m)}] \in K_T^0(\Gr(r;E))$, the map $\phi$ that assigns to a matroid $M$ of rank $r$ on $E$ a fixed polynomial expression in the classes $[\mathcal E^{(1)}_M], \ldots, [\mathcal E^{(m)}_M] \in K_T^0(X_E)$ defined in \Cref{prop:welldefinedgeneral} is valuative.
\end{prop}

\begin{proof}
It follows from the definition of the classes $[\mathcal E^{(i)}_M]$ (\Cref{prop:welldefinedgeneral}) that, for each permutation $\sigma\in \mathfrak S_E$, the restriction of the assignment $\phi(M)$ to the $T$-fixed point $p_\sigma \in X_E^T$ is completely determined by $B_{\sigma}(M)$.  Hence, the map $\phi$ factors through $M\mapsto ( \langle B_\sigma(M) \rangle)_{\sigma\in \mathfrak S_E} \in \bigoplus_{\sigma\in \mathfrak S_E} \ZZ^{\binom{E}{r}}$. Thus, the assignment is valuative by \Cref{lem:Bsigmaval}.
\end{proof}

We conclude with a lemma that will be useful later for deducing results for arbitrary matroids from realizable matroids.  Recall that an element $e\in E$ is a {loop} (resp.\ {coloop}) in a  matroid $M$ if no basis of $M$ contains $e$ (resp.\ every basis of $M$ contains $e$).

\begin{lem}
\label{lem:looplessval}
Let $M$ be a matroid on $E$.  Then there exist matroids $M_1, \ldots, M_\ell$, all realizable over $\CC$, and integers $a_1, \ldots, a_\ell$, such that $1_{P(M)} = \sum_{i=1}^\ell a_i 1_{P(M_i)}$.  If $M$ is loopless (resp.\ coloopless), then all $M_i$'s can also be taken to be loopless (resp.\ coloopless).
\end{lem}

\begin{proof}
\cite[Theorem 5.4]{derksen2010valuative} states that the $\ZZ$-span of the indicator functions of base polytopes of matroids admits a basis consisting of matroids known as Schubert matroids, which are all realizable over $\CC$.  In particular, one can write $1_{P(M)} = \sum a_i 1_{P(M_i)}$ with $M_i$ being Schubert matroids.

Suppose now that $M$ is loopless; the coloopless case is proved similarly.
We first claim that if we have polytopes $P_1,\ldots,P_\ell$ and integers $b_i$ such that $\sum b_i1_{P_i}=0$, and all $P_i$ lie in a half-space $H^+$ bounded by a hyperplane $H$, then the sum is also zero if we restrict to just the polytopes completely contained in $H$.  Assuming the claim, for $j\in E$ take the hyperplane $H_j=\{x_j=0\}$ and half-space $H_j^+=\{x_j\ge 0\}$ in $\RR^E$. For all $i$ we have $P(M_i)\subset H_j^+$, and we have $P(M_i)\subset H_j$ if and only if $j$ is a loop in $M_i$. Therefore, by considering each $j\in E$ in turn, we can remove all polytopes with loops from the sum $\sum a_i 1_{P(M_i)}$ without affecting the equation $1_{P(M)} = \sum a_i 1_{P(M_i)}$, and the result follows.

For the proof of the claim, we note that the set of all polytopes contained in $H^+$, which we denote as $\mathscr P$, is intersection closed.  \cite[Theorem 1]{Gro78} states that a function $f$ on such a set $\mathscr P$ of polytopes factors through the map $P\mapsto 1_P$ if it satisfies $f(\emptyset) = 0$ and $f(P\cup Q) + f(P\cap Q) = f(P) + f(Q)$ for any  $P,Q\in \mathscr P$ with $P\cup Q\in \mathscr P$.
We verify that the function $h$ mapping $P\in \mathscr P$ to $1_P$ if $P \subset H$ and $0$ otherwise satisfies this relation as follows.
If $P\subseteq Q$ or $Q\subseteq P$, or if $P\cap Q \not\subset H$, or if $P\cup Q\subset H$, then the relation is easily verified for $h$.
But the condition for $P$ and $Q$ leaves no other cases: Suppose otherwise. Then $(P\cup Q)\setminus H$ is convex and nonempty.  The relatively closed decomposition $(P\cup Q)\setminus (P\cap Q)=(P\setminus Q)\cup (Q\setminus P)$ disconnects $(P\cup Q)\setminus (P\cap Q)$, so the connectedness of $(P \cup Q)\setminus H$ implies that $P\setminus Q\subset H$ or $Q\setminus P\subset H$. But then $P\subset H$ or $Q\subset H$, so the convexity of $P\cup Q$ together with $P\cup Q\not\subset H$ necessitates that $P\subseteq Q$ or $Q\subseteq P$, in contradiction to the assumption.
\end{proof}

\begin{rem}
\Cref{lem:looplessval} gives rise to the following recurring theme for matroid constructions motivated from geometry.  Let $f\colon  \Gr(r;E)\to A$ be a function from the Grassmannian, i.e.~the space of realizations of matroids of rank $r$, to an abelian group $A$.  Suppose $f$ satisfies the property that
\begin{itemize}
    \item[(i)] the value $f(L)$ only depends on the matroid $M$ that $L\in \Gr(r;E)$ realizes.
\end{itemize}
If the function $f$, now considered as a function on the set of realizable matroids, extends to a function $\widetilde f$ on the set of all (not necessarily realizable) matroids satisfying
\begin{itemize}
    \item[(ii)] the assignment $M\mapsto \widetilde f(M)$ is valuative,
\end{itemize}
then \Cref{lem:looplessval} implies that such an extension is unique.
Many matroid constructions motivated from the geometry of realizations of matroids are characterized by the two properties (i) and (ii) above.  These constructions include: tautological $K$ and Chern classes (Propositions~\ref{prop:tautdependsonmatroid} \& \ref{prop:valpolys}), Bergman classes (\Cref{cor:Bergman}), Chern-Schwartz-MacPherson classes (\Cref{cor:CSM}), the combinatorial biprojective classes defined in \Cref{defn:projClass} (\Cref{prop:projClassDefn}), the $K$-class denoted $y(M)$ for a matroid $M$ in \cite{FS12} (\Cref{defn:yM}), and the assignment of $[\mathcal{E}]\mapsto [\mathcal{E}_M]$ in \Cref{prop:welldefinedgeneral} (\Cref{prop:valpolysgeneral}).
\end{rem}

\subsection{Matroid duality and direct sum}
\label{sec:Duality}
For a matroid $M$ of rank $r$ with ground set $E$, the \textbf{dual matroid} $M^\perp$ is the matroid of rank $|E|-r$ with ground set $E$ whose bases are $\{E\setminus B \mid B \text{ a basis of }M\}$.
If $L\subseteq \CC^E$ is a realization of $M$, then $L^\perp = (\CC^E/L)^\vee \subseteq (\CC^E)^\vee \simeq \CC^E$ is a realization of $M^\perp$, where the isomorphism $(\CC^E)^\vee \simeq \CC^E$ is induced by the  standard basis of $\CC^E$.
Recall from \Cref{sec:CremonaInv} the Cremona involution $\crem\colon  X_E \to X_E$, induced by the map of tori $T/\CC^* \to T/\CC^*$ mapping $[t_0: \cdots : t_n]$ to $[t_0^{-1}: \cdots : t_n^{-1}]$.

\begin{prop}\label{prop:duality}
Let $M$ be a matroid on $E$.  Then, one has
$\crem[\mathcal{S}_M]=[\mathcal{Q}_{M^{\perp}}^{\vee}]$ and $\crem[\mathcal{Q}_M]=[\mathcal{S}_{M^{\perp}}^{\vee}]$. In particular, $\crem c(\mathcal{S}_M,u)=c(\mathcal{Q}_{M^{\perp}}^{\vee},u)$ and $\crem c(\mathcal{Q}_M,u)=c(\mathcal{S}_{M^{\perp}}^{\vee},u)$.
\end{prop}

\begin{proof}
For a permutation $\sigma\in \mathfrak S_E$, by \Cref{rem:crem} we have $(\crem [\cS_M])_\sigma = \sum_{i\in B_{\overline\sigma}(M)} T_i$.  Because $P(M^\perp) = -P(M) + (1,1,\ldots, 1)$, for any $v_0>v_1>\ldots > v_n$, the basis $E\setminus B$ of $M^\perp$ which maximizes the pairing $\langle \be_{E\setminus B},v_0\be_{\sigma(0)}+\ldots + v_n\be_{\sigma(n)}\rangle$ corresponds to the basis $B$ of $M$ which minimizes $\langle \be_{B},v_0\be_{\sigma(0)}+\ldots + v_n\be_{\sigma(n)}\rangle$. But this was earlier shown to be a defining property of the reverse-lex-basis $B=B_{\overline{\sigma}}$. Hence $[\cQ_{M^\perp}^\vee]_\sigma = \sum_{i\in B_{\overline\sigma}(M)}T_i$ as well.  The proof for $\crem[\cQ_M] = [\cS_{M^{\perp}}^\vee]$ is similar. 
\end{proof}

Let $E = E_1 \sqcup E_2$ be a disjoint union of nonempty subsets $E_1$ and $E_2$ of $E$.  For matroids $M_1$ and $M_2$ on $E_1$ and $E_2$, respectively, their \textbf{direct sum} $M_1 \oplus M_2$ is a matroid whose base polytope is given by $P(M_1 \oplus M_2) = P(M_1) \times P(M_2)$, where we have identified $\RR^E = \RR^{E_1} \times \RR^{E_2}$.  A matroid $M$ is \textbf{connected} if it is not a direct sum of two matroids on nonempty ground sets, and is \textbf{disconnected} otherwise.  

\begin{prop}\label{prop:directsum1}
Let $M$ be a matroid on $E$.  Then $M$ can be uniquely written as a direct sum $M_1 \oplus \cdots \oplus M_k$ of connected matroids $M_i$ on $E_i$ satisfying $E = \bigsqcup_{i=1}^k E_i$, called the connected components of $M$.  The number $k$ of connected components of $M$ satisfies $\dim P(M) = |E|-k$.
\end{prop}
\begin{proof}
See \cite[4.2.9]{Oxl11} and \cite[Proposition 2.4]{FS05}.
\end{proof}

Let us establish an analogous direct sum behavior for the tautological $K$-classes of matroids.
We prepare by noting that, for a nonempty subset $E'\subseteq E$, the projection map $\RR^E/\RR\mathbf 1\to \RR^{E'}/\RR\mathbf 1$ induces a map of fans $\Sigma_E\to \Sigma_{E'}$.
If we write $E' = \{j_0, \ldots, j_{n'}\} \subseteq E$ with $j_0 < \cdots < j_{n'}$, then the cone in $\Sigma_E$ corresponding to a permutation $\sigma\in \mathfrak S_E$ maps to the cone in $\Sigma_{E'}$ corresponding to the permutation $\sigma_{E'} \in \mathfrak S_{E'}$ defined by
\[
\sigma_{E'}(j_k) = \sigma(i_k) \text{ where $0\leq i_0 < \cdots < i_{n'} \leq n$ such that $\{\sigma(i_0), \ldots, \sigma(i_{n'})\} = E'$}.
\]
In particular, considering $X_{E'}$ as a $T$-variety via the projection map $T = (\CC^*)^E \to (\CC^*)^{E'} = T'$, we have a $T$-equivariant map of $T$-varieties $X_E \to X_{E'}$, and under which the point $p_\sigma\in X_E^T$ corresponding to $\sigma\in \mathfrak S_E$ maps to the point $p_{\sigma_{E'}}\in X_{E'}^T$.

\begin{prop}\label{prop:directsum2}
Let $M$ be a matroid on $E$ with connected components $M_1, \ldots, M_k$ on $E_1, \ldots, E_k \subseteq E$.
For each $i = 1, \ldots, k$, let $f_i\colon  X_E \to X_{E_i}$ be the toric morphism induced by the map of fans $\Sigma_E \to \Sigma_{E_i}$ arising from the projection $\RR^E/\RR\mathbf 1 \to \RR^{E_i}/\RR\mathbf 1$.  Then, we have
\[
[\cS_M] = \sum_{i=1}^k f_i^*[\cS_{M_i}] \quad\text{and}\quad [\cQ_M] = \sum_{i=1}^k f_i^*[\cQ_{M_i}] \quad\text{ as elements in } K_T^0(X_E).
\]
\end{prop}

\begin{proof}
Since $M = M_1 \oplus \cdots \oplus M_k$, any basis $B$ of $M$ can be uniquely written $B = B_1 \sqcup \cdots \sqcup B_k$ where $B_i$ is a basis of $M_i$ for each $i = 1, \ldots, k$.  Thus, from the description of the permutation $\sigma_{E_i}$ above for each $i = 1, \ldots, k$, it follows that 
$B_\sigma(M) = \bigsqcup_{i=1}^k B_{\sigma_{E_i}}(M_i)$.
We also have $(f_i^*[\cS_{M_i}])_\sigma = [\cS_{M_i}]_{\sigma_{E_i}}$ because $f_i$ maps the point $p_\sigma \in X_E^T$ to $p_{\sigma_{E_i}}\in X_{E_i}^T$, and the desired equality for $[\cS_M]$ follows.  The proof for $[\cQ_M]$ is similar.
\end{proof}

\section{Beta invariants via tautological classes}
\label{sec:betavia}
In this section, we record an immediate corollary of \Cref{thm:4degintro} as \Cref{thm:betainvar}, which is an expression for the beta-invariants of matroids in terms of the tautological Chern classes of matroids.  We will use this specialization of \Cref{thm:4degintro} in the next two sections to study certain Chow classes associated to matroids.

\begin{defn}
For a matroid $M$, denote by $\beta(M)$ and $\beta(M^\perp)$ the coefficients of the linear terms $x$ and $y$, respectively, in the Tutte polynomial $T_M(x,y)$.  The quantity $\beta(M)$ is called the \textbf{beta invariant} of the matroid $M$.
\end{defn}

The notation $\beta(M)$ and $\beta(M^\perp)$ is consistent with matroid duality since $T_M(x,y) = T_{M^\perp}(y,x)$.  Beta invariants of matroids were introduced and studied in \cite{crapoBeta}.  We express beta invariants in terms of tautological Chern classes of matroids as follows.

\begin{thm}\label{thm:betainvar}
Let $M$ be a matroid of rank $r$ on ground set $E$.  Then,
\[
\int_{X_E} c_{r-1}(\cS_M^\vee)c_{|E|-r}(\cQ_M) = \beta(M) \quad\text{and}\quad \int_{X_E} c_{r}(\cS_M^\vee)c_{|E|-r-1}(\cQ_M) = \beta(M^\perp),
\]
where we set by convention $c_{-1}(\mathcal E) = 0$ for a $K$-class $[\mathcal E]$.
\end{thm}

\begin{proof}
Note that Tutte polynomials have no constant terms, and that $c_k(\mathcal{S}_M^{\vee})=0$ for $k>r$ and $c_\ell(\mathcal{Q}_M)=0$ for $\ell>|E|-r$.
Thus, by setting $x = y = 0$ in \Cref{thm:4degintro}, we find
\[
\sum_{k+ \ell = n} \left(\int_{X_E} c_k(\cS_M^\vee) c_\ell(\cQ_M)\right) z^kw^\ell = z^{r}w^{|E|-r} \Big(\beta(M) \frac{1}{z}+{\beta}(M^\perp)\frac{1}{w}\Big),
 \]
 so we conclude the desired equalities.
\end{proof}

\begin{rem}\label{rem:beta}
If $M$ has a realization $L$, there are a number of known geometric manifestations of the beta invariant. \cite[Theorem 5.1]{Spe09} showed that $\beta(M)$ equals the number of points in the intersection of $\overline{T\cdot L}\subset \Gr(r;E)$ with the Schubert subvariety of $\Gr(r;E)$ corresponding to $r$-dimensional subspaces that form a flag with a generic point-hyperplane pair.  For an alternate proof of \Cref{thm:betainvar} using this geometry and valuativity (\Cref{prop:valpolys}), see the end of \Cref{proof:betainvarAlt}.

In the discussion at the end of \cite[\S5]{Spe09}, it is noted that the beta invariant arose in \cite[Theorem 28]{CHKS06} as the top Chern class of the log-tangent sheaf on a certain compactification of the projectivized hyperplane arrangement complement $\mathbb{P}(L)\setminus \bigcup \mathbb{P}(\mathcal{H}_i)$ where $\mathcal{H}_i=\{x_i=0\}\cap L$, which also appeared in \cite[\S2.2]{HKT06}.  The top Chern class of the log-tangent sheaf is known by the logarithmic Poincar\'e-Hopf theorem to compute the Euler characteristic of $\mathbb{P}(L)\setminus \bigcup \mathbb{P}(\mathcal{H}_i)$, which is $(-1)^{\dim \PP(L)}\beta(M)$ by \cite[Theorem 5.2]{OS80}.  
These computations were motivated by a problem that involves counting in a very-affine linear subvariety the number of critical points of a product of powers of linear forms, first considered by Varchenko over the real numbers \cite{Var95}, and established over the complex numbers in \cite{OT95}.
A generalization of these results on very-affine linear subvarieties to those on arbitrary very-affine smooth subvarieties is \cite{huh_2013}.
For us, this will manifest in the geometric constructions in \Cref{thm:BergmanGeom} and \Cref{thm:CSMGeom}, where we show that $c_{|E|-r}(\mathcal{Q}_M)$ is the Chow class of the ``wonderful compactification'' $W_L$ of $\mathbb{P}(L)\setminus \bigcup \mathbb{P}(\mathcal{H}_i)$, and $\mathcal{S}_L|_{W_L}$ is a trivial extension of the log-tangent sheaf of the compactification $W_L \supset (\mathbb{P}(L)\setminus \bigcup \mathbb{P}(\mathcal{H}_i))$.  See \Cref{rem:betainvarByLog}.

Lastly, in an attempt to formulate a tropical Poincar\'e-Hopf formula, Rau in \cite{Rau20} showed that a certain tropical intersection also computes $\beta(M)$.  We give a geometric interpretation of this in \Cref{sec:troplogPoinc}.
\end{rem}

\section{Bergman classes via tautological classes}
\label{sec:Bergmanvia}

In this section and the next, we use the special case of \Cref{thm:4degintro} (\Cref{thm:betainvar}) and the matroid minors decomposition property (\Cref{prop:tautrest}) 
to express certain Chow classes associated to matroids in terms of the tautological Chern classes of matroids.
We recover along the way several known results about these Chow classes.
The Chow classes of interest will be phrased in terms of Minkowski weights.  We review Minkowski weights in \S\ref{sec:minkowskiweights}, and then we consider the Bergman classes of matroids in \S\ref{subsec:MWChern} and \S\ref{sec:wonderfulbergman}.
In the next section, we consider the Chern-Schwartz-MacPherson classes of matroids.

\subsection{Minkowski weights}
\label{sec:minkowskiweights}

Let $\Sigma$ be a polyhedral fan in $\RR^m$ that is rational over the lattice $\ZZ^m$.  Let $\Sigma(d)$ denote the set of $d$-dimensional cones of the fan $\Sigma$.
Recall that $\Sigma$ is said to be unimodular if for every cone $\tau$ in $\Sigma$, the primitive vectors of the rays of $\tau$ form a subset of a $\ZZ$-basis of $\ZZ^m$.  
The toric variety $X_{\Sigma}$ is smooth if and only if $\Sigma$ is unimodular.
When $\Sigma$ is unimodular and complete, the validity of the Poincar\'e duality for the Chow ring $A^\bullet(X_{\Sigma})$ of $X_{\Sigma}$ states that we have an isomorphism 
$A^\bullet(X_{\Sigma})\cong \text{Hom}(A^{m-\bullet}(X_{\Sigma}),\mathbb{Z})$ given by sending
\[
\xi \in A^{m-d}(X_{\Sigma}) \text{ to the function } A^d(X_{\Sigma}) \to \ZZ \text{ defined by } [Z]\mapsto \int_{X_{\Sigma}}\xi \cdot [Z].
\]
Because the Chow classes of torus-orbit-closures of a toric variety generate its Chow ring, the function $\int_{X_{\Sigma}} \xi \cdot (-)\colon  A^d(X_{\Sigma}) \to \ZZ$ is determined by its values on the complementary dimensional torus-orbit-closures $[Z_{\tau}]$ corresponding to cones $\tau\in \Sigma(d)$.
Because the classes $[Z_{\tau}]$ are in general not linearly independent, an assignment $\Sigma(d) \to \ZZ$ of integers to each $[Z_{\tau}]$ must satisfy certain ``balancing conditions'' to define a map $A^d(X_\Sigma)\to \ZZ$.
These observations lead to the notion of a Minkowski weight.

\begin{defn}\label{defn:MW}
A $d$-dimensional \textbf{Minkowski weight} on a unimodular fan $\Sigma$ is a function $\Delta\colon  \Sigma(d) \to \ZZ$ such that the following balancing condition is satisfied for every cone $\tau'\in \Sigma(d-1)$:
\[
\sum_{\tau\succ \tau'} \Delta(\tau)u_{\tau'\setminus\tau} \in \operatorname{span}(\tau')
\]
where the summation is over all cones $\tau\in \Sigma(d)$ containing $\tau'$, and $u_{\tau'\setminus\tau}$ denotes the primitive generator of the  unique ray of $\tau$ that is not in $\tau'$.  Write $\operatorname{MW}_d(\Sigma)$ for the set of $d$-dimensional Minkowski weights on $\Sigma$.
\end{defn}
\begin{thm}\cite[Theorem 3.1]{FS97}\label{thm:FS}
Let $\Sigma$ be a complete unimodular fan of dimension $m$, and let $X_\Sigma$ be its toric variety.  Then, for every $0\leq d \leq m$, one has an isomorphism
\[
A^{m-d}(X_\Sigma) \overset\sim\to \operatorname{MW}_d(\Sigma) \quad\text{defined by}\quad \xi \mapsto \left( \tau \mapsto \int_X \xi \cdot [Z_\tau]\right).
\]
\end{thm}

For a smooth complete toric variety $X_\Sigma$, when a Chow class $\xi\in A^\bullet(X_\Sigma)$ maps to a Minkowski weight $\Delta \in \operatorname{MW}_\bullet(\Sigma)$ by the isomorphism in \Cref{thm:FS}, we say that $\Delta$ and $\xi$ are \textbf{Poincar\'e duals} of each other, which is notated by writing
\[
\xi \cap \Delta_{\Sigma} = \Delta.
\]
See \cite[Chapter 6]{MS15} for a further treatment of tropical intersection theory.

\subsection{Minkowski weights of Tautological classes}
\label{subsec:MWChern}
We compute the Poincar\'e duals of the tautological Chern classes of a matroid in \Cref{prop:MWChern}.  In the next subsection, we will use a special case of \Cref{prop:MWChern} concerning the top Chern classes to study the Bergman classes of matroids.
We prepare with the following lemma.
Recall that we say that a matroid is a \textbf{loop} (resp.\ \textbf{coloop}) if it has rank 0 (resp.\ rank 1) and its ground set has cardinality 1.

\begin{lem}\label{lem:degChern}
For a matroid $M$ of rank $r$ with ground set $E$, we have
\begin{align*}
\int_{X_E} c(\cQ_M) &= \begin{cases}
1 & \text{if $M$ is a loop, or is the rank $1$ uniform matroid $U_{1,E}$}\\
0 & \text{otherwise.}
\end{cases}\\
\int_{X_E} c(\cS_M^{\vee}) &= \begin{cases}
1 & \text{if $M$ is a coloop, or is the corank $1$ uniform matroid $U_{n,E}$}\\
0 & \text{otherwise}.
\end{cases}
\end{align*}
\end{lem}

\begin{proof}
Let us prove the statement for $\cQ_M$; the statement for $\cS_M^\vee$ is proved similarly.  First, because $X_E$ is $n$-dimensional and $\cQ_M$ has rank $n+1 - r$, we have that the expression $\int_{X_E} c(\mathcal{Q}_M) =  \int_{X_E} (1 + c_1(\cQ_M) + \cdots + c_{n+1-r}(\cQ_M))$ equals $0$ unless $r\le 1$. If $r=0$, so that $M = U_{0,E}$, then setting $r = 0$ in the second identity of \Cref{thm:betainvar} gives $$\int_{X_E}c(\mathcal{Q}_M)=\int_{X_E} c_{|E|-1}(\mathcal{Q}_M)=\beta(M^{\perp})=
\begin{cases}1 & \text{if $|E|=1$}\\0 & \text{if $|E|\geq 2$},
\end{cases}$$
where for the last equality we note from the definition of Tutte polynomials that $T_{U_{0,E}}(x,y) = y^{|E|}$.
If $r=1$, so that $M = U_{0,E_\ell} \oplus U_{1, E\setminus E_\ell}$ for some subset $\emptyset\subseteq E_\ell \subsetneq E$, then setting $r = 1$ in the first identity of \Cref{thm:betainvar} gives
$$\int_{X_E}c(\mathcal{Q}_M)=\int_{X_E} c_{|E|-1}(\mathcal{Q}_M)=\beta(M)=\begin{cases}1 & \text{if $E_\ell = \emptyset$, i.e.\ $M = U_{1,E}$}\\0 & \text{otherwise},
\end{cases}$$
where for the last equality we note from the definition of Tutte polynomials that $T_{U_{0,E_\ell} \oplus U_{1,E\setminus E_\ell}}(x,y) = y^{|E_\ell|}T_{U_{1,E\setminus E_\ell}}(x,y) = y^{|E_\ell|}(x+y+\cdots + y^{|E\setminus E_\ell|-1})$.
\end{proof}

\begin{prop}\label{prop:MWChern}
Let $M$ be a matroid of rank $r$ with ground set $E$, and write $\crk_M = |E|-r$ for its corank.  Let $j$ and $k$ be nonnegative integers such that $j+k = |E|-1$.  For every chain $\mathscr S\colon  \emptyset \subsetneq S_1 \subsetneq \cdots \subsetneq S_{k} \subsetneq E$ of $k$ nonempty proper subsets of $E$, we have
\[
\int_{X_E} c_j(\mathcal{Q}_M)\cdot [Z_{\mathscr{S}}] = \begin{cases}1 & \begin{split} &\text{for $i = 0, \ldots, k$, exactly $\crk_M-j $ minors $M|S_{i+1}/S_{i}$ are loops}\\ &\text{and the rest are rank 1 uniform matroids $U_{1,S_{i+1}\setminus S_{i}}$}\end{split}\\
\\
0 & \text{otherwise.}\end{cases}
\]
\[
\int_{X_E} c_j(\mathcal{S}_M^\vee)\cdot [Z_{\mathscr{S}}] = \begin{cases}1 & \begin{split} &\text{for $i = 0, \ldots, k$, exactly $r-j $ minors $M|S_{i+1}/S_{i}$ are coloops}\\ &\text{and the rest are corank $1$ uniform matroids $U_{|S_{i+1}\setminus S_i|-1,S_{i+1}\setminus S_{i}}$}\end{split}\\
\\
0 & \text{otherwise.}\end{cases}
\]
\end{prop}

\begin{proof}
We prove the result for $\mathcal{Q}_M$, the result for $\mathcal{S}_M^{\vee}$ can be proved nearly identically, or by invoking the duality property (\Cref{prop:duality}).
By the matroid minors decomposition property (\Cref{prop:tautrest}), with $u$ a formal variable we have that 
\[
c(\mathcal{Q}_M,u)|_{Z_{\mathscr S}}=\bigotimes_{i=0}^k c(\mathcal{Q}_{M|S_{i+1}/ S_{i}},u) \in \bigotimes_{i=0}^k A^\bullet(X_{S_{i+1}\setminus S_{i}})[u].
\]
\Cref{lem:degChern} implies that, for each $i =0, \ldots, k$, we have $\int_{X_{S_{i+1}\setminus S_i}} c(\cQ_{M|S_{i+1}/S_i},u) = u^{\crk_{M|S_{i+1}/S_i}-1}$ if the minor $M|S_{i+1}/ S_i$ is a loop,  $u^{\crk_{M|S_{i+1}/ S_i}}$ for a rank 1 uniform matroid, and $0$ otherwise, which yields the desired statement since $\sum_{i=0}^k \crk_{M|S_{i+1}/S_i}=\crk_M$.
\end{proof}

\subsection{Bergman classes of matroids and wonderful compactifications of hyperplane arrangement complements}
\label{sec:wonderfulbergman}
By considering the Minkowski weight for the top tautological Chern class $c_{|E|-r}(\mathcal{Q}_M)$ of a matroid $M$, we recover the notion of the Bergman class of a matroid and its relation to the geometry of the wonderful compactification of a hyperplane arrangement complement.

\begin{defn}\label{defn:Bergman}
For a matroid $M$ with ground set $E$, a subset $F\subseteq E$ is a \textbf{flat} of $M$ if $\operatorname{rk}_M(F\cup \{e\}) > \operatorname{rk}_M(F)$ for all $e\in E\setminus F$. The \textbf{Bergman fan} $\Sigma_M$ of a loopless matroid $M$ of rank $r$ is a subfan of $\Sigma_E$ whose set of maximal cones consists of $\operatorname{Cone}(\overline\be_{F_1}, \ldots, \overline\be_{F_{r-1}})$ for each maximal chain of nonempty proper flats $\emptyset\subsetneq F_1 \subsetneq \cdots \subsetneq F_{r-1} \subsetneq E$ of $M$.
The \textbf{Bergman class} $\Delta_M$ of a matroid $M$ of rank $r$ with ground set $E$ is the function $\Delta_M\colon  \Sigma_E(r-1) \to \ZZ$ defined by 
 \[
 \operatorname{Cone}(\overline\be_{S_1}, \ldots, \overline\be_{S_{r-1}}) \mapsto \begin{cases}
 1 &\text{if $M$ is loopless and $\operatorname{Cone}(\overline\be_{S_1}, \ldots, \overline\be_{S_{r-1}})  \in \Sigma_M$}\\
 0 &\text{otherwise.}
 \end{cases}
 \]
\end{defn}

The Bergman fan of a matroid was previously studied in \cite{Stu02,FS05,AK06,Spe08} as a tropical linear space, and is a key object in the Hodge theory of matroids \cite{AdiHuhKatz}.  We caution that in these works, the terminology ``Bergman fan'' of a matroid $M$ sometimes refers to a more coarsened fan structure on the support of $\Sigma_M$ than defined here.
Here, we show that $\Delta_M$ is a Minkowski weight, and identify it with $c_{|E|-r}(\cQ_M)\cap \Delta_{\Sigma_E}$, the Poincar\'e dual of the top Chern class of $\cQ_M$.

\begin{thm}\label{thm:BergmanComb}
Let $M$ be a matroid of rank $r$ with ground set $E$. For every chain $\mathscr S\colon  \emptyset \subsetneq S_1 \subsetneq \cdots \subsetneq S_{r-1} \subsetneq E$ of $(r-1)$ nonempty proper subsets of $E$, we have
\[
\int_{X_E} c_{|E|-r}(\cQ_M) \cdot [Z_{\mathscr S}]= \Delta_M(\operatorname{Cone}(\overline{\be}_{S_1},\ldots,\overline{\be}_{S_{r-1}})).
\]
In particular, the function $\Delta_M$ is a Minkowski weight, and we have the equality of Minkowski weights $$c_{|E|-r}(\cQ_M)\cap \Delta_{\Sigma_E}=\Delta_M.$$
\end{thm}

\begin{proof}
By \Cref{prop:MWChern}, we have that $\int_{X_E} c_{|E|-r}(\cQ_M) \cdot [Z_{\mathscr S}]$ equals 1 if all minors $M|S_{i+1}/S_i$ are rank 1 uniform matroids and equals 0 otherwise.  Thus, we need to show that the chain $\mathscr S\colon  \emptyset \subsetneq S_1 \subsetneq \cdots \subsetneq S_{r-1} \subsetneq E$ satisfies $M|S_{i+1}/S_i = U_{1, S_{i+1}\setminus S_i}$ for all $i = 0, \ldots, r-1$ if and only if $M$ is loopless and $\mathscr S$ is a maximal chain of nonempty proper flats of $M$.

We first show that if $M|S_{i+1}/S_i = U_{1, S_{i+1}\setminus S_i}$ for all $1\leq i \leq r-1$, then $S_i$ is a flat for all $1\leq i \leq r-1$.
If a subset $S \subsetneq E$ is not a flat, then by definition there exists $e\in E\setminus S$ such that $\operatorname{rk}_M(S \cup \{e\}) = \operatorname{rk}_M(S)$, and hence the element $e$ is a loop in $M/S$.  If $S$ is a flat of corank 1 in $M$, then $M/S$ is a matroid of rank 1 with no loops, and hence is the matroid $U_{1,E\setminus S}$.  Moreover, note that if $S\subseteq E$ is a flat of $M$, then the flats of $M|S$ are the flats of $M$ contained in $S$.  Hence, by backwards induction on $i$ starting with $i = r-1$, we conclude that $M|S_{i+1}/S_i = U_{1, S_{i+1}\setminus S_i}$ for $i=1,\ldots,r-1$ if and only if $S_i$ is a flat of corank 1 in $M|S_{i+1}$, and hence a flat in $M$, for all $i = 1, \ldots, r-1$.  
Lastly, if $S_1$ is a flat of $M$, then it contains the loops of $M$, so that $M|S_1$ is loopless if and only if $M$ is loopless.
\end{proof}
The fact that $\Delta_M$ is a Minkowski weight recovers the computation from \cite[Proposition~4.3]{H18} that $\Delta_M$ satisfies the balancing condition of \Cref{defn:MW}. Also, in  \Cref{rem:modular}, we show that a set of Chern roots of $\cQ_M$ is given by $\{\alpha - \alpha_{\mathscr F_i}\}_{i = r, \ldots, n}$ for certain modular filters $\mathscr F_i$.  In this light, \Cref{thm:BergmanComb} states that the Minkowski weight of the product $\prod_{i=r}^n(\alpha - \alpha_{\mathscr F_i})$ equals the Bergman class $\Delta_M$.  In \cite[Remark 31]{huh2014rota}, the same statement is deduced via a tropical intersection formula in \cite{AR10}.

\begin{rem}
By a similar computation, one can recover the computation of the Minkowski weight for the ``truncated Bergman fan'' $\Delta_{M,[d_1,d_2]}$ from \cite{huh2014rota,Kat14}, which in our notation is equal to the Minkowski weight $\alpha^{r-1-d_2}\beta^{d_1-1}c_{|E|-r}(\mathcal{Q}_M) \cap \Delta_{\Sigma_E}$ by \Cref{thm:BergmanComb}. Indeed, along the lines of \Cref{prop:MWChern}, we have by \Cref{prop:tautrest} and \Cref{cor:alphabetarest} that
$$\alpha^i\beta^jc(\mathcal{Q}_M)|_{Z_{\mathscr S}}=\beta^jc(\mathcal{Q}_{M|S_1},u)\otimes \bigotimes_{i=1}^{k-1}c(\mathcal{Q}_{M|S_{i+1}/S_i},u)\otimes \alpha^ic(\mathcal{Q}_{M/S_k},u),$$
and one can similarly apply \Cref{thm:4degintro} to each term of the tensor product, noting that for any flat $F$, the expression $|\mu(\emptyset,F)|$ in \cite{huh2014rota,Kat14} is equal to $T_{M|F}(1,0)$.
\end{rem}

For $M$ a realizable matroid, we now show how our expression of $\Delta_M$ as $c_{|E|-r}(\cQ_M)\cap \Delta_{\Sigma_E}$ recovers the relation between the class $\Delta_M$ and the geometry of wonderful compactifications of hyperplane arrangement complements introduced in \cite{dCP95}.

\begin{defn}
Let $L\subseteq \CC^E$ be a realization of a matroid $M$, and let $\PP L \subset \PP^n$ be the projectivization.  The \textbf{wonderful compactification} $W_L$ associated to the realization $L$ is the closure of $\PP L \cap (T/\CC^*)$ in $X_E$.\footnote{Our definition here may look different from the one in \cite{dCP95}.  First, the wonderful compactification as defined here is the ``maximal building set'' wonderful compactification, whereas \cite{dCP95} more generally studies wonderful compactifications from arbitrary maximal building sets.  Second, the wonderful compactification is originally constructed via blow-ups.  From the fact that $X_E$ can be constructed as a series of blow-ups from $\PP^n$, one can deduce the equivalence between the description of $W_L$ as a blow-up and the description here as a closure in $X_E$.  See for example \cite[Section~6]{H18} for an exposition of this equivalence.}
\end{defn}

\begin{rem}\label{rem:SNC}
For $i\in E$, let $H_i$ be the coordinate hyperplane of $\CC^E$.  A realization $L\subseteq \CC^E$ of a matroid defines a projective hyperplane arrangement $\{\PP\mathcal H_i \subset \PP L\}_{i\in E}$ on $\PP L$, where $\mathcal H_i = L \cap H_i$.  In other words, the set $\PP L \cap (T/\CC^*)$ is the complement of the projective hyperplane arrangement.   \cite[Theorem 4.2]{dCP95} states that the wonderful compactification $W_L$ is the compactification of $\PP L \cap (T/\CC^*)$ whose boundary $D_{W_L} = W_L \setminus (\PP L \cap (T/\CC^*))$ is a union of smooth irreducible divisors with simple normal crossing.
\end{rem}

The following theorem relates the wonderful compactification to the tautological quotient bundle $\cQ_L$ of a realization $L$ of a matroid.

\begin{thm}\label{thm:BergmanGeom}
Let $L \subseteq \CC^E$ be a subspace of dimension $r$, and for $\mathbf v \in \CC^E$, let $s_{\mathbf v}$ be the global section of the tautological quotient bundle $\cQ_L$ obtained as the image under $\uCCinv^E \to \cQ_L$ of the constant $\mathbf v\in \CC^E$ section of the bundle $\uCCinv^E$.  Then, for any $\mathbf a\in (\CC^*)^E$ we have
\[
W_{\mathbf a^{-1} L} = \text{the vanishing loci $\{p\in X_E \mid s_{\mathbf a}(p) = 0\}$ of the global section $s_{\mathbf a}$ of $\cQ_L$}.
\]
In particular, since $\cQ_L$ is globally generated by $\uCCinv^E \to \cQ_L$, for any $\mathbf a\in (\CC^*)^E$ we have
\[
[W_{\mathbf a^{-1} L}] = c_{|E|-r}(\cQ_L)\quad\text{as elements in }A^{\bullet}(X_E).
\]
\end{thm}

\begin{proof}
We prove the theorem for $\mathbf a = \mathbf 1 = (1,1,\ldots, 1)$, since the general case follows by $T$-equivariance.
By the short exact sequence $0\to \mathcal{S}_L\to \uCCinv^E \to \mathcal{Q}_L \to 0$, for any $p\in X_E$ we have $s_{\mathbf 1}(p)=0$ precisely when $\mathbf 1$ belongs to the fiber $\mathcal{S}_L|_p$ of $\cS_L$ at $p$.
Recall that for $\overline t$, the point in the dense open torus $T/\CC^*$ corresponding to a point $t\in T$, the fiber $\cS_L|_{\overline t}$ is by definition $t^{-1} L$.

We first treat the case where $L$ is contained in a coordinate hyperplane $\{x_i=0\}$, i.e. the corresponding matroid has a loop. In this case, by construction we have $W_L=\emptyset$.  Also, the fiber of $\mathcal{S}_L$ over $\overline{t}$ in the dense open torus $T/\CC^*$ is always contained in $\{x_i=0\}$. Since this is a closed condition, the same true is for every fiber of $\mathcal{S}_L\to X_E$. Since $\mathbf 1\not\in \{x_i=0\}$, this means $\mathbf 1$ is not contained in any fiber of $\mathcal{S}_L$, even on the boundary of $X_{E}$. That is, the section $s_{\mathbf 1}$ is nowhere vanishing on $X_E$.

We now treat with the case where $L$ is not contained in a coordinate hyperplane.
On the dense open torus $T/\CC^*$ we have
\[
s_{\mathbf 1}(\overline t) = 0 \iff \mathbf 1 \in \mathcal{S}_L|_{\overline t}=t^{-1}L \iff t\in L \iff  \overline t \in \PP L.
\]
Thus, the vanishing locus of the section $s_{\mathbf 1}$ on the dense open torus $T/\CC^*$ is $\PP L \cap (T/\mathbb{C}^*)$.  Since $W_L$ is the closure in $X_E$ of $\PP L \cap (T/\mathbb{C}^*)$, we are done once we show that the vanishing loci of $s_\mathbf 1$ is irreducible

To show irreducibility, we consider the map $\pi\colon \mathcal{S}_L \to \uCCinv^E= X_{E} \times \CCinv^E \to \CCinv^E$, which is dominant since $L$ is not contained in a coordinate hyperplane.  Note that for $\mathbf a\in \CC^E$, the fiber $\pi^{-1}(\mathbf a) = \{(x,\mathbf a) \in X_E \times \CCinv^E \colon \mathbf a\in \cS_L|_x\}$ is isomorphic to the vanishing loci of the section $s_{\mathbf a}$.  The affine bundle $\cS_L$ is irreducible, and since a general fiber of any dominant map of two varieties is pure-dimensional, we have that a general fiber of $\pi$ is irreducible if a general fiber of the restriction of $\pi$ to the open subset $\cS_L|_{T/\CC^*}$ is irreducible.  By $T$-equivariance, the previous conclusion that $\{s_{\mathbf 1}=0\}\cap (T/\mathbb{C}^*)=\PP L\cap (T/\mathbb{C}^*)$ thus implies that a general fiber of $\pi$ is irreducible.  Again by $T$-equivariance, the fiber $\pi^{-1}(\mathbf 1)$ is a general fiber and hence irreducible, as desired.
\end{proof}

Combining \Cref{thm:BergmanComb} with \Cref{thm:BergmanGeom}, we obtain the following properties of the Bergman class of a matroid.

\begin{cor}\label{cor:Bergman}
Let $M$ be a matroid with ground set $E$.  Then,
\begin{enumerate}[label = (\roman*)]
\item If $L\subseteq \CC^E$ is a realization of $M$, then the Chow class $[W_L] \in A^{\bullet}(X_E)$ of the wonderful compactification is the Poincar\'e dual of $\Delta_M$, and in particular is independent of the realization.
\item \label{Bergman:val}The assignment $M \mapsto \Delta_M$ is valuative.
\end{enumerate}
\end{cor}

Part (i) recovers \cite[Proposition 4.2]{KP11}.  Part (ii) follows by applying \Cref{prop:valpolys} to \Cref{thm:BergmanComb}.
\Cref{lem:looplessval} implies that the two properties in \Cref{cor:Bergman} characterize the assignment $M\mapsto \Delta_M$.

\begin{rem}\label{rem:whyinverse}
For a realization $L\subseteq \CC^E$, we defined $\cS_L$ to be the vector bundle on $X_E$ such that the fiber over $\overline t \in T/\CC^*$ is the subspace $t^{-1}L \subseteq \CC^E$.  If we defined $\cS_L$ such that the fiber at $\overline t$ is $tL \subseteq \CC^E$ instead, the proof of \Cref{thm:BergmanGeom} implies that the vanishing loci of the section $s_{\mathbf 1}$ on the open torus $T/\CC^*$ is not the linear space $\PP L \cap (T/\CC^*)$, but instead its Cremona image $\crem(\PP L \cap (T/\CC^*))$.  In particular, such alternate definitions of $\cS_L$ and $\cQ_L$ result in $c_{|E|-r}(\cQ_M) \cap \Delta_{\Sigma_E} = \crem \Delta_M$.
\end{rem}

\section{Chern-Schwartz-MacPherson classes via tautological classes}
\label{sec:CSMvia}

By considering the Poincar\'e duals of products of tautological Chern classes, we recover the notion of Chern-Schwartz-MacPherson (CSM) classes of a matroid studied in \cite{LdMRS20}, and their relation to the geometry of hyperplane arrangement complements.

\subsection{Minkowski Weights of products of Chern classes}
We compute the Poincar\'e duals of any product $c_i(\mathcal{S}_M^{\vee})c_j(\mathcal{Q}_M)$ of tautological Chern classes of a matroid $M$.
We will use this to express the CSM classes of matroids in terms of the tautological Chern classes.

\begin{prop}\label{prop:MWChernProd}
Let $M$ be a matroid with ground set $E$, and write $\rk_M$ and $\crk_M$ for its rank and corank, respectively.  For a chain $\mathscr S\colon  \emptyset\subsetneq S_1 \subsetneq \cdots \subsetneq S_k \subsetneq E$ of nonempty proper subsets of $E$, we have
\[
\int_{X_E} c(\mathcal{S}_M^{\vee},z)c(\mathcal{Q}_M,w) \cdot [Z_{\mathscr{S}}] =z^{\rk_M}w^{\crk_M}\prod_{i=0}^k \Big(\beta(M|S_{i+1}/S_i)\frac{1}{z}+\beta((M|S_{i+1}/S_i)^{\perp})\frac{1}{w}\Big).
\]
\end{prop}

\begin{proof}
By the matroid minor decomposition property \Cref{prop:torusrest} we have
\begin{align*}
\int_{X_E} c(\mathcal{S}_M^{\vee},z)c(\mathcal{Q}_M,w)\cdot [Z_{\mathscr{S}}] &= \int_{Z_{\mathscr S}}\big(c(\mathcal{S}_M^{\vee},z)c(\mathcal{Q}_M,w)|_{Z_{\mathscr{S}}}\big)\\&=\prod_{i=0}^k \int_{X_{S_{i+1}\setminus S_i}} c(\mathcal{S}_{M|S_{i+1}/S_{i}},z) c(\mathcal{Q}_{M|S_{i+1}/S_{i}},w).
\end{align*}
By \Cref{thm:betainvar}, for a matroid $M'$ on a ground set $E'$ we have
$$\int_{X_{E'}} c(\mathcal{S}_{M'}^{\vee},z)c(\mathcal{Q}_{M'},w) = z^{\rk_{M'}}w^{\crk_{M'}} \big(\beta(M')\frac{1}{z}+\beta({M'}^{\perp})\frac{1}{w}\big).$$
Applying this with $M'=M|S_{i+1}/S_i$ for each $i = 0, \ldots, k$, the result follows.
\end{proof}

\subsection{CSM classes of matroids and hyperplane arrangement complements}
The following definition of CSM classes of a matroid was introduced in \cite[Definition 5]{LdMRS20}.

\begin{defn}\label{defn:CSM}
Let $M$ be a matroid of rank $r$ on ground set $E$.  For $0\leq k \leq r-1$, define the \textbf{$k$-dimensional CSM class} of $M$ to be the function $\csm_k(M)\colon  \Sigma_E(k) \to \ZZ$ defined by
\[
\operatorname{Cone}(\overline\be_{S_1}, \ldots, \overline\be_{S_k}) \mapsto \begin{cases}
(-1)^{r-1-k}\prod_{i=0}^k\beta(M|S_{i+1}/S_i) &\text{if $S_i$ is a flat of $M$ for all $0\leq i \leq k$}\\
0 &\text{otherwise}
\end{cases}
\]
for every chain $\emptyset \subsetneq S_1 \subsetneq \cdots \subsetneq S_k \subsetneq E$ of $k$ nonempty proper subsets of $E$.
\end{defn}

\begin{rem}
One may alternately define the function $\csm_k(M)\colon  \Sigma_E(k) \to \ZZ$ to be
\[
\operatorname{Cone}(\overline\be_{S_1}, \ldots, \overline\be_{S_k}) \mapsto (-1)^{r-1-k}\prod_{i=0}^k\beta(M|S_{i+1}/S_i)
\]
for the following reason.
If a matroid $M$ has a loop, say $e\in E$, then $M = U_{0,\{e\}}\oplus M\setminus \{e\}$, so that $\beta(M) = 0$ because $T_M(x,y)$ is divisible by $y = T_{U_{0,\{e\}}}(x,y)$.  Now, arguing the same way as in the proof of \Cref{thm:BergmanComb}, we have that the minors $M|S_{i+1}/S_i$ are loopless for all $i = 1, \ldots, k$ if and only if the subsets $S_i$ are flats of $M$ for all $i = 1, \ldots, k$, and $S_0 = \emptyset$ is a flat of $M$ if and only if $M$ is loopless. \end{rem}

We express the CSM classes of matroids in terms of products of tautological Chern classes of $M$ as follows.

\begin{thm}\label{thm:CSMComb}
Let $M$ be a matroid of rank $r$ with ground set $E$.  Then, for every $k = 0,1, \ldots, r-1$, and for every chain $\mathscr S\colon  \emptyset \subsetneq S_1 \subsetneq \cdots \subsetneq S_k\subsetneq E$ of $k$ nonempty proper subsets of $E$, we have
\[
\int_{X_E} c_{r-1-k}(\cS_M)c_{|E|-r}(\cQ_M) \cdot [Z_{\mathscr  S}] =
(-1)^{r-1-k}\prod_{i=0}^k\beta(M|S_{i+1}/S_i).
\]
In particular, for every $k=0,1,\ldots,r-1$, the function $\operatorname{csm}_k(M)$ is a Minkowski weight, and we have the equality of Minkowski weights $$c_{r-1-k}(\mathcal{S}_M)c_{|E|-r}(\cQ_M)\cap \Delta_{\Sigma_E}=\operatorname{csm}_k(M).$$
\end{thm}
\begin{proof}
Since $c_{r-1-k}(\cS_M^\vee) = (-1)^{r-1-k}c_{r-1-k}(\cS_M)$, we show that
$$\int_{X_E} c_{r-1-k}(\mathcal{S}_M^{\vee})c_{|E|-r}(Q_M)\cdot [Z_{\mathscr{S}}]= \prod_{i=0}^k \beta(M|S_{i+1}/S_i).$$
Since we are considering the $(|E|-r)$-th Chern class of $\cQ_M$, we first extract the $w^{|E|-r} = w^{\crk_M}$ coefficient from \Cref{prop:MWChernProd}.  Because no terms involving a $\frac{1}{w}$ can thus contribute, the coefficient of $w^{\crk_M}$ in the expression in \Cref{prop:MWChernProd} is
\[
z^{\rk_M}\prod_{i=0}^k \beta(S_{i+1}/S_i)\frac{1}{z} = z^{r - k-1} \prod_{i=0}^k \beta(S_{i+1}/S_i),
\]
so that $\int_{X_E} c_{r-1-k}(\mathcal{S}_M^{\vee})c_{|E|-r}(Q_M)\cdot [Z_{\mathscr{S}}]= \prod_{i=0}^k \beta(M|S_{i+1}/S_i)$ as desired.
\end{proof}

We note that setting $k=r-1$ recovers \Cref{thm:BergmanComb}. The fact that $\csm_k(M)$ are Minkowski weights recovers the computation \cite[Theorem 2.3]{LdMRS20} that the functions $\csm_k(M)$ satisfy the balancing condition of \Cref{defn:MW}.  Setting $x = y= 0$ in \Cref{thm:delcont} and applying it to \Cref{thm:CSMComb}, one recovers \cite[Proposition 5.2]{LdMRS20}, which states that CSM classes of a matroid $M$ on $E$ satisfies a deletion-contraction relation
\[
f_* \operatorname{csm}_k(M) = 
\begin{cases}
0 & \text{if $i\in E$ is a loop or a coloop}\\
\operatorname{csm}_k(M\setminus i) - \operatorname{csm}_k(M/i) & \text{otherwise}
\end{cases}
\]
for all $k$, where $f\colon  X_E \to X_{E\setminus i}$ is the map in \Cref{defn:projmap}.

\medskip
For $M$ a realizable matroid, we now show how our expression of $\csm_k(M)$ as $c_{r-1-k}(\cS_M)c_{|E|-r}(\cQ_M) \cap \Delta_{\Sigma_E}$ recovers the relation between $\csm_k(M)$ and the geometry of Chern-Schwartz-MacPherson (CSM) classes of hyperplane arrangement complements.  We do this by first reviewing how CSM classes relate to log-tangent sheaves, and then by relating log-tangent sheaves to the tautological subbundles of realizations of matroids.

\medskip
CSM classes, introduced in \cite{Macpherson,Schwartz}, are generalizations of characteristic classes of smooth and complete algebraic varieties. See \cite{A05} for a survey and \cite{A06} for a construction. For a constructible subset $Z$ of a complete complex algebraic variety $X$, its CSM class is an element $\operatorname{csm}(Z) \in A_{\bullet}(X)$ that equals the total Chern class $[Z]\cap c(\mathcal T_Z)$ of its tangent bundle $\mathcal T_Z$ when $Z$ is a smooth complete variety.  We consider CSM classes of hyperplane arrangement complements.

\begin{defn}\label{def:CSML}
For a realization $L\subseteq \CC^E$ of a loopless matroid $M$ of rank $r$ (i.e. an $r$-dimensional subspace not contained in a coordinate hyperplane), let $\mathcal C(L) = \PP L \cap (T/\CC^*)$ be the hyperplane arrangement complement.
We decompose its CSM class $\csm(\mathcal C(L))\in A^{\bullet}(X_{E})$ as
\begin{align*}
    \csm(\mathcal C(L)) = \sum_{k=0}^{r-1} \csm_k(\mathcal C(L)),
\end{align*}
where $\csm_k(\mathcal C(L))$ is the graded piece lying in $A_k(X_{E})\cong A^{n-k}(X_{E})$. For a realization $L\subseteq \CC^E$ of a matroid $M$ containing loops, the intersection $\PP L \cap (T/\CC^*)$ is necessarily empty, and hence we define $\csm(\mathcal C(L))$  
to be $0\in A^{\bullet}(X_E)$.
\end{defn}

Recall from \Cref{rem:SNC} that the wonderful compactification $W_L$ is the compactification of $\PP L \cap (T/\CC^*)$ with the simple normal crossing boundary $D_{W_L} = W_L\setminus (\PP L \cap (T/\CC^*))$. CSM classes of realizable matroids are related to log tangent sheaves from the following fact.

\begin{prop}
\label{prop:csmlogtangent}
Assume that $L \subset \CC^E$ is not contained in a coordinate hyperplane. The CSM class $\csm(\mathcal C(L))$ is equal to the Chern class $c(\mathcal T_{W_L}(-\log D_{W_L}))\in A^\bullet(W_L)$ pushed forward to $A^\bullet(X_E)$. 
\end{prop}

\begin{proof}
The CSM class of $\mathcal C(L)$ inside of $W_L$ is given by $c(\mathcal T_{W_L}(-\log D_{W_L}))\in A^\bullet(W_L)$ as it is the complement of a simple normal crossing divisor in a smooth variety \cite[Theorem 1]{A99}. Since CSM classes behave compatibly with pushforward along proper maps,
$\csm(\mathcal C(L))\in A^\bullet(X_E)$ is given by the Chern classes of the log tangent sheaf $\mathcal T_{W_L}(-\log D_{W_L})$ pushed forward under $A^\bullet(W_L)\to A^\bullet(X_E)$. 
\end{proof}

\begin{rem}
For $D$ a divisor in a variety $X$, local sections of $T_X(-\log D)$ should be viewed as vector fields that are tangent to $D$. We will only consider the case when $D$ is a simple normal crossing divisor and $X$ is smooth. In this case, the intuition above is formalized in a short exact sequence
\begin{align*}
    0\to T_X(-\log D)\to T_X\to \bigoplus \mathcal{O}_{D_i}(D_i)\to 0,\label{eq:exact}
\end{align*}
where $D_i$ are the components of $D$.  See \cite[Section 2]{A99} or \cite[Section 3]{S96}.
\end{rem}

The following theorem relates the log tangent sheaf to tautological bundles.

\begin{thm}\label{thm:CSMGeom}
For $L\subseteq \CC^E$ a dimension $r$ subspace of $\mathbb{C}^E$ not contained in a coordinate hyperplane, we have a short exact sequence
\begin{align*}
0\to \mathcal{O}_{W_L}\to \mathcal{S}_L|_{W_L} \to \mathcal T_{W_L}(-\log D_{W_L})\to 0.
\end{align*}
of sheaves on the wonderful compactification $W_L$.  As a consequence, for any $L \subset \CC^E$ and each $k = 0, \ldots, r-1$, we have
\[
\csm_k(\mathcal C(L)) = c_{r-1-k}(\cS_L)c_{|E|-r}(\cQ_L) \quad\text{as elements in $A^\bullet(X_E)$}.
\]
\end{thm}

\begin{proof}
If $L$ is contained in a coordinate hyperplane then its matroid must have at least one loop. It follows from \Cref{prop:MWChern} that $c_{|E|-r}(\cQ_L) =0$ and hence $\csm_k(\mathcal C(L)) = c_{r-1-k}(\cS_L)c_{|E|-r}(\cQ_L)  = 0$ for all $0 \leq k \leq r-1$, in agreement with \Cref{def:CSML}. 

For the rest of the argument we assume that $L$ is not contained in a coordinate hyperplane. In this case, it is sufficient to prove the first claim of the theorem, since the second is an immediate consequence. 

To begin, we claim there is a short exact sequence
\begin{align*}
0\to T_{W_L}(-\log D_{W_L})\to T_{X_{E}}(-\log D_{X_{E}})|_{W_L}\to N_{W_L/X_{E}}\to 0.
\end{align*}
The components $D_{W_L,S}$ of $D_{W_L}$ are in bijection with partial intersections of the hyperplanes $\mathcal{H}_i=L\cap \{x_i=0\}$, which are the nonempty proper flats $S\subset E=\{0,\ldots,n\}$ of the matroid $M$.
The components $D_{X_E,S}$ of $D_{X_{E}}$ are similarly indexed by all nonempty proper subsets $S\subset E$. 

Consider the following diagram, where the dashed arrows are maps that we need to show exist. Here the left two vertical columns are the defining short exact sequence for log-tangent sheaves.
\begin{center}
    \begin{tikzcd}
    & 0 \ar[d] & 0 \ar[d] & 0 \ar[d] &\\
    0 \ar[r] & T_{W_L}(-\log D_{W_L})\ar[d]\ar[r, dashed] & T_{X_{E}}(-\log D_{X_{E}})|_{W_L} \ar[d] \ar[r] & N_{W_L/X_{E}} \ar[d]\ar[r] & 0\\
    0 \ar[r] & T_{W_L}\ar[d]\ar[r] & T_{X_{E}}|_{W_L} \ar[d] \ar[r] & N_{W_L/X_{E}} \ar[d]\ar[r] & 0\\    
    0 \ar[r] & \displaystyle\bigoplus_{\substack{S\text{ proper}\\ \text{flat of $M$}}}{\mathscr{O}_{D_{W_L,S}}(D_{W_L,S})}  \ar[r, dashed] \ar[d] & \displaystyle\bigoplus_{\emptyset\subsetneq S\subsetneq \{0,\ldots,n\}}{\mathscr{O}_{D_{X_E,S}}(D_{X_E,S})|_{W_L}} \ar[r] \ar[d] & 0 \ar[r] \ar[d] & 0\\
    & 0 & 0 & 0 &
    \end{tikzcd}
\end{center}
Following \cite[Proof of Theorem 6.3(1)]{H18}, each divisor $D_{X_E,S}$ for $S$ a proper subset of $E$ intersects $W_M$ if and only if $S$ is a flat of $M$ (a maximal collection of indices such that $\cap_{i \in S} \mathcal{H}_i$ intersects in a fixed subspace), and in this case $D_{X_E,S}\cap W_L=D_{W_L,S}$ scheme-theoretically. Thus, $\bigoplus_{S\text{ flat of }M}{\mathscr{O}_{D_{W_L,S}}(D_{W_L,S})}$ is isomorphic to $\bigoplus_{\emptyset\subsetneq S\subsetneq \{0,\ldots,n\}}{\mathscr{O}_{D_{X_E,S}}(D_{X_E,S})|_{W_L}}$. This then implies the top map exists, and by the nine-lemma the top row is exact.

Next, the log-tangent sheaf $T_{X_{E}}(-\log D_{X_{E}})$ is trivial since $X_{E}$ is a smooth complete toric variety, and fits into the short exact sequence
$$0\to \underline{\mathbb{C}}\to \bigoplus_{i=0}^{n} \underline{\mathbb{C}}t_i\frac{\partial}{\partial t_i}\to T_{X_{E}}(-\log D_{X_{E}})\to 0,$$
where the first inclusion takes $\mathbf{1}\mapsto \sum t_i\frac{\partial}{\partial t_i}$.

Pulling back the inclusion $T_{W_L}(-\log D_{W_L})\hookrightarrow T_{X_{E}}(-\log D_{X_{E}})|_{W_L}$ under the surjective mapping $\bigoplus_{i=0}^{n} \underline{\mathbb{C}}t_i\frac{\partial}{\partial t_i}\to T_{X_{E}}(-\log D_{X_{E}})$ restricted to $W_L$, we get some subbundle $\mathcal{F}\subset \bigoplus_{i=0}^{n} \underline{\mathbb{C}}|_{W_L}t_i\frac{\partial}{\partial t_i}$. This yields the following diagram.
\begin{center}
\begin{tikzcd}
0\ar[r] &  \underline{\mathbb{C}}|_{W_L} \ar[r] \arrow[transform canvas={xshift=0.25ex},-]{d} \arrow[transform canvas={xshift=-0.25ex},-]{d} &  \mathcal{F} \ar[r] \ar[d, hook] & T_{W_L}(-\log D_{W_L}) \ar[r] \ar[d, hook] & 0\\
0\ar[r] &  \underline{\mathbb{C}}|_{W_L} \ar[r]  &  \bigoplus_{i=0}^{n}\underline{\mathbb{C}}|_{W_L}t_i\frac{\partial}{\partial t_i} \ar[r] & T_{X_{E}}(-\log D_{X_{E}})|_{W_L} \ar[r]  & 0
\end{tikzcd}
\end{center}
Under the identification\footnote{this identification does not respect the natural $T$-equivariant structure, which is to act trivially on the left hand side} $\bigoplus_{i=0}^n\underline{\mathbb{C}}t_i\frac{\partial}{\partial t_i}\cong \uCCinv^{n+1}$ with $t_i\frac{\partial}{\partial t_i}\mapsto e_i$, we need to check  $\mathcal{F}$ agrees with $\mathcal{S}_L|_{W_L}$. 

To check this, it suffices to restrict to $W_L\cap T$. Fix a point $t\in W_L\cap T$. In coordinates, $t$ is specified by a point in $(\CC^{*})^{n+1}$ up to the diagonal scaling under $\CC^{*}$. As a subbundle of $\uCCinv^{n+1}$, the fiber of $\mathcal{S}_L|_{W_L\cap T}$ over $t$ is $t^{-1}L$. 
The fiber of $\mathcal{F}$ over $t$ is all $v=(v_0,\ldots,v_n)$ such that $\sum_{i=0}^{n}v_it_i\frac{\partial}{\partial t_i}$ lies in the tangent space to $t$ at $t\in L\cap (\mathbb{C}^{\times})^{n+1}\subset \CC^{n+1}$. This is equivalently described as the set of all $v=(v_0,\ldots,v_n)\in \CC^{n+1}$ such that $(t_0 v_0,\ldots, t_n v_n)\in L\subset \CC^{n+1}$, which is $t^{-1}L$ as well. 

Since $\mathcal{F}$ agrees with $\mathcal{S}_L|_{W_L}$, the top row of the commutative diagram gives us our desired short exact sequence. 
\end{proof}

\begin{rem}\label{rem:betainvarByLog}
\Cref{thm:CSMGeom} implies $c_{r-1}(\mathcal{S}_L)|_{W_L}=c_{r-1}(T_{W_L}(-\log D_{W_L}))$ and \Cref{thm:BergmanGeom} implies $c_{r-1}(\mathcal{S}_L)|_{W_L}=(-1)^{r-1}c_{r-1}(\mathcal{S}^{\vee}_L)c_{|E|-r}(\cQ_M)$. 
A logarithmic version of the Poincar\'e-Hopf theorem \cite[Theorem 4.1]{S96} implies that $c_{r-1}(T_{W_L}(-\log D_{W_L}))$ equals the topological Euler characteristic of the projective hyperplane arrangement corresponding to $L$. The topological Euler characteristic is equal to $(-1)^{r-1}\beta(M)$ by \cite[Theorem~5.2]{OS80}. This yields an alternative proof of \Cref{thm:betainvar} for realizable matroids, which can be extended to all matroids via valuativity (\Cref{prop:valpolys} and \Cref{lem:looplessval}) by arguing similarly as in the end of \Cref{proof:betainvarAlt}.
\end{rem}

Combining \Cref{thm:CSMComb} with \Cref{thm:CSMGeom}, we obtain the following properties of the CSM classes of a matroid.

\begin{cor}\label{cor:CSM}
Let $M$ be a matroid or rank $r$ with ground set $E$.  Then, for each $k = 0, \ldots, r-1$
\begin{enumerate}[label = (\roman*)]
\item The CSM class $\csm_k(\mathcal C(L))$ of the hyperplane arrangement complement $\mathcal C(L)$ is the Poincar\'e dual of $\csm_k(M)$ for any realization $L\subseteq \CC^E$ of $M$.
\item The assignment $M \mapsto \csm_k(M)$ is valuative.
\end{enumerate}
\end{cor}

Part (i) recovers \cite[Theorem 3.1]{LdMRS20}.  Part (ii), which follows from applying \Cref{prop:valpolys} to \Cref{thm:CSMComb}, recovers \cite[Theorem 4.1]{LdMRS20}.
For each $k = 0, \ldots, r-1$, \Cref{lem:looplessval} implies that the two properties in \Cref{cor:CSM} characterize the assignment $M\mapsto \csm_k(M)$.

\section{Positivity and log-concavity}
\label{sec:logconcviataut}

When a matroid $M$ has a realization $L\subset \CC^E$, the vector  bundles $\cS_L^\vee$ and $\mathcal Q_L$ are globally generated and hence nef, so that their relative $\mathcal O(1)$ classes satisfy positivity and log-concavity properties listed in \cite[\S1.6]{L04}. 
In this section, we show that these properties persist for the $K$-classes $[\cS_M^\vee]$ and $[\cQ_M]$ for an arbitrary (not necessarily realizable) matroid $M$.  In particular, we prove \Cref{thm:masterlogconcintro}, which states that the unifying Tutte formula $t_M(x,y,z,w)$ of \Cref{thm:4degintro} satisfies a log-concavity property.

\medskip
We will proceed in three steps.  In \S\ref{sec:biprojoftaut}, we define the Chow class $[\PP(\cS_M) \times_{X_E} \PP(\cQ_M^\vee)]$ on $X_E\times \PP^n \times \PP^n$ which equals the Chow class of the biprojectivation $\PP(\cS_L) \times_{X_E} \PP(\cQ_L^\vee)$ when $L$ is a realization of $M$.
In \S\ref{sec:Minkweightuntwist}, we show that the class $[\PP(\cS_M) \times_{X_E} \PP(\cQ_M^\vee)]$ is a pushforward of the Poincar\'e dual of a Minkowski weight supported on a ``Lefschetz fan'' in the sense of \cite[Definition 1.5]{ADH}.
In \S\ref{sec:Logconcprops} and \S\ref{subsec:lorentzian}, we derive positivity and log-concavity properties via the tropical Hodge theory of Lefschetz fans developed in \cite[\S5]{ADH}.  Our approach allows us to avoid the intricate computations with the bipermutohedron carried out in \cite[\S2 \& \S4]{ADH}; see \Cref{rem:compareADH}.

\subsection{Bi-projectivizations of tautological classes}
\label{sec:biprojoftaut}

For a vector bundle $\mathcal E$ on a variety $X$, we denote by $\PP(\mathcal E) = \operatorname{Proj}_X\operatorname{Sym}^\bullet(\mathcal E^\vee) = (\mathcal{E}\setminus (X\times \{0\}))/\mathbb{C}^*$ its projectivization.
Let $\PPinv^n = \PP(\CCinv^E)$.
For a realization $L\subseteq \CC^E$ of a matroid $M$, one has inclusions $\PP(\cS_L) \subseteq X_E \times \PPinv^n$ and $\PP(\cQ_L^\vee) \subseteq X_E \times (\PPinv^n)^\vee$, and hence we can form the bi-projective bundle
$$\mathbb{P}(\mathcal{Q}_L^{\vee})\times_{X_{E}}\mathbb{P}(\mathcal{S}_L)\subset X_E\times (\PPinv^n)^\vee \times \PPinv^n.$$
We now define a combinatorial abstraction of the Chow class of the bi-projectivization $\PP(\cQ_L^\vee) \times_{X_E} \PP(\cS_L)$ for arbitrary matroids.
Let us denote the following elements in $A^1(X_E \times (\PPinv^n)^\vee \times \PPinv^n)$ by
\[
\delta\text{ and }\epsilon = \text{the pullbacks of the hyperplane classes of $(\PPinv^n)^\vee$ and $\PPinv^n$, respectively}.
\]
Let $\mu\colon  X_E \times (\PPinv^n)^\vee \times \PPinv^n \to X_E$ be the projection map, and for a class $\xi\in A^\bullet(X_E)$, we often write $\xi$ also for the pullback $\mu^*\xi$ when we trust that no confusion will arise.

\begin{defn}\label{defn:projClass}
For matroids $M_1$ and $M_2$ of ranks $r_1$ and $r_2$ on the common ground set $E$, we define a Chow class $[\PP(\cQ^\vee_{M_1}) \times_{X_E} \PP(\cS_{M_2})] \in A^\bullet(X_E \times (\PPinv^n)^\vee \times \PPinv^n)$ by
\[
[\PP(\cQ^\vee_{M_1}) \times_{X_E} \PP(\cS_{M_2})] = \sum_{k=0}^{r_1}\sum_{\ell = 0}^{n+1-r_2} c_k(\mathcal{S}_{M_1}^{\vee})c_\ell(\mathcal{Q}_{M_2})\delta^{r_1-k}\epsilon^{n+1-r_2-\ell}.
\]
\end{defn}

The Chow class defined in \Cref{defn:projClass} has the following characterizing properties.

\begin{prop}\label{prop:projClassDefn}
Let notations be as in the above definition.  The class $[\PP(\cQ^\vee_{M_1}) \times_{X_E} \PP(\cS_{M_2})]$ satisfies and is determined by the following two properties.
\begin{enumerate}[label = (\roman*)]
\item If $L_1$ and $L_2\subseteq \CC^E$ are realizations of $M_1$ and $M_2$, respectively, then $[\PP(\cQ^\vee_{L_1}) \times_{X_E} \PP(\cS_{L_2})] = [\PP(\cQ^\vee_{M_1}) \times_{X_E} \PP(\cS_{M_2})]$ as Chow classes in $A^\bullet(X_E \times (\PPinv^n)^\vee \times \PPinv^n)$, and
\item the assignment $(M_1, M_2) \mapsto [\PP(\cQ^\vee_{M_1}) \times_{X_E} \PP(\cS_{M_2})]$ is valuative in each factor $M_1$ and $M_2$.
\end{enumerate}
\end{prop}

\begin{proof}
The property (i) is immediate by the formula for the Chow class of the projectivization of a subbundle \cite[Proposition 9.13]{EH16}, noting that $\cS_{L_1}^\vee = (\uCCinv^E)^\vee/(\cQ_{L_1}^\vee)$ and $\cQ_{L_2} = \uCCinv^E/\cS_{L_2}$.  The property (ii) follows from \Cref{prop:valpolys}.  That these two properties characterize $[\PP(\cQ^\vee_{M_1}) \times_{X_E} \PP(\cS_{M_2})]$ follows from \Cref{lem:looplessval}.
\end{proof}

The following proposition relates mixed intersections of certain nef divisors with $[\mathbb{P}(\mathcal{Q}^{\vee}_{M_1})\times_{X_E} \mathbb{P}(\mathcal{S}_{M_2})]$ to the mixed intersections appearing in \Cref{thm:4degintro}.

\begin{prop}
\label{prop:Segre}
Let $M_1$ and $M_2$ be matroids of rank $r_1$ and $r_2$ on the common ground set $E$.  Then, the pushfoward map $\mu_*\colon  A^\bullet(X_E \times (\PPinv^n)^\vee \times \PPinv^n) \to A^\bullet(X_E)$ satisfies for all nonnegative integers $k$ and $\ell$
\[
\mu_*(\delta^{n-r_1+k}\epsilon^{r_2-1+\ell}[\mathbb{P}(\mathcal{Q}_{M_1}^{\vee})\times_{X_{E}}\mathbb{P}(\mathcal{S}_{M_2})])= c_k(\mathcal{S}_{M_1}^{\vee})c_\ell(\mathcal{Q}_{M_2}).
\]
In particular, for $i+j+k+\ell=n$ we have $$
\int_{X_E \times (\PPinv^n)^\vee \times \PPinv^n}\alpha^i\beta^j\delta^{n-r_1+k}\epsilon^{r_2-1+\ell}[\mathbb{P}(\mathcal{Q}_{M_1}^{\vee})\times_{X_{E}}\mathbb{P}(\mathcal{S}_{M_2})]= \int_{X_E}\alpha^i\beta^jc_k(\mathcal{S}_{M_1}^{\vee})c_\ell(\mathcal{Q}_{M_2}).$$
\end{prop}

\begin{proof}
As each $\delta$ and $\epsilon$ is a hyperplane class pullback from a projective space $\PP^n$, for any integers $i,j\geq 0$ we have $\mu_*(\delta^i\epsilon^j) = 1$ if $i = j = n$, and 0 otherwise.  We conclude by the definition of $[\mathbb{P}(\mathcal{Q}_{M_1}^{\vee})\times_{X_{E}}\mathbb{P}(\mathcal{S}_{M_2})]$ and the push-pull formula.
\end{proof}

\begin{rem}
We note at this point that we can conclude \Cref{thm:masterlogconcintro} when a matroid $M$ of rank $r$ has a realization $L\subseteq \mathbb{C}^E$ as follows.  
The classes $\alpha, \beta, \delta, \epsilon$ on $X_{A_E}\times (\PPinv^n)^{\vee}\times \PPinv^n$ are nef divisors, and hence restrict to nef divisors on the variety $\mathbb{P}(\mathcal{Q}_{L}^{\vee})\times_{X_E}\mathbb{P}(\mathcal{S}_{L})$.
By \Cref{prop:Segre} and \Cref{thm:4degintro}, we have
$$\sum_{i+j+k+\ell=n}\left(\int_{\mathbb{P}(\mathcal{Q}_{L}^{\vee})\times_{X_E}\mathbb{P}(\mathcal{S}_{L})}\alpha^i\beta^j\delta^k\epsilon^\ell\right)x^iy^jz^kw^\ell=\frac{1}{x+y}(y+z)^r(x+w)^{n+1-r}T_M(\frac{x+y}{y+z},\frac{x+y}{x+w}).$$
One concludes the desired log-concavity property by the classical Khovanskii-Teissier inequality for intersection multiplicities of nef divisors (see \cite{Khovanskii,Teissier} or \cite[ Corollary 1.6.3 (i)]{L04}).
\end{rem}

In the next few subsections, we will show \Cref{thm:masterlogconcintro} for arbitrary matroids by relating the intersection in \Cref{prop:Segre} with an equivalent intersection on a Lefschetz fan as defined in \cite{ADH}.  The log-concavity will then follow from the validity of mixed Hodge-Riemann relations on the Lefschetz fan.

\medskip
We conclude here with an observation that allows us to assume matroids to be loopless or coloopless under certain contexts. 
For a matroid $M$, let $M^{coloop\to loop}$ be $M$ with all coloops replaced by loops, and $M^{loop\to coloop}$ be $M$ with all loops replaced by coloops.

\begin{prop}\label{prop:noloopcoloop}
Let $M_1$ and $M_2$ be matroids of rank $r_1$ and $r_2$ on the common ground set $E$.
Then,
\[
[\PP(\cQ_{M_1}^\vee) \times_{X_E} \PP(\cS_{M_2})] =  [\PP(\cQ_{M_1^{coloop \to loop}}^\vee) \times_{X_E} \PP(\cS_{M_2^{loop \to coloop}})]
\]
as elements in $A^\bullet(X_E \times (\PPinv^n)^\vee \times \PPinv^n)$.
\end{prop}

\begin{proof}
From the definition of $[\PP(\cQ_{M_1}^\vee) \times_{X_E} \PP(\cS_{M_2})]$, it suffices to show for any matroid $M$ that $$c(\cS_{M^{loop \to coloop}}) = c(\cS_{M}),\text{ and }c(\cQ_{M^{loop \to coloop}}) = c(\cQ_{M}).$$ 
We only prove $c(\cS_{M^{loop \to coloop}}) = c(\cS_{M})$, as $c(\cQ_{M^{loop \to coloop}}) = c(\cQ_{M})$ is proved similarly.  Let $E_\ell \subseteq E$ be the (possibly empty) set of loops in $M$.  Then, we have $M = M|({E\setminus E_\ell}) \oplus U_{0,E_\ell}$ and $M^{loop\to coloop}=M|({E\setminus E_\ell}) \oplus U_{|E_\ell|,E_\ell}$.  Hence, since $\cS_{0,E_\ell} = 0$ and $\cS_{|E_\ell|,E_\ell} = [\uCCinv^{E_\ell}]$, \Cref{prop:directsum2} implies that $[\cS_{M^{loop \to coloop}}]$ equals the sum of $[\cS_M]$ and a trivial bundle (the pullback of $[\uCCinv^{E_\ell}]$).  Thus, their Chern classes coincide.
\end{proof}

\subsection{Minkowski weight of a birational model of the biprojectivization}
\label{sec:Minkweightuntwist}

Let $M_1$ and $M_2$ be matroids on the common ground set $E$.  
The goal of this subsection is to express the Chow class $[\mathbb{P}(\mathcal{Q}_{M_1}^{\vee})\times_{X_{E}}\mathbb{P}(\mathcal{S}_{M_2})]$ as a pushforward of a Minkowski weight that has certain ``Lefschetz properties'' we'll exploit in the next subsection.

\medskip
Let $\Sigma'$ be a pure $d$-dimensional subfan of an $m$-dimensional complete unimodular fan $\Sigma$. We say that $\Sigma'$ is a \textbf{balanced fan} if the function
\[
\Delta\colon  \Sigma(d) \to \ZZ \quad\text{defined by}\quad \Delta(\tau) = \begin{cases} 1 & \tau\subseteq \tau' \text{ for some } \tau'\in  \Sigma'(d) \\ 0 & \text{otherwise}
\end{cases}
\]
is a Minkowski weight on $\Sigma$, in which case we say that $\Delta\in \operatorname{MW}_d(\Sigma)$ is the \textbf{constant-1 Minkowski weight} on $\Sigma'$.
We write $[\Sigma'] \in A^{m-d}(X_\Sigma)$ for the Poincar\'e dual of $\Delta$.
\begin{eg}
The Bergman class $\Delta_M$ is the constant-1 Minkowski weight on the Bergman fan $\Sigma_M$ of a loopless matroid $M$ (see for example \Cref{thm:BergmanComb}).
More generally, the Minkowski weights of $c_k(\mathcal{S}_M^{\vee})$ and $c_\ell(\mathcal{Q}_M)$ are constant-$1$ Minkowski weights (\Cref{prop:MWChern}).
\end{eg}

\begin{thm}
\label{thm:refinedfanbiproj}
Consider $X_E \times (\PPinv^n)^\vee \times \PPinv^n$ as a toric variety with dense open torus $(T/\CC^*)^3$.  Then, there exists a smooth projective toric variety $X_{\Sigma}$ associated to a unimodular fan $\Sigma$ in $(\RR^E/\RR\mathbf 1)^3$, together with a birational toric morphism $\widetilde{\phi}\colon X_{\Sigma}\to X_E \times (\PPinv^n)^\vee \times \PPinv^n$, such that for $M_1$ a coloopless matroid and $M_2$ a loopless matroid we have
\begin{enumerate}
\item the product of fans $\Sigma_E \times \Sigma_{M_1^\perp} \times \Sigma_{M_2}$ is a coarsening of a subfan ${\Sigma_{E,M_1^{\perp},M_2}}\subset \Sigma$, and
\item under the pushfoward map $\widetilde\phi_*\colon  A^\bullet(X_\Sigma) \to A^\bullet(X_E\times (\PPinv^n)^{\vee}\times \PPinv^n)$ we have
\[
\widetilde\phi_*[{\Sigma_{E,M_1^{\perp},M_2}}] = [\mathbb{P}(\mathcal{Q}_{M_1}^{\vee})\times_{X_{E}}\mathbb{P}(\mathcal{S}_{M_2})].
\]
\end{enumerate}
\end{thm}

We prepare the proof by stating some facts from tropical intersection theory.
For a $d$-dimensional very-affine subvariety $Y$ in a torus $(\CC^*)^m$, the \textbf{tropicalization} of $Y$, denoted $\operatorname{trop}(Y)$, is a pure $d$-dimensional polyhedral complex in $\RR^m$, along with $\ZZ$-valued weight function $w$ on the set of its $d$-dimensional polyhedral cells.  The weight $w$ has the property that, for every complete unimodular fan $\Sigma$ in $\RR^m$ containing a subfan whose support equals $\operatorname{trop}(Y)$, the assignment
\[
\Delta_Y\colon  \Sigma(d) \to \ZZ \quad\text{defined by}\quad \tau \mapsto \begin{cases}
w(C) & \text{if there exists a polyhedral cell $C$ in $\operatorname{trop}(Y)$ containing $\tau$}\\
0 & \text{otherwise}
\end{cases}
\]
is a Minkowski weight on $\Sigma$.
One such unimodular fan $\Sigma$ can be constructed from the Gr\"obner fan of $Y$, which is a fan in $\RR^m$, not necessarily unimodular, that contains a subfan whose support equals $\operatorname{trop}(Y)$.
See \cite[Ch.~3]{MS15} for an introduction to tropicalizations and Gr\"obner fans.
The tropicalization of a product of very-affine subvarieties is the product of each tropicalization.
We will need the following two facts about tropicalizations from the theory of tropical compactifications.

\begin{lem}\label{lem:tropicalfacts}
\
\begin{enumerate}[label = (\alph*)]
\item\label{tropicalfacts:linear} Let $L\subset \CC^E$ be a realization of a loopless matroid $M$.  Then, the tropicalization of $\PP L \cap (T/\CC^*)$ is a polyhedral complex whose support equals the support of the Bergman fan $\Sigma_M$, along with the weight function $w$ that is constantly 1.  Moreover, the Bergman fan refines a subfan of the Gr\"obner fan of $\PP L \cap (T/\CC^*)$.
\item\label{tropicalfacts:class}  Let $Y\subseteq (\CC^*)^m$ be a very-affine subvariety, and $\Sigma$ a complete unimodular fan in $\RR^m$ that refines the Gr\"obner fan of $Y$ and contains a subfan whose support equals $\operatorname{trop}(Y)$.  Then, the Chow class $[\overline Y]\in A^\bullet(X_\Sigma)$ of the closure of $Y$ in the toric variety $X_\Sigma$ is equal to the Poincar\'e dual of the Minkowski weight $\Delta_Y$ defined by $\operatorname{trop}(Y)$.
\end{enumerate}
\end{lem}

\begin{proof}
The first part of \ref{tropicalfacts:linear} was first observed in \cite[\S9.3]{Stu02}; see \cite[Theorem 4.1.11]{MS15} for a proof. The second part that $\Sigma_M$ refines the Gr\"obner fan is \cite[Exercise 4.7.(7)]{MS15}, which was implicitly stated in \cite[Theorem 1]{AK06}.  The statement \ref{tropicalfacts:class} follows from combining \cite[Proposition 9.4]{Kat09} with \cite[Theorem 14.9]{G13}, or from combining \cite[Theorem 6.4.17]{MS15} with \cite[Theorem 6.7.7]{MS15}.
\end{proof}

We are now ready to prove \Cref{thm:refinedfanbiproj}.

\begin{proof}[Proof of \Cref{thm:refinedfanbiproj}]
First, we clarify how we are treating $X_E\times (\PPinv^n)^{\vee}\times \PPinv^n$ as a toric variety.
The standard basis of $\CC^E$ induces an isomorphism $\CC^E \simeq (\CC^E)^\vee$, and by forgetting the $T$-action on $\PPinv^n$ and $(\PPinv^n)^\vee$,
we identify $X_E\times (\PPinv^n)^{\vee}\times \PPinv^n=X_E\times \mathbb{P}^n\times \mathbb{P}^n$,
where the latter is a product of three toric varieties, each with open dense torus $(T/\CC^*)$.

We now specify the map $\widetilde\phi$ first on the torus $(T/\CC^*)^3$.
Define
$\phi_{\text{trop}}\colon (\mathbb{Z}^{n+1}/\mathbb{Z}\mathbf{1})^3\to (\mathbb{Z}^{n+1}/\mathbb{Z}\mathbf{1})^3$
by $\phi(u_0, u_1, u_2)=(u_0, u_0+u_1,  -u_0 + u_2)$. This induces an invertible map of tori $\phi\colon (T/\mathbb{C}^*)^3\to (T/\mathbb{C}^*)^3$ given by
$\phi(t_0,t_1,t_2)=(t_0,t_0 t_1, t_0^{-1} t_2).
$
Let $\Sigma$ be a unimodular fan in $(\RR^E/\RR\mathbf 1)^3$ that sufficiently refines the fan $\Sigma_E^3 = \Sigma_E \times \Sigma_E \times \Sigma_E$ such that $\phi_{\text{trop}}$ induces a map of fans $\Sigma\to \Sigma_E \times \Sigma_n \times \Sigma_n$, where $\Sigma_n$ is the fan of the toric variety $\PP^n$.
Such a fan $\Sigma$ can be constructed by noting that a collection of normal fans of polytopes admit a common refinement \cite[Proposition 6.2.13.(b)]{CLS11} and by applying the toric resolution of singularities \cite[Theorem 11.1.9]{CLS11}.
Then, the invertible map $\phi$ of tori extends to a birational toric morphism $\widetilde{\phi}\colon X_{\Sigma}\to X_E \times \PP^n \times \PP^n$. 

We now verify that $\Sigma$ and $\widetilde\phi$ satisfies the desired properties.  The property (1) is immediate from the construction. Indeed, $\Sigma_E \times \Sigma_{M_1^\perp} \times \Sigma_{M_2}$ is a subfan of $\Sigma_E^3$, and $\Sigma$ refines $\Sigma_E^3$, we can set $\Sigma_{M,M_1^\perp, M_2}$ to be the subfan of $\Sigma$ whose support equals the support of $\Sigma_E \times \Sigma_{M_1^\perp} \times \Sigma_{M_2}$.

For the property (2), we first claim that both $[\Sigma_{E,M_1^\perp, M_2}]$ and $[\PP(\cQ_{M_1}^\vee)\times_{X_E} \PP(\cS_{M_2})]$ are valuative separately in coloopless matroids $M_1$ and loopless matroids $M_2$.
For $[\Sigma_{E,M_1^\perp, M_2}]$, by considering the corresponding constant-1 Minkowski weight on $\Sigma$, the desired valuativity is equivalent to the valuativity of the indicator function for the support of the fan $\Sigma_E\times \Sigma_{M_1^\perp} \times \Sigma_{M_2}$.
\Cref{cor:Bergman}.\ref{Bergman:val} implies that the assigment $M\mapsto \text{(indicator function for the support of } \Sigma_M)$ is valuative, and similarly for the assignment $M\mapsto \text{(indicator function for the support of } \Sigma_{M^\perp})$ since $M \mapsto 1_{P(M^\perp)} = 1_{- P(M) + \mathbf 1}$ is valuative.
Thus, we conclude that $[\Sigma'_{E,M_1^\perp, M_2}]$ is valuative separately in coloopless matroids $M_1$ and loopless matroids $M_2$.  The valuativity for $[\PP(\cQ_{M_1}^\vee)\times_{X_E} \PP(\cS_{M_2})]$ follows from its definition and \Cref{prop:valpolys}.
Applying \Cref{lem:looplessval} to these valuative properties, we conclude that it suffices to show $\widetilde\phi_*[\Sigma'_{E,M_1^\perp,M_2}] = [\PP(\cQ_{M_1}^\vee)\times_{X_E} \PP(\cS_{M_2})]$ assuming that $M_1$ and $M_2$ both have realizations.

Now, let $L_1 \subseteq \CC^E$ and $L_2 \subseteq \CC^E$ be realizations of $M_1$ and $M_2$, respectively.  Recall that under $\CC^E \simeq (\CC^E)^\vee$, the subspace $L_1^\perp = (\CC^E/L_1)^\vee \subseteq (\CC^E)^\vee \simeq \CC^E$ realizes $M_1^\perp$.
Let $Y_{L_1, L_2}$ be the intersection of $X_E \times \PP L_1^\perp \times \PP L_2 \subseteq X_E \times \PP^n \times \PP^n = X_E \times (\PPinv^n)^\vee \times \PPinv^n$ with the open dense torus, namely,
$$Y_{L_1,L_2}=(T/\mathbb{C}^*)\times (\mathbb{P}L_1^{\perp} \cap T/\mathbb{C}^*) \times (\mathbb{P}L_2\cap T/\mathbb{C}^*).$$
Note that $Y_{L_1, L_2}$ is nonempty because $M_1^\perp$ and $M_2$ are assumed to be loopless.  The very-affine subvariety $Y_{L_1, L_2}$ is an ``untwisting'' of $\PP(\cQ_{L_1}^\vee)\times_{X_E} \PP(\cS_{L_2})$ on the open dense torus $(T/\CC^*)^3$ in the sense that $\phi$ maps $Y_{L_1, L_2}$ isomorphically onto
$(\mathbb{P}(\mathcal{Q}_{L_1}^{\vee})\times_{X_E}\mathbb{P}(\mathcal{S}_{L_2})) \cap (T/\mathbb{C}^*)^3$.
Indeed, for any $t_0 \in T/\CC^*$, the fibers over $\{t_0\}\times (T/\mathbb{C}^*)^2$ are
\begin{align*}
(Y_{L_1,L_2})_{\{t_0\}\times (T/\mathbb{C}^*)^2}=&\{t_0\}\times (\mathbb{P}L_1^{\perp}\cap T/\mathbb{C}^*)\times (\mathbb{P}L_2\cap T/\mathbb{C}^*)\quad\text{and} \\
(\mathbb{P}(\mathcal{Q}_{L_1}^{\vee})\times_{X_E}\mathbb{P}(\mathcal{S}_{L_2}))_{\{t_0\}\times (T/\mathbb{C}^*)^2}=&\{t_0\}\times (t_0\mathbb{P}L_1^{\perp}\cap T/\mathbb{C}^*)\times (t_0^{-1}\mathbb{P}L_2\cap T/\mathbb{C}^*),
\end{align*}
and $\phi$ was given by $\phi(t_0,t_1,t_2)=(t_0,t_0 t_1, t_0^{-1} t_2)$.
Thus, denoting $\overline{Y_{L_1,L_2}}$ for the closure of $Y_{L_1,L_2}$ in $X_\Sigma$, we have that $\widetilde{\phi}_*[\overline{Y_{L_1,L_2}}]=[\mathbb{P}(\mathcal{Q}_{L_1}^{\vee})\times_{X_E}\mathbb{P}(\mathcal{S}_{L_2})]$.
It only remains to show that $[\overline{Y_{L_1,L_2}}]=[{\Sigma_{E,M_1^{\perp},M_2}}]$.
\Cref{lem:tropicalfacts}.\ref{tropicalfacts:linear} implies that $\operatorname{trop}(Y_{L_1, L_2})$ is a polyhedral complex whose support equals the support of $\Sigma_E \times \Sigma_{M_1^\perp} \times \Sigma_{M_2}$, and the weight function is constantly 1.  Then, since \Cref{lem:tropicalfacts}.\ref{tropicalfacts:class} implies that the Chow class $[\overline{Y_{L_1,L_2}}]\in A^\bullet(X_\Sigma)$ is Poincar\'e dual to the Minkowski weight on $\Sigma$ defined by the tropicalization $\operatorname{trop}(Y_{L_1, L_2})$, we conclude that $[\overline{Y_{L_1,L_2}}]=[{\Sigma_{E,M_1^{\perp},M_2}}]$, as desired.
\end{proof}

\begin{rem}\label{rem:compareADH}
The fan ${\Sigma_{E,M_1^\perp,M_2}}$, serving as our combinatorial model of a biprojective bundle, is valuative separately in $M_1$ and $M_2$, allowing us to reduce to the realizable case. However, $\Sigma_{E,M^{\perp},M}$ is \emph{not} valuative in $M$. Similarly, the ``conormal fan'' ${\Sigma_{M,M^\perp}}$, whose support coincides with $\Sigma_M \times \Sigma_{M^\perp}$, is not valuative in $M$. In both cases, the final degree computation yields a valuative answer, which for us gave $t_M(x,y,z,w)$ and for \cite{ADH} gave $T_M(x,0)$, despite the fans not being valuative. This prevents one from automatically extending the final degree computation from the realizable case to all matroids.

However, in both cases, the valuativity of the final expression can be explained by instead working with the degrees $\alpha^i\beta^jc_k(\mathcal{S}_M^{\vee})c_\ell(\mathcal{Q}_M)$. Then, the valuativity follows from \Cref{prop:valpolys}.
In contrast to our equivariant methods, valuativity seems less accessible using non-equivariant methods, as evidenced by the intricate computations with the bipermutohedron in \cite[\S2 and \S4]{ADH} required to generalize the realizable case done in \cite{H15}.
\end{rem}

\subsection{Log-concavity for the Tutte polynomial}
\label{sec:Logconcprops}

In this subsection, we combine \Cref{prop:Segre} and \Cref{thm:refinedfanbiproj} with properties of ``Lefschetz fans'' established in \cite[\S5]{ADH} to prove \Cref{thm:masterlogconcintro}, reproduced below.
Recall that the coefficients of a homogeneous polynomial $f\in \RR[x_0, \ldots, x_N]$ of degree $d$ form a \textbf{log-concave unbroken array} if, for any $0\leq i < j \leq N$ and a monomial $x^{m}$ of degree $d'\leq d$, the coefficients of $\{x_i^kx_j^{d-d'-k}x^{m}\}_{0\leq k \leq d-d'}$ in $f$ form a nonnegative log-concave sequence with no internal zeros.

\newtheorem*{thm:masterlogconcintro}{\Cref{thm:masterlogconcintro}}
\begin{thm:masterlogconcintro}
For a matroid $M$ of rank $r$ with ground set $E$, the coefficients of the polynomial
\[
t_M(x,y,z,w) = (x+y)^{-1}(y+z)^{r}(x+w)^{|E|-r}T_M\Big(\frac{x+y}{y+z},\frac{x+y}{x+w}\Big)
\]
form a log-concave unbroken array.
\end{thm:masterlogconcintro}

We prepare by stating the tools we need from the tropical Hodge theory developed in \cite[\S5]{ADH}.  A \textbf{Lefschetz fan} is a pure-dimensional unimodular fan $\Sigma$, not necessarily complete, that is a balanced and satisfies certain Lefschetz properties \cite[Definition 1.5]{ADH}.  In our case, the properties of a Lefschetz fan we need are collected in the following proposition.

\begin{lem}\label{lem:lefschetzfan}
\
\begin{enumerate}[label = (\alph*)]
\item \label{lefschetzfan:support}\cite[Theorem 1.6]{ADH} If $\Sigma$ is a Lefschetz fan, then any unimodular fan whose support equals that of $\Sigma$ is a Lefschetz fan.
\item \label{lefschetzfan:product} \cite[Lemma 5.26]{ADH} A product of Lefschetz fans is a Lefschetz fan.
\item \label{lefschetzfan:Bergman} \cite[Theorem 8.9]{AdiHuhKatz} The Bergman fan of a loopless matroid is a Lefschetz fan.
\item \label{lefschetzfan:logconc} \cite[Theorem 5.28]{ADH} Let $\Sigma$ be an $m$-dimensional projective unimodular fan, and let $\Sigma'$ be a $d$-dimensional subfan that is a Lefschetz fan, which as a balanced fan defines the Chow class $[\Sigma']\in A^{m-d}(X_{\Sigma})$.  Then, for nef divisors $D_1, D_2, \ldots, D_k \in A^1(X_{\Sigma})$ on the projective toric variety $X_{\Sigma}$, the sequence $(a_0, \ldots, a_{d-(k-2)})$ defined by
\[
a_i = \int_{X_{\Sigma}} D_1^{i}D_2^{d-(k-2)-i}D_3 \cdots D_k \cdot [\Sigma'] \]
is a nonnegative sequence that is log-concave with no internal zeros.\footnote{\cite[Theorem 5.28]{ADH} does not state no internal zeros, but its proof implies that the sequence $(a_0, \ldots, a_{d-k-2})$ is a limit of log-concave positive sequences.  A limit of such sequences is a log-concave sequence with no internal zeros; see \cite[Lemma 34]{H12} for a proof.}
\end{enumerate}
\end{lem}

We can now prove a strengthening of \Cref{thm:masterlogconcintro}.

\begin{thm}\label{thm:logconcM1M2}
For matroids $M_1$ and $M_2$ on the common ground set $E$, the coefficients of the polynomial
\begin{align}\label{eq:Lorentzian2}\tag{$\star$}
\sum_{i+j+k+\ell=n} \Big( \int_{X_E} \alpha^i\beta^jc_k(\mathcal{S}_{M_1}^{\vee})c_\ell(\mathcal{Q}_{M_2}) \Big) x^iy^jz^kw^\ell
\end{align}
form a log-concave unbroken array.
\end{thm}

\begin{proof}
By \Cref{prop:Segre}, the polynomial \eqref{eq:Lorentzian2} is equal to
\[
\sum_{i+j+k+\ell = n} \Big(\int_{X_E \times (\PPinv^n)^\vee \times \PPinv^n} \alpha^i\beta^j\delta^{n-r_1+k}\epsilon^{r_2-1+\ell}[\mathbb{P}(\mathcal{Q}_{M_1}^{\vee})\times_{X_{E}}\mathbb{P}(\mathcal{S}_{M_2})] \Big) x^iy^jz^kw^\ell.
\]
In this expression, by \Cref{prop:noloopcoloop} we may assume that $M_1$ is coloopless and $M_2$ is loopless.  Then, by \Cref{thm:refinedfanbiproj} and the push-pull formula we have that the above equals
\[
\sum_{i+j+k+\ell=n}
    \left(\int_{X_{\Sigma}} \widetilde\phi^*\alpha ^i \widetilde\phi^*\beta^j \widetilde{\phi}^{*}\delta^{k+\crk_{M_1}-1}\widetilde{\phi}^{*}\epsilon^{\ell+\rk_{M_2}-1}[{\Sigma_{E,M_1^{\perp},M_2}}]\right)x^iy^jz^kw^\ell.
\]
That the coefficients of this polynomial form a log-concave unbroken array is now a result of \Cref{thm:refinedfanbiproj} and \Cref{lem:lefschetzfan} as follows.  
Because $\Sigma_E$, $\Sigma_{M_1^\perp}$, and $\Sigma_{M_2}$ are Lefschetz fans (\Cref{lem:lefschetzfan}.\ref{lefschetzfan:Bergman}), so is their product (\Cref{lem:lefschetzfan}.\ref{lefschetzfan:product}).  Since the product $\Sigma_E \times \Sigma_{M_1^\perp}\times \Sigma_{M_2}$ and the fan $\Sigma_{E,M_1^\perp,M_2}$ have the same support  (\Cref{thm:refinedfanbiproj}), the fan $\Sigma_{E,M_1^\perp,M_2}$ is also a Lefschetz fan (\Cref{lem:lefschetzfan}.\ref{lefschetzfan:support}).  Because the divisors $\widetilde\phi^*\alpha, \widetilde\phi^*\beta, \widetilde\phi^*\delta, \widetilde\phi^*\epsilon$ are nef divisors on $X_{\Sigma}$, being pullbacks of nef divisors,  that the coefficients of the polynomial form a log-concave unbroken array follows from \Cref{lem:lefschetzfan}.\ref{lefschetzfan:logconc}.
\end{proof}

\begin{proof}[Proof of \Cref{thm:masterlogconcintro}]
Set $M_1 = M_2 = M$ in \Cref{thm:logconcM1M2}, and apply \Cref{thm:4degintro}.
\end{proof}

\subsection{Denormalized Lorentzian polynomials}\label{subsec:lorentzian}
Let us note a strengthening of \Cref{thm:logconcM1M2} that we will only need in \S\ref{sec:flagmatroids} when we consider flag matroids.  The theorem is strengthened in two ways.

\medskip
First, we use the language of Lorentzian polynomials developed in \cite{BH20}.  For a homogeneous degree $d$ polynomial $f = \sum_{\bm \in \ZZ_{\geq0}^N} a_{\bm} x^{\bm} \in \RR[x_1, \ldots, x_N]$, its \textbf{normalization} is $N(f) = \sum_{\bm \in \ZZ_{\geq0}^{N}} a_{\bm} \frac{x^{\bm}}{\bm !}$ where $\bm ! = m_1! \cdots m_N!$.  The polynomial $f$ is said to be the denormalization of $N(f)$.  The polynomial $f$ is a \textbf{strictly Lorentzian polynomial} if every monomial of degree $d$ has positive coefficient in $f$ and every $(d-2)$-th coordinate partial derivative of $f$ is a quadric form with signature $(+,-,-, \ldots, -)$.  It is a \textbf{Lorentzian polynomial} if $f$ is a limit of strictly Lorentzian polynomials.
\cite[Example 2.26]{BH20} combined with \cite[Theorem 2.10]{BH20} implies that the coefficients of a denormalized Lorentzian polynomial form a log-concave unbroken array.
A minor modification of the proof of \cite[Theorem 4.6]{BH20} applied to the mixed Hodge-Riemann relations for Lefschetz fans \cite[Definition 5.6.(2)]{ADH} implies the following strengthening of \Cref{lem:lefschetzfan}.\ref{lefschetzfan:logconc}.

\begin{lem}\label{lem:Lorentzian}
Let $\Sigma$ be a $m$-dimensional projective unimodular fan, and let $\Sigma'$ be a $d$-dimensional subfan that is a Lefschetz fan, which as a balanced fan defines the Chow class $[\Sigma']\in A^{m-d}(X_{\Sigma})$.  Then, for nef divisors $D_1, D_2, \ldots, D_N \in A^1(X_{\Sigma})$ on the projective toric variety $X_{\Sigma}$, the polynomial $f\in \RR[x_1, \ldots, x_N]$ defined by
\[
f = \sum_{i_1 + \cdots + i_N = d} \left(\int_{X_{\Sigma}} D_1^{i_1} \cdots D_N^{i_N} \cdot [\Sigma']\right) x_1^{i_1}\cdots x_N^{i_N}
\]
is a denormalization of a Lorentzian polynomial.
\end{lem}

Second, we note that one can define multi-projectivization analogues of the definition of the biprojectivization class $[\PP(\cQ_{M_1}^\vee) \times_{X_E} \PP(\cS_{M_2})]$ for any number of $\cQ_{M_i}^\vee$'s and $\cS_{M_j}$'s.  Proofs of the multi-projectivization analogues of \Cref{prop:Segre}, \Cref{prop:noloopcoloop}, and \Cref{thm:refinedfanbiproj} are essentially identical to the proofs given for the biprojectivization case here.  As a result, we obtain the following.

\begin{thm}\label{thm:logconcMmany}
Let $M_1, \ldots, M_m$ and $M'_1, \ldots, M'_{m'}$ be matroids with the common ground set $E$.  Then, the polynomial defined by
\[
\sum \left(\int \alpha^i \beta^j c_{k_1}(\cS_{M_1}^\vee)\cdots c_{k_m}(\cS_{M_m}^\vee) c_{\ell_1}(\cQ_{M'_1})\cdots c_{\ell_{m'}}(\cQ_{M'_{m'}}) \right) x^iy^jz_1^{k_1}\cdots z_m^{k_m} w_1^{\ell_1}\cdots w_{m'}^{\ell_{m'}}
\]
is a denormalization of a Lorentzian polynomial, where $\int = \int_{X_E \times ((\PPinv^n)^\vee)^m \times (\PPinv^n)^{m'}}$ and the sum is over all $i+j + k_1 +\cdots +k_m + \ell_1 +\cdots +\ell_{m'} = n$.
\end{thm}

\begin{proof}
Similarly as in the proof of \Cref{thm:logconcM1M2}, the multi-projectivization analogues of \Cref{prop:Segre}, \Cref{prop:noloopcoloop}, and \Cref{thm:refinedfanbiproj} expresses the polynomial as a sum over intersection numbers of nef divisors multiplied to the Chow class of a Lefschetz fan $\Sigma_{E, M_1^\perp, \ldots, M_m^\perp, M_1', \ldots, M'_{m'}}$.  Then, one concludes by applying \Cref{lem:Lorentzian}.
\end{proof}

\section{A K-theory to Chow theory bridge}
\label{sec:KtoChow}

In this section, we develop a method special to permutohedral varieties that allows us to translate $K$-theoretic computations to Chow-theoretic computations, as stated in \Cref{thm:fakeHRRintro}.  We apply this method to the tautological $K$-classes of matroids to bridge the $K$-theoretic and the Chow-theoretic approach to studying matroid invariants.
Our method is a replacement, not a derivative, of the classical Hirzebruch-Riemann-Roch theorem, although the statement looks similar.  As before, let $E = \{0,1,\ldots, n\}$.

\subsection{A Hirzebruch-Riemann-Roch-type formula}

Recall that $\widetilde\Sigma_E$ denotes the normal fan in $\RR^E$ of the permutohedron, whose quotient fan in $\RR^E/\RR\mathbf1$ is the fan $\Sigma_E$ of the permutohedral variety $X_E$.  To state our GHRR-type formula for the permutohedral variety $X_E$, we recall from \Cref{thm:localization} that we had the identifications $K_0^T(X_E) \simeq PLaur(\widetilde\Sigma_E)$ and $A^\bullet_T(X_E) \simeq PPoly(\widetilde\Sigma_E)$ where
\[
\begin{split}
PLaur(\widetilde\Sigma_E) &= \text{the ring of piecewise Laurent polynomials in variables $T_0, \ldots, T_n$ on the fan $\widetilde\Sigma_E$,}\\
PPoly(\widetilde\Sigma_E) &= \text{the ring of piecewise polynomials in variables $t_0, \ldots, t_n$ on the fan $\widetilde\Sigma_E$.}
\end{split}
\]
Let $\chi^T\colon  K_T^0(X_E) \to K_T^0(\pt)$ be the $T$-equivariant pushforward map of $K$-rings (i.e.\ the $T$-equivariant sheaf Euler characteristic), and let $\int^T\colon  A^\bullet_T(X_E) \to A^{\bullet-n}_T(\pt)$ be the $T$-equivariant pushfoward map of Chow rings (i.e.\ the $T$-equivariant degree map).  We relate these two pushforward maps as follows.

\begin{thm}
\label{thm:fakeHRR}
Denote by $A^\bullet(X_E)[\prod(1+t_i)^{-1}]$ the ring obtained from $A_T^\bullet(X_E)$ by adjoining the inverse of the polynomial $(1+t_0)(1+t_1)\cdots(1+t_n)$.  Then,
the map $\zeta_{X_E}^T\colon K_0^T(X_{E})\to  {A_T^\bullet}(X_{E})[\prod(1+t_i)^{-1}]$ defined by sending
\[
f(T_0,\ldots,T_n)\mapsto f(t_0+1,\ldots,t_n+1) \quad\text{for $f$ a piecewise Laurent polynomial on $\widetilde\Sigma_E$}
\]
is a ring isomorphism, which descends to a ring isomorphism $\zeta_{X_E}\colon K^0(X_{E})\to A^\bullet(X_{E})$.  Moreover, the following diagrams commute
\begin{center}
\begin{tikzcd}
K_0^T(X_{E})\ar[r,"\zeta_{X_{E}}^T"]\ar[d,"\chi^T"'] & A_T^\bullet(X_{E})[\prod (1+t_i)^{-1}]\ar[d,"\int^T( c^T(S^{\vee}_{U_{n,E}}) \cdot - )"]
&&
K^0(X_{E})\ar[r,"\zeta_{X_{E}}"]\ar[d,"\chi"'] &A^\bullet(X_{E})\ar[d,"\int ((1+\alpha + \cdots + \alpha^n)\cdot -)"]\\
K_0^T(\pt)\ar[r,"\zeta_{\pt}^T"] & A_T^\bullet(\pt)[\prod (1+t_i)^{-1}]&&\mathbb{Z}\ar[r,equal]&\mathbb{Z}.
\end{tikzcd}
\end{center}
In particular, denoting by $\deg_\alpha\colon  A^\bullet(X_E)\to \ZZ$ the map $\xi \mapsto \int_{X_E} (1+\alpha+ \cdots + \alpha^n) \cdot \xi$, we have for a $K$-class $[\mathcal E]\in K^0(X_E)$ that
\[
\chi_{X_E}([\mathcal E]) = \deg_\alpha (\zeta_{X_{E}}[\mathcal E]).
\]
\end{thm}

For the proof of the commutativity of the diagrams, we will need the  Atiyah-Bott localization formulas for $K$-theory and Chow, which we specialize to permutohedral varieties using the identification of the torus action on the tangent spaces to the torus-fixed points of $X_E$  at the end of  \S\ref{subsec:permutohedralcone}.

\begin{thm}\label{thm:pushforward}
Let $X_E$ be the permutohedral variety.
\begin{enumerate}[label = (\alph*)]
\item \label{pushforward:K} \cite[4.7]{Nie74} The $T$-equivariant Euler characteristic $\chi^T$ satisfies
\[
\chi^T([\mathcal E]) = \sum_{\sigma\in \mathfrak S_{E}}\frac{[\mathcal E]_{\sigma}}{\left(1 - \frac{T_{\sigma(1)}}{T_{\sigma(0)}}\right) \cdots \left(1- \frac{T_{\sigma(n)}}{T_{\sigma(n-1)}}\right)}\in K_T^0(\pt)=\mathbb{Z}[T_0^{\pm},\ldots,T_n^{\pm}].
\]

\item \label{pushforward:Chow} \cite[Corollary~1]{edidingraham} The $T$-equivariant degree map $\int^T$ satisfies
\[
\int^T(\xi) =\sum_{\sigma\in \mathfrak S_{E}}\frac{\xi_{\sigma}}{(t_{\sigma(0)}-t_{\sigma(1)})\dots (t_{\sigma(n-1)}-t_{\sigma(n)})}\in A^\bullet_T(\pt)=\mathbb{Z}[t_0,\ldots,t_n].
\]
\end{enumerate}
\end{thm}

To obtain the non-equivariant Euler characteristic $\chi([\mathcal E])$ (resp.\ the non-equivariant degree $\int\xi$), one evaluates the respective sum in \Cref{thm:pushforward} at $T_0 = \cdots T_n = 1$ (resp.\ $t_0 = \cdots = t_n = 0$) after simplifying the expression to a Laurent polynomial (resp.\ a polynomial).  Implicit in the theorem is that such simplification always occurs, and that in the case of $\int^T$, if the sum is a rational function of degree less than 0, then the sum simplifies to zero.

\begin{proof}[Proof of \Cref{thm:fakeHRR}]
That $\zeta_{X_E}^T$ is a ring isomorphism is clear once we show that the map is well-defined.
Let $\sigma$ and $\sigma'$ be any maximal cones in $\widetilde \Sigma_E$ sharing a codimension 1 face.  The linear span of the face is $\{x\in \RR^E \mid x_i= x_j\}$ for some $i\neq j \in E$.
By definition, an element $f\in PLaur(\widetilde{\Sigma}_E) \simeq K_0^T(X_{E})$ satisfies $f_\sigma \equiv f_{\sigma'} \mod (T_i - T_j)$.  Since $(t_i +1) - (t_j + 1) = t_i - t_j$, its image $\zeta_{X_E}^T(f)$ also satisfies $\zeta_{X_E}^T(f)_\sigma \equiv \zeta_{X_E}^T(f)_{\sigma'} \mod (t_i - t_j)$, and hence is a well-defined element in $PPoly(X_E)[\prod(1+t_i)^{-1}] \simeq A^\bullet_T(X_E)[\prod(1+t_i)^{-1}]$.

We now show that the isomorphism $\zeta_{X_E}^T$ descends to an isomorphism of the non-equivariant rings.
We first recall from \Cref{thm:localization} that the kernels of the surjections to the non-equivariant rings are
\begin{equation}\label{eq:ideals}
\begin{split}
I_K = \ker(K^0_T(X_E)\twoheadrightarrow K^0(X_E)) &= \begin{matrix} \text{the ideal generated by $f(T_0, \ldots, T_n) - f(1,\ldots, 1)$}\\ \text{for $f$ a global Laurent polynomial,}\end{matrix}\\
I_A = \ker(A^\bullet_T(X_E)\twoheadrightarrow A^\bullet(X_E)) &= \begin{matrix} \text{the ideal generated by $f(t_0, \ldots, t_n) - f(0,\ldots, 0)$}\\ \text{for $f$ a global polynomial.}\end{matrix}
\end{split}
\end{equation}
Since the image of the global polynomial $\prod_{i\in E}(1+t_i) \in PPoly(\Sigma_E)$ under $A^\bullet_T(X_E) \twoheadrightarrow A^\bullet_T(X_E)/I_A \simeq A^\bullet(X_E)$ is 1, and in particular invertible,
the universal property of localization naturally induces a map $A_T^\bullet(X_{E})[\prod (1+t_i)^{-1}]\to A^\bullet(X_{E})$. As localization commutes with quotient, we have $A^\bullet_T(X_{E})[\prod(1+t_i)^{-1}]/I_A'=A^\bullet(X_{E})$, where $I_A'$ is the ideal $I_A[\prod (1+t_i)^{-1}]$.
We need to show that $\zeta_{X_{E}}^T(I_K)=I_A'$. Given a generator of $I_A'$, say $f(t_0,\ldots,t_n)-f(0,\ldots0)$ where $f \in \ZZ[t_0, \ldots, t_n]$ is a polynomial, we have $$(\zeta^T_{X_{E}})^{-1}(f(t_0,\ldots,t_n)-f(0,\ldots0))=f(T_0-1,\ldots,T_n-1)-f(0,\ldots, 0)\in I_K$$
since $f(T_0-1, \ldots, T_n-1)$ is a Laurent polynomial evaluating to $f(0, \ldots, 0)$ when $T_0 = \cdots T_n = 1$.
Thus, we have $(\zeta^T_{X_{E}})^{-1}(I_A')\subset I_K.$ Conversely, given a generator of $I_K$, say $f(T_0,\ldots,T_n)-f(1,\ldots,1)$ where $f\in \ZZ[T_0^\pm, \ldots, T_n^\pm]$ is a Laurent polynomial, we can write $f(T_0,\ldots,T_n)=(\prod_{i\in E} T_i^{-1})^m g(T_0,\ldots,T_n)$ for a polynomial $g(T_0,\ldots,T_n)$, so that
$$\zeta^T_{X_{E}}(f(t_0,\ldots,t_n)-f(1,\ldots,1))=( \textstyle\prod_{i\in E}(1+t_i)^{-1})^m(g(t_0+1,\ldots,t_n+1)-g(1,\ldots,1))\in I_A'.$$
Since $\prod_{i\in E} (1+t_i)^{-1}$ is a unit and $g(t_0+1,\ldots,t_n+1)-g(1,\ldots,1)$ is a generator of $I_A'$, we have $\zeta^T_{X_{E}}(I_K)=I_A'$.  We conclude that $\zeta_{X_{E}}^T(I_K)=I_A'$, and hence that the isomorphism $\zeta_{X_E}^T$ descends to an isomorphism of the respective quotients by the ideals, yielding $\zeta_{X_E}\colon  K^0(X_E) \to A^\bullet(X_E)$.

Finally, we now show that the two diagrams commute.  Let $f\in PLaur(\widetilde \Sigma_E) \simeq K_0^T(X_E)$.
Using \Cref{thm:pushforward}.\ref{pushforward:K}, we compute that
\[
(\zeta_{\pt}^T\circ \chi^T)(f) = \zeta_{\pt}^T \left( \sum_{\sigma\in \mathfrak S_E} \frac{f_\sigma(T_0, \ldots, T_n)}{(1-\frac{T_{\sigma(1)}}{T_{\sigma(0)}})\cdots (1-\frac{T_{\sigma(n)}}{T_{\sigma(n-1)}})} \right) = \sum_{\sigma\in \mathfrak S_E} \frac{f_\sigma(t_0+1,\ldots,t_n+1)}{\prod_{i=0}^{n-1} (1-\frac{1+t_{\sigma(i+1)}}{1+t_{\sigma(i)}})}.
\]
Note that for any permutation $\sigma\in \mathfrak S_E$, we have
\[
\frac{\prod_{i=0}^{n-1} (t_{\sigma(i)}-t_{\sigma(i+1)})}{\prod_{i=0}^{n-1} (1-\frac{1+t_{\sigma(i+1)}}{1+t_{\sigma(i)}})}=\prod_{i=0}^{n-1}(1+t_{\sigma(i)}) = c^T({\mathcal{S}^{\vee}_{U_{n,E}}})_\sigma, 
\]
where the last equality follows from $[\cS_{U_{n,E}}^\vee]_\sigma = \sum_{i=0}^{n-1} T_{\sigma(i)}$.
We thus compute further that
\[
\sum_{\sigma\in \mathfrak S_E} \frac{f_\sigma(t_0+1,\ldots,t_n+1)}{\prod_{i=0}^{n-1} (1-\frac{1+t_{\sigma(i+1)}}{1+t_{\sigma(i)}})} = \sum_{\sigma\in \mathfrak S_E} \frac{f_\sigma(t_0+1, \ldots, t_n+1)c^T(\cS_{U_{n,E}}^\vee)_\sigma}{\prod_{i=0}^{n-1} (t_{\sigma(i)}-t_{\sigma(i+1)})} = \sum_{\sigma\in \mathfrak S_E} \frac{(\zeta_{X_E}^T f)_\sigma \cdot c^T(\cS_{U_{n,E}}^\vee)_\sigma}{\prod_{i=0}^{n-1} (t_{\sigma(i)}-t_{\sigma(i+1)})},
\]
and by \Cref{thm:pushforward}.\ref{pushforward:Chow} we have
\[
\sum_{\sigma\in \mathfrak S_E} \frac{(\zeta_{X_E}^T f)_\sigma \cdot c^T(\cS_{U_{n,E}}^\vee)_\sigma}{\prod_{i=0}^{n-1} (t_{\sigma(i)}-t_{\sigma(i+1)})} = \int^T c^T(\cS_{U_{n,E}}^\vee) \cdot \zeta_{X_E}^T(f),
\]
showing that the left diagram of the theorem commutes.  The commutativity of the right diagram follows since non-equivariantly $c(\mathcal{S}^{\vee}_{U_{n,E}})= 1+\alpha+\alpha^2+\cdots + \alpha^n \in A^\bullet(X_E)$, as noted in \Cref{eg:alphabeta}.
\end{proof}

\begin{rem}\label{rem:ToddHard}
For a smooth projective variety $X$, the Hirzebruch-Riemann-Roch (HRR) theorem states that there exists an isomorphism $\ch$ called the Chern character map
\[
\ch\colon K_0(X)\otimes \mathbb{Q}\overset\sim\to A^\bullet(X)\otimes \mathbb{Q},
\]
and that there exists a Chow class $\Td(X) \in A^\bullet(X)_\QQ$ called the Todd class of $X$ such that the diagram
\begin{center}
\begin{tikzcd}
K_0(X)\otimes \mathbb{Q}\ar[r, "\ch", "\sim"']\ar[d, "\chi"']&A^\bullet(X)\otimes \mathbb{Q}\ar[d,"\int_X( \Td(X) \cdot -)"]\\
\mathbb{Z} \ar[r,equal]&\mathbb{Z}
\end{tikzcd}
\end{center}
commutes, where $\chi$ denotes the Euler characteristic.  See \cite[Chapter 14]{EH16} for an account of this.  We note that for the permutohedral variety $X_E$, the map $\zeta_{X_E}$ in \Cref{thm:fakeHRR} is different from $\ch$, and the class $1+\alpha+\cdots+\alpha^n \in A^\bullet(X_E)$ is not equal to the Todd class of $X_E$.
While the HRR theorem is a standard tool for translating between $K$-theory computations and Chow ring computations, sufficiently explicit descriptions of Todd classes, even in the case of permutohedral varieties, are often not available.  See \cite{CL20} for a study of Todd classes of permutohedral varieties in low dimensions.
\end{rem}

The map $\zeta_{X_E}$ of \Cref{thm:fakeHRR} behaves particularly well for a family of $K$-classes defined as follows.

\begin{defn}
\label{defn:simpleChernroots}
We say that a $T$-equivariant $K$-class $[\mathcal E]\in K_T^0(X_{E})$ \textbf{has simple Chern roots} if for each permutation $\sigma$ there is a sequence of non-negative integers $\bm_{\sigma} = (m_{\sigma,0},\ldots,m_{\sigma,n})\in \ZZ^E_{\geq 0}$ such that
\[
[\mathcal E]_\sigma= \bm_\sigma \cdot \bT = \sum_{i=1}^n m_{\sigma,i}T_i \in \ZZ[T_0^\pm, \ldots, T_n^\pm].
\]
\end{defn}

For instance, for any matroid $M$ the $K$-classes $[\mathcal S_M^\vee]$ and $[\mathcal Q_M^\vee]$ have simple Chern roots.
One can verify that the property of having simple Chern roots is closed under virtual extensions, subbundles, and quotients of $T$-equivariant $K$-classes, but we will not need this fact.  We now show how $\zeta_{X_E}$ behaves for $K$-classes with simple Chern roots.

\begin{prop}\label{prop:simpleChern}
Let $[\mathcal E] \in K_T^0(X_E)$ have simple Chern roots. Then, with $u$ a formal variable, we have
\begin{align*}
&\sum_{i\geq 0} \zeta_{X_{E}}([ \textstyle \bigwedge^i\mathcal{E}]) u^i=(u+1)^{\rk(\mathcal{E})}c(\mathcal{E},\frac{u}{u+1}), \quad\text{and}
\\
&\displaystyle \sum_{i\geq 0}\zeta_{X_{E}}([ \textstyle \bigwedge^i \mathcal{E}^{\vee}])u^i=(u+1)^{\rk(\mathcal{E})}c(\mathcal{E})^{-1}c(\mathcal{E},\frac{1}{u+1}).
\end{align*}
Equivalently, we have
\begin{align*}
&\sum_{i\geq 0} \zeta_{X_{E}}([ \textstyle \bigwedge^{\rk(\mathcal{E})-i}\mathcal{E}]) u^i=(u+1)^{\rk(\mathcal{E})}c(\mathcal{E},\frac{1}{u+1}), \quad\text{and}\\
&\displaystyle \sum_{i\geq 0}\zeta_{X_{E}}([ \textstyle \bigwedge^{\rk(\mathcal{E})-i}\mathcal{E}^{\vee}])u^i=(u+1)^{\rk(\mathcal{E})}c(\mathcal{E})^{-1}c(\mathcal{E},\frac{u}{u+1}).\end{align*}
\end{prop}

\begin{proof}
We prove more strongly the $T$-equivariant versions of the statements in terms of the isomorphism $\zeta_{X_E}^T\colon  K_T^0(X_E) \overset\sim\to A^\bullet_T(X_E)[\prod(1+t_i)^{-1}]$.
It follows from the definition that, for each permutation $\sigma\in \mathfrak S_E$, there is a multi-subset $I_{\sigma}$ of $E$ with size $|I_\sigma| = \operatorname{rk}(\mathcal E)$ such that $[\mathcal E]_{\sigma} = \sum_{i\in I_{\sigma}} T_i$ and $[\mathcal E^\vee]_{\sigma} = \sum_{i\in I_{\sigma}} T_i^{-1}$.
We now compute the exterior powers and the Chern classes as in \S\ref{subsec:prelimChernRoots}.  We have $\big(\sum_{i\geq 0} [ \textstyle \bigwedge^i\mathcal{E}] u^i \big)_\sigma = \prod_{i\in I_\sigma} (1+uT_i)$ and likewise $\big(\sum_{i\geq 0} [ \textstyle \bigwedge^i\mathcal{E}^\vee] u^i \big)_\sigma = \prod_{i\in I_\sigma} (1+uT_i^{-1})$ for every permutation $\sigma\in \mathfrak S_E$, and thus
\[
\begin{split}
\Big( \zeta_{X_{E}}^T \sum_{i\geq 0} ([ \textstyle \bigwedge^i\mathcal{E}]) u^i \Big)_\sigma =  \displaystyle \prod_{i\in I_{\sigma}}(1+u(1+t_i)) & =(u+1)^{|I_{\sigma}|} \prod_{i\in I_{\sigma}}(1+t_i \textstyle \frac{u}{u+1})\\
& = (u+1)^{\operatorname{rk}(E)} c^T(\mathcal E,  \textstyle \frac{u}{u+1})_\sigma,
\end{split}
\]
and
\[
\begin{split}
\Big( \zeta_{X_{E}}^T \sum_{i\geq 0} ([ \textstyle \bigwedge^i\mathcal{E}^\vee]) u^i \Big)_\sigma = \displaystyle \prod_{i\in I_{\sigma}}(1+u(1+t_i)^{-1}) & = \frac{(u+1)^{|I_{\sigma}|}}{\prod_{i\in I_{\sigma}} (1+t_i)}\prod_{i \in I_{\sigma}}(1+ t_i\textstyle \frac{1}{u+1})\\
& = (u+1)^{\rk(\mathcal{E})}c^T(\mathcal{E})_\sigma^{-1}c^T(\mathcal{E}, \textstyle\frac{1}{u+1})_\sigma,
\end{split}
\]
as desired.
\end{proof}

We note the following characterizing property of the isomorphism $\zeta_E$ of \Cref{thm:fakeHRR}.

\begin{cor}
The ring isomorphism $\zeta_{X_E}\colon  K(X_E) \to A^\bullet(X_E)$ is the unique ring homomorphism such that for any realization $L\subseteq \CC^E$ of a matroid, the $K$-class $[\mathcal O_{W_L}]$ of the structure sheaf of the wonderful compactification is sent to the Chow class $[W_L]$.
\end{cor}

\begin{proof}
That $W_L$ is the vanishing loci of a section $\mathcal O_{X_E} \to \cQ_L$ (\Cref{thm:BergmanGeom}) implies that we have the Koszul resolution
\[
0 \to \det \cQ_L^\vee \to \cdots \to  \textstyle\bigwedge^2 \cQ_L^\vee \to \cQ_L^\vee \to \mathcal O_{X_E} \to \mathcal O_{W_L}\to 0,
\]
and hence $[\mathcal O_{W_L}] = \sum_{i \geq 0} (-1)^i [\bigwedge^i\cQ_L^\vee]$.
Applying \Cref{prop:simpleChern} then yields $\zeta_{X_E}([\mathcal O_{W_L}]) = [W_L]$.  That this property characterizes $\zeta_{X_E}$ follows from \cite[Propositions 2.32 and 5.13]{Ham17}, which showed that Bergman classes of realizable loopless matroids, in fact those of loopless Schubert matroids, span $A^\bullet(X_E)$.
\end{proof}

\begin{rem}
\label{rem:polymatroid}
Let $P \subset \RR^E_{\geq 0}$ be a generalized permutohedron (defined in \S\ref{subsec:Tdiv}) that is contained in the nonnegative orthant.   Such $P$ defines a $T$-equivariant $K$-class $[\mathcal E_P]$ with simple Chern roots by
\[
[\mathcal E_P]_\sigma = \bm_\sigma \cdot \bT \quad\text{for $\sigma\in \mathfrak S_E$}
\]
where $\bm_\sigma$ is the vertex of $P$ maximizing the pairing $\langle - , v_0\be_{\sigma(0)} + \cdots +v_n\be_{\sigma(n)}\rangle$ for any $v_0 > \cdots > v_n$.
\Cref{thm:localization}.\ref{localization:K} implies that $[\mathcal E_P]$ is a well-defined element of $K_T^0(X_E)$ because each edge of such $P$ is parallel to $\be_i - \be_j$ for some $i\neq j \in E$, since the normal fan of $P$ coarsens $\widetilde\Sigma_{E}$.
For example, if $M$ is a matroid we have $[\cS_M^\vee] = [\mathcal E_{P(M)}]$.
This suggests that many results in this section may generalize from matroids to generalized permutohedra contained in the nonnegative orthant, often called discrete polymatroids.  A particular family of discrete polymatroids known as flag matroids are studied in \S\ref{sec:flagmatroids}.
\end{rem}

\subsection{Fink-Speyer's K-theoretic interpretation of Tutte polynomials}\label{subsec:FSTutte}

The authors of \cite{FS12} expressed the Tutte polynomial of a matroid via the $K$-theory of the Grassmannian.  To state their theorem, we need the following definition.

\begin{defn}\label{defn:yM}
For a matroid $M$ of rank $r$, let $y(M)\in K^0(\Gr(r;E))$ be the $K$-class determined by the following two properties:
\begin{enumerate}[label = (\roman*)]
\item If $L\subseteq \CC^E$ is a realization of $M$, then $y(M) = [\mathcal O_{\overline{T\cdot L}}]$, the $K$-class of the structure sheaf of the torus-orbit-closure in the Grassmannian, and
\item the assignment $M\mapsto y(M)$ is valuative.
\end{enumerate}
\end{defn}

The class $y(M)$ is well-defined because (i) for a realization $L\subseteq \CC^E$ of a matroid $M$ the $K$-class $[\mathcal O_{\overline{T\cdot L}}]$ only depends on the matroid $M$, and (ii) the assignment for realizable matroids $M\mapsto [\mathcal O_{\overline{T\cdot L}}]$ is valuative.  For a proof see \cite[Proposition A.5]{Spe09} and the remark following it.  Our definition of $y(M)$ here agrees with the definition of $y(M)$ given via the $T$-equivariant $K$-theory of the Grassmannian in \cite{FS12} because both satisfy the defining properties (i) and (ii) above.

\medskip
Fink and Speyer established the following $K$-theoretic interpretation of Tutte polynomials in \cite[Theorem 5.2]{FS12}.  Here, we show that applying the HRR-type formula (\Cref{thm:fakeHRR}) to our unifying Tutte formula \Cref{thm:4degintro} recovers this result.
Recall the notation that $\cS$ and $\cQ$ are the tautological subbundle and the quotient bundle, respectively, on the Grassmannian $\Gr(r;E)$.

\begin{thm}\label{thm:FSTutte}
Let $M$ be a matroid with ground set $E$, and $T_M(u,v)$ its Tutte polynomial.  Then,
\[
T_M(u,v) = \sum_{i=0}^r\sum_{j=0}^{|E|-r}\chi_{\Gr(r;E)}\big( y(M)[\det \mathcal{S}^{\vee}][ \textstyle\bigwedge^i\mathcal{S}][\bigwedge^j\mathcal{Q}^{\vee}] \big)(u-1)^i(v-1)^j.
\]
\end{thm}

We prepare the proof with a lemma, which allows one to translate certain Euler characteristic computations on Grassmannians to those on permutohedral varieties.

\begin{lem}\label{lem:yM}
Let $M$ be a matroid of rank $r$ with ground set $E$, and let $[\mathcal E]\in K_T^0(\Gr(r;E))$.  Then, for the class $[\mathcal E_M] \in K_T^0(X_E)$ defined in \Cref{prop:welldefinedgeneral}, we have an equality of non-equivariant Euler characteristics
\[
\chi_{\Gr(r;E)}\big( y(M)[\mathcal E] \big) = \chi_{X_E}\big( [\mathcal E_M] \big).
\]
\end{lem}

\begin{proof}
For a fixed $[\mathcal E]\in K_T^0(\Gr(r;E))$, we note that the assignment $[\mathcal E]\mapsto [\mathcal E_M]$ is valuative by \Cref{prop:valpolysgeneral}.
Since $M\mapsto y(M)$ is also valuative, by \Cref{lem:looplessval} it suffices to verify the desired equality for $M$ with a realization $L \subseteq \CC^E$.  The projection formula yields
\[
\chi_{\Gr(r;E)}\big( y(M)[\mathcal E] \big) = \chi_{\Gr(r;E)}\big( \mathcal O_{\overline{T\cdot L}} \cdot [\mathcal E] \big) = \chi_{\overline{T\cdot L}}\big([\mathcal E]|_{\overline{T\cdot L}} \big),
\]
As the normal fan of the base polytope $P(M)$ coarsens the fan $\Sigma_E$, the induced map of toric varieties $\psi\colon  X_E \to X_{P(M)} \simeq\overline{T\cdot L}$ satisfies $\psi_*\psi^* [\mathcal O_{\overline{T\cdot L}}] = [\mathcal O_{\overline{T\cdot L}}]$ by \cite[Theorem 9.2.5]{CLS11}, so using the projection formula and \Cref{prop:welldefinedgeneral} yields
$$\chi_{\overline{T\cdot L}}\big([\mathcal E]|_{\overline{T\cdot L}} \big)=\chi_{X_E} \big(\varphi_L^*[\mathcal E]\big)=\chi_{X_E} \big(\crem \varphi_L^*[\mathcal E]\big)=\chi_{X_E} \big([\mathcal E_M]\big)$$
as desired.
\end{proof}

\begin{proof}[Proof of \Cref{thm:FSTutte}]
By \Cref{lem:yM}, we have
\begin{multline*}
\sum_{i=0}^r\sum_{j=0}^{|E|-r}\chi_{\Gr(r;E)}\Big( y(M)[\det \mathcal{S}^{\vee}][ \textstyle\bigwedge^i\mathcal{S}][\bigwedge^j\mathcal{Q}^{\vee}] \Big)(u-1)^i(v-1)^j = \\
\displaystyle \sum_{i=0}^{r}  \sum_{j = 0}^{|E|-r}  \chi_{X_E} \Big([\det \mathcal S_M^\vee] [\textstyle{\bigwedge^{i}}\mathcal S_M] [\textstyle{\bigwedge^{j}}\mathcal Q_M^\vee]\Big) (u-1)^i(v-1)^j,
\end{multline*}
which by \Cref{thm:fakeHRR} equals
\[
\deg_\alpha \Big(\sum_{i,j}\zeta_{X_{E}} \big( [\det \mathcal S_M^\vee] [\textstyle{\bigwedge^{i}}\mathcal S_M] [\textstyle{\bigwedge^{j}}\mathcal Q_M^\vee] \big) (u-1)^i(v-1)^j \Big),
\]
where we recall the notation that $\deg_\alpha(-) = \int_{X_E} (1+\alpha+ \cdots + \alpha^n) \cdot (-)$.
We now claim that
$$\sum_{i,j}\zeta_{X_{E}}([\det \mathcal S_M^\vee] [\textstyle{\bigwedge^{i}}\mathcal S_M] [\textstyle{\bigwedge^{j}}\mathcal Q_M^\vee]) (u-1)^i(v-1)^j = u^{r}v^{|E|-r} c(\mathcal S^\vee_M, u^{-1}) c(\mathcal Q^\vee_M, 1-v^{-1}).$$
Noting $[\det \mathcal{S}_M^{\vee}][\bigwedge^i\mathcal{S}_M]=[\bigwedge^{r-i}\mathcal{S}_M^{\vee}]$, we write
$$\sum_{i,j} [\det\mathcal{S}_M^{\vee}][ \textstyle\bigwedge^i\mathcal{S}_M][\bigwedge^j\mathcal{Q}_M](u-1)^i(v-1)^j = \Big(\displaystyle\sum_i[ \textstyle\bigwedge^{r-i}\mathcal{S}_M^{\vee}](u-1)^i\Big) \textstyle\Big(\displaystyle\sum_j[ \textstyle \bigwedge^i\mathcal{Q}_M^{\vee}](v-1)^j\Big).$$
Because $\zeta_{X_{E}}$ is a ring homomorphism, and $\mathcal{S}_M^{\vee},\mathcal{Q}_M^{\vee}$ have simple Chern roots, applying \Cref{prop:simpleChern} thus verifies our claim.  Now, specializing the equality for the polynomial $t_M(x,y,z,w)$ in \Cref{thm:4degintro} at $x = 1$ and $y = 0$, we obtain
$$
\deg_{\alpha}(c_k(\mathcal{S}_M^{\vee},z)c_\ell(\mathcal{Q}_M,w))=
\sum_{i+k+\ell=n} \Big( \int_{X_E} \alpha^ic_k(\mathcal{S}_M^{\vee})c_\ell(\mathcal{Q}_M) \Big) z^kw^\ell=z^{r}(1+w)^{|E|-r}T_M( \textstyle \frac{1}{z},\frac{1}{1+w}).$$ Setting $u=z^{-1}$ and $v=(1+w)^{-1}$ in this final formula then yields
$$u^rv^{|E|-r}\deg_{\alpha}(c(\mathcal{S}_M^{\vee},u^{-1})c(\mathcal{Q}_M^{\vee},1-v^{-1}))=T_M(u,v),$$
and we thus conclude the desired formula for $T_M(u,v)$.
\end{proof}

\subsection{Cameron-Fink's lattice-point-counting interpretation of Tutte polynomials}
Let $M$ be a matroid with ground set $E$.
By using the relationship between toric geometry and Ehrhart-style lattice point counting \cite[Ch.~9]{CLS11}, and by recalling the fact that $[\mathcal O(D_{-P(M)})] = [\det \cS_M^\vee]$ from \Cref{eg:minusPM}, we have an equality of polynomials in $\QQ[t,u]$
\begin{align*}
\chi_{X_E}(\mathcal{O}(t\alpha+u\beta)[\det \mathcal{S}_M^{\vee}]) &= \text{the number of lattice points inside $-P(M)+t\Delta+u\nabla$}\\
&= \text{the number of lattice points inside $P(M)+t\nabla+u\Delta$},
\end{align*}
where $\Delta = \operatorname{Conv}(\be_i \mid i\in E)$ and $\nabla = -\Delta$ are the standard simplex and the negative standard simplex in $\RR^E$, respectively.
The authors of \cite{CF22} denoted this polynomial $Q_M(t,u)$.  With $\Psi\colon  \QQ[t,u] \to \QQ[x,y]$ defined as the invertible linear map sending $\binom{t}{i}\binom{u}{j} \mapsto x^iy^j$ for all $i,j\geq 0$, \cite[Theorem 3.2]{CF22} expressed the Tutte polynomial of $M$ in terms of the polynomial $Q'_M(x,y)\in \QQ[x,y]$ defined by
\[
Q'_M(x+1,y+1) = \Psi \big( Q_M(t,u) \big).
\]
We show that applying our HRR-type formula (\Cref{thm:fakeHRR}) to \Cref{thm:4degintro} recovers \cite[Theorem 3.2]{CF22}.  Combined with \Cref{thm:FSTutte}, which was also obtained by applying \Cref{thm:fakeHRR} to \Cref{thm:4degintro}, this answers a conjecture of Cameron and Fink (see the discussion after \cite[Theorem 3.4]{CF22}) on the relationship between their expression for the Tutte polynomial of $M$ and the $K$-theoretic computations in \cite{FS12}.

\begin{thm}\label{thm:CFTutte}
Let $M$ be a matroid of rank $r$ with ground set $E$, and $Q'_M(x,y)\in \QQ[x,y]$ be the polynomial as defined above.  Then we have
\[
Q'_M(x+1, y+1) = (x+y+1)^{-1} (y+1)^r (x+1)^{|E|-r} T_M \big(\frac{x+y+1}{y+1},\frac{x+y+1}{x+1}\big).
\]
Equivalently, letting $t_M$ be the polynomial in \Cref{thm:4degintro}, we have $Q'_M(x+1, y+1) = t_M(x+1, y, 1, 0)$.
\end{thm}

\begin{proof}
Recall from \Cref{eg:alphabeta} that $[\mathcal{O}(\alpha)]=[\mathcal{Q}_{U_{n,E}}]$ and $[\mathcal O(\beta)] = [\cS_{U_{1,E}}^\vee]$.
Setting $u=0$ in the last two lines of \Cref{prop:simpleChern}, one has that if $[\mathcal E]$ has simple Chern roots, then $\zeta_{X_E}([\det \mathcal E])= c(\mathcal E)$ and $\zeta_{X_E} ([\det \mathcal E^\vee]) = c(\mathcal E)^{-1}$.  In particular, since $\cS_M^\vee$ has simple Chern roots, we have
that $\zeta_{X_{E}}([\det \mathcal{S}_M^\vee])=c(\mathcal{S}_M^{\vee})$, and since $\cQ_{U_{n,E}}^\vee$ and $\cS_{U_{1,E}}^\vee$ have rank 1 (so taking $\det$ makes no change) and have simple Chern roots, we have that
\begin{align*}
&\zeta_{X_E}([\mathcal O(\alpha)]) = c(\cQ_{U_{n,E}}^\vee)^{-1}= (1-\alpha)^{-1}\quad\text{and} \\
&\zeta_{X_E}([\mathcal O(\beta)]) = c(\cS_{U_{1,E}}^\vee) = (1+\beta).
\end{align*}
\Cref{thm:fakeHRR} thus implies that
\begin{multline*}
Q_M(t,u) = \chi(\mathcal{O}(t\alpha+u\beta)[\det \mathcal{S}_M^{\vee}])\\
=\deg_{\alpha}((1-\alpha)^{-t}(\beta+1)^u c(\mathcal{S}_M^{\vee}))=\int_{X_E} (1-\alpha)^{-t-1}(1+\beta)^{u}c(\mathcal{S}_M^{\vee}).
\end{multline*}
Before applying the linear map $\Psi\colon  \QQ[t,u] \to \QQ[x,y]$, we note the following observations.  First, we have $(1-\alpha)^{-t-1} =  (1+\alpha+\alpha^2+\cdots)^{t+1} = \sum_{i\geq 0} \alpha^i \binom{t+i}{i}$.  Moreover, for each $i\geq 0$, the Vandermonde identity $\sum_{k = 0}^i \binom{i}{i-k}\binom{t}{k} = \binom{t+i}{i}$ implies that $\Psi\big(\binom{t+i}{i} \big) = (x+1)^i$.
Lastly, we have $(1+\beta)^u = \sum_{j\geq 0} \beta^j \binom{u}{j}$.  Thus, we conclude that
\begin{align*}
Q'_M(x+1,y+1) = \Psi(Q_M(t,u)) &= \Psi\left(\int_{X_E} (1-\alpha)^{-t-1}(1+\beta)^u c(\cS_M^\vee)\right)\\
&= \int_{X_E} \Big(\sum_{i\geq 0}\alpha^i(x+1)^i \Big) \Big(\sum_{j\geq 0} \beta^j y^j\Big) c(\cS_M^\vee)\\
&= t_M(x+1, y, 1, 0)
\end{align*}
where the last equality follows from \Cref{thm:4degintro}.
\end{proof}

\subsection{Ehrhart and volume polynomials of generalized permutohedra}
\label{subsec:EhrVol}
It was noted in \cite[Theorem 11.3]{P09} that the formula for the number of lattice points of a generalized permutohedra $P$ is obtained from the formula \cite[Theorem 9.3]{P09} for the volume of the Minkowski sum $P + \Delta$  by ``replacing powers with raised powers.''
Let us explain here how this phenomenon is another consequence of our \Cref{thm:fakeHRRintro}.

\medskip
For a nonempty subset $S\subseteq E$, let $\Delta_S = \operatorname{Conv}(\be_i \mid i\in S)\subset \RR^E$ be the $S$-standard simplex, which defines a divisor class $[D_{\Delta_S}] \in A^1(X_E)$ as denoted in \S\ref{subsec:Tdiv}.  It is known that the set of divisor classes $\{[D_{\Delta_S}] \mid \emptyset\subsetneq S\subseteq E\}$ form a basis of $A^1(X_E)$; see for instance \cite[\S3.3]{BES}.  In other words, for any generalized permutohedron $P\subset \RR^E$, there exist integers $y_S$ for $\emptyset\subsetneq S \subseteq E$ such that
\[
[D_P] = (y_E - 1)[D_{\Delta_E}] + \sum_{\emptyset\subsetneq S\subsetneq E} y_S [D_{\Delta_S}].
\]
We will soon see the convenience of using $(y_E-1)$ instead of $y_E$ for the coefficient of $[D_{\Delta_E}]$.
Since $\chi(\mathcal O(D_P)) = $ (the number of lattice points inside $P$), we wish to compute
\[
\chi\Big(\mathcal O(D_{\Delta_E})^{\otimes (y_E-1)} \otimes \bigotimes_{\emptyset\subsetneq S\subsetneq E} \mathcal O(D_{\Delta_S})^{\otimes y_S}\Big).
\]

To do so, let us define a matroid $H_S$ for each nonempty subset $S\subseteq E$ by $H_S = U_{|S|-1, S} \oplus U_{|E\setminus S|, E\setminus S}$.  Since $H_S$ has corank 1, similar computations as in \Cref{eg:alphabeta} shows that $[\cQ_{H_S}] = [\mathcal O(D_{\Delta_S})]$.  Then, by \Cref{prop:simpleChern} we have $\zeta_{X_E}([\mathcal O(D_{\Delta_S})]) = c(\cQ_{H_S}^\vee)^{-1} = (1 - [D_{\Delta_S}])^{-1}$.  Thus, applying \Cref{thm:fakeHRR} yields
\begin{align*}
\chi\big(\mathcal O(D_P) \big) &= \deg_\alpha\Big((1-[D_{\Delta_E}])^{-y_E +1} \prod_{\emptyset\subsetneq S\subsetneq E} (1-[D_{\Delta_S}])^{-y_S}\Big) \\
&= \int_{X_E} (1-[D_{\Delta_E}])^{-y_E} \prod_{\emptyset\subsetneq S\subsetneq E} (1-[D_{\Delta_S}])^{-y_S}\\
&= \int_{X_E}\prod_{\emptyset\subsetneq S\subseteq E} (1-[D_{\Delta_S}])^{-y_S}
\end{align*}
where  for the second to last equality we noted $\alpha = [D_{\Delta_E}]$ so that $(1+\alpha+ \cdots + \alpha^n) = (1-[D_{\Delta_E}])^{-1}$.  Let us define a map $\widehat\Psi$, related to but different from the previous map $\Psi$, by
\[
\widehat\Psi\colon  \QQ[y_S \mid \emptyset\subsetneq S\subseteq E] \to \QQ[\hat y_S \mid  \emptyset\subsetneq S\subseteq E]
\quad\text{where}\quad \prod_S \binom{y_S+i_S-1}{i_S} \mapsto \prod_S \hat y_S^{i_S}.
\]
Since $\binom{-y}{i} = (-1)^i\binom{y+i-1}{i}$, the map $\widehat \Psi$ is related to $\Psi$ by $\widehat\Psi = \Psi \circ \operatorname{neg}$ where $\operatorname{neg}\colon  \QQ[y_S \mid \emptyset\subsetneq S\subseteq E] \to \QQ[y_S \mid \emptyset\subsetneq S\subseteq E]$ is the involution defined by $\prod_S \binom{y_S}{i_S} \mapsto \prod_S (-1)^{i_S}\binom{-y_S}{i_S}$.
 Now, for a nilpotent ring element $D$, one has an identity $(1-D)^{-y} = \sum_{i\geq 0} (-1)^i\binom{-y}{i} D^i$ as polynomials in $y$.  Thus, we have that
\begin{align*}
\widehat\Psi\left( \int_{X_E} \prod_{\emptyset\subsetneq S\subseteq E} (1-[D_{\Delta_S}])^{-y_S} \right) 
= \int_{X_E} \prod_{\emptyset\subsetneq S\subseteq E} \Big(\sum_{i\geq 0}[D_{\Delta_S}]^i \hat y_S^i\Big).
\end{align*}
Let $V = \int_{X_E} \prod_{\emptyset\subsetneq S\subseteq E} \Big(\sum_{i\geq 0}[D_{\Delta_S}]^i \hat y_S^i\Big) \in \QQ[\hat y_S \mid \emptyset\subsetneq S\subseteq E]$ be the polynomial above.  The normalization (as defined in \S\ref{subsec:lorentzian}) of the polynomial $V$ is
\[
N(V) = \int_{X_E} \prod_{\emptyset\subsetneq S\subseteq E} \Big(\sum_{i\geq 0}[D_{\Delta_S}]^i  \frac{\hat y_S^i}{i!}\Big) = \frac{1}{n!}\int_{X_E}\Big( \sum_{\emptyset\subsetneq S\subseteq E} \hat y_S [D_{\Delta_S}]\Big)^n,
\]
which is the Euclidean volume of the polytope $P + \Delta_E$ when one sets $\hat y_S = y_S$ for all $S\subseteq E$.  A formula for $N(V)$ is given in \cite[Theorem 9.3]{P09}.  See \cite[Theorem 5.2.4]{BES} for a generalization and a proof via matroid theory.  The operation of ``replacing powers with raising powers'' defined in \cite{P09} sends $\hat y^a/a!$ to $\binom{y+a-1}{a}$, which is the composition $\widehat\Psi^{-1} \circ N^{-1}$.  Thus,
we have recovered \cite[Theorem 11.3]{P09}, which stated that applying the operation of ``replacing powers with raising powers'' to the Euclidean volume polynomial $N(V)$ gives the polynomial that measures the number of lattice points of $P$.

\subsection{Speyer's g-polynomial of a matroid}
Speyer defined the $g$-polynomial of a (loopless and coloopless) $\CC$-realizable matroid in \cite[\S3]{Spe09} via the $K$-theory of Grassmannians, and used the invariant to give bounds on the complexity of matroid polytope subdivisions.  The $g$-polynomial of an arbitrary (loopless and coloopless) matroid was later defined in \cite[Remark 6.4]{FS12}, and it remains open whether the positivity property of the $g$-polynomial for $\CC$-realizable matroids persists for arbitrary matroids.

\medskip
We give a Chow-theoretic formula for the $g$-polynomial.
Our formula below proves \cite[Conjecture 1]{LdMRS20} (and corrects for the missing global sign depending on the number of connected components).

\begin{thm}\label{thm:gPoly}
Let $M$ be a loopless and coloopless matroid of rank $r$ on ground set $E$.  Let $\operatorname{comp}(M)$ be the number of connected components of $M$. Then we have
\[
g_M(s) = (-1)^{\operatorname{comp}(M)}\sum_{i \geq 0} \deg_\alpha \Big( c(\mathcal Q_M^\vee) c_{r-i}(\mathcal S_M^\vee) c_{|E|-r}(\mathcal Q_M)\Big) (-s)^i.
\]
\end{thm}

\begin{proof}
\cite[Theorem 6.5]{FS12} states that $(-1)^{\operatorname{comp}(M)}g_M(-s) = H_M(s)$ where
\[
H_M(x+y-xy) = \sum_{i=0}^{r} \sum_{j=0}^{|E|-r} \chi_{Gr(r;E)} \Big( y(M) [ \textstyle{\bigwedge^i}\mathcal S ] [ \textstyle{\bigwedge^j}\mathcal Q^\vee] \Big) (x-1)^i(y-1)^j.
\]
Combining \Cref{lem:yM}, \Cref{thm:fakeHRR}, and \Cref{prop:simpleChern} yields
\begin{multline}\label{eq:gPoly}
\sum_{i=0}^{r} \sum_{j=0}^{|E|-r} \chi_{Gr(r;E)} \Big( y(M) [ \textstyle{\bigwedge^i}\mathcal S ] [ \textstyle{\bigwedge^j}\mathcal Q^\vee] \Big) (x-1)^i(y-1)^j = \\
\sum_{i=0}^{r} \sum_{j=0}^{|E|-r} \deg_\alpha\Big( c(\mathcal Q_M^\vee) c_i(\mathcal S_M^\vee) c_j(\mathcal Q_M^\vee) \Big) x^{r-i}y^{|E|-r-j}(y-1)^j.
\end{multline}
Here we have used that $c(\mathcal Q_M^\vee) = c(\mathcal S_M^\vee)^{-1}$, which follows from $[\mathcal S_M^\vee]+[ \mathcal Q_M^\vee] = [\underline\CC^E]$.

Now, note that the coefficient of $s^i$ in the polynomial $H_M(s)$ is the same as the coefficient of $x^i$ in $H_M(x+y-xy)$.  To get the the coefficient of $x^i$ in the right-hand-side of \eqref{eq:gPoly}, we set $y=0$ and obtain the coefficient of $x^i$ to be
\[
(-1)^{|E|-r} \deg_\alpha \Big( c(\mathcal Q_M^\vee) c_{r-i}(\mathcal S_M^\vee) c_{|E|-r}(\mathcal Q_M^\vee) \Big) = \deg_\alpha \Big( c(\mathcal Q_M^\vee) c_{r-i}(\mathcal S_M^\vee) c_{|E|-r}(\mathcal Q_M) \Big).
\]
The desired formula for $g_M(s)$ follows.
\end{proof}

A similar formula holds for the generalization of the $g$-polynomial of a matroid to that of a matroid morphism (see \cite[\S7.1]{DES21}), but we don't include the details here.

\section{Flag matroids}
\label{sec:flagmatroids}

Flag matroids generalize matroids, just as partial flag varieties generalize Grassmannians.  Tutte polynomials of matroids were generalized to those of flag matroids via the $K$-theory of partial flag varieties in \cite{CDMS18} and \cite{DES21}.  In this section, use the tools developed from our framework of tautological $K$-classes of matroids to give a Chow-theoretic formula for such generalizations of Tutte polynomials.  As a result, we answer \cite[Conjecture 7.7]{DES21}, and establish a log-concavity property for characteristic polynomials of flag matroids, positively answering \cite[Conjecture 9.4]{CDMS18}.

\subsection{Flag matroids}
We review flag matroids, and extend few results in \Cref{subsec:FSTutte} about the $K$-theory of Grassmannians to that of partial flag varieties.  We omit the proofs as they only require minor changes from proofs of analogous statements for matroids.

\begin{defn}
A \textbf{flag matroid} of rank $\boldsymbol r = (r_1, \ldots, r_k)$ is a sequence $\MM = (M_1, \ldots, M_k)$ of matroids of ranks $\boldsymbol r$ on a common ground set $E$ such that for all $i = 1, \ldots, k-1$, any flat of $M_i$ is a flat of $M_{i+1}$.  A \textbf{realization} of a flag matroid $\MM$ is a flag $\boldsymbol L\colon  L_1\subseteq \cdots \subseteq L_k \subseteq \CC^E$ of linear subspaces such that $L_i$ is a realization of $M_i$ for all $i = 1, \ldots, k$.
\end{defn}

More generally, replacing partial flag varieties by generalized flag varieties of arbitrary finite Coxeter type gives rise to Coxeter matroids, introduced in \cite{GS87a, GS87b}.  See \cite{BGW03} for a treatment.  Flag matroids are the Coxeter matroids of type $A$ in this framework.  When the flag matroid has only two constituents $(M_1, M_2)$, it is often called a \textbf{morphism of matroids}, denoted $M_1 \twoheadleftarrow M_2$.  See \cite{EH20} and references therein for a slightly more general definition of morphism of matroids.

\medskip
The relation between matroids and the $K$-theory of Grassmannians generalize in the following way.  See \cite{CDMS18} for a survey with proofs.  For a sequence of nonnegative integers $\boldsymbol r\colon  0\leq r_1 \leq \cdots \leq r_k \leq |E|$, let $\Fl(\boldsymbol r;E)$ be the partial flag variety consisting of flags of linear subspaces of the respective dimensions in $\CC^E$.  The torus $T$ acts on $\Fl(\boldsymbol r;E)$ via its standard action on $\CC^E$.
For a realization $\boldsymbol L\in \Fl(\boldsymbol r;E)$ of a flag matroid $\MM$, we have that the torus-orbit-closure $\overline{T\cdot \boldsymbol L}$ is isomorphic to $X_{P(\MM)}$, where $P(\MM)$ is the  \textbf{base polytope} of the flag matroid $\MM = (M_1, \ldots, M_k)$ defined as the Minkowski sum $\sum_{i=1}^k P(M_i)$ of the base polytopes of its constituent matroids.  One has a commuting diagram
\[
\begin{tikzcd}
&X_E \ar[r] \ar[rd, "\varphi_{\boldsymbol L}"'] &X_{P(\MM)}\simeq \overline{T\cdot \boldsymbol L} \ar[d, hook] \ar[r] &\prod_{i=1}^k X_{P(M_i)} \simeq \prod_{i=1}^k \overline{T\cdot L_i}\ar[d,hook]\\
& & \Fl(\boldsymbol r;E) \ar[r, "\prod_i \pi_i"] &\prod_{i=1}^k \Gr(r_i;E).
\end{tikzcd}
\]
Thus, for a class $[\mathcal E^{(i)}] \in K_T^0(\Gr(r_i;E))$, the class $[\mathcal E^{(i)}_{M_i}]\in K_T^0(X_E)$ defined via \Cref{prop:welldefinedgeneral} coincides with the class $\crem (\varphi_{\boldsymbol L} \circ \pi_i)^* [\mathcal E^{(i)}]$.
The notion of valuativity generalizes to flag matroids.  See \cite{ESS21} for an in-depth study of valuativity for Coxeter matroids in general, and see \cite{BEZ21} for a study of subdivisions of base polytopes of flag matroids.

\medskip
Similar to the class $y(M)$ of a matroid $M$ in \Cref{defn:yM}, one can also define a $K$-class $y(\MM)$ in $K^0(\Fl(\boldsymbol r;E))$ of a flag matroid $\MM$ by the following two determining properties: (i) If $\boldsymbol L$ is a realization of $\MM$ then $y(\MM) = [\mathcal O_{\overline{T\cdot \boldsymbol L}}]$, and (ii) the assignment $\MM\mapsto y(\MM)$ is valuative.
Its well-definedness follows from \cite[Equation (8.7)]{CDMS18} and \cite[Remark 2.11]{DES21}.  See \cite[Definition 8.19]{CDMS18} for a definition via the $T$-equivariant $K$-theory of $\Fl(\boldsymbol r;E)$.
The class $y(\MM)$ satisfies the following analogue of \Cref{lem:yM}, whose proof is similar\footnote{One subtlety is that \Cref{lem:looplessval} does not generalize easily to flag matroids, but this is remedied by \cite[Corollary 3.16]{ESS21}.  Alternatively, one can prove both this lemma and the original one (\Cref{lem:yM}) by using the Atiyah-Bott localization formula (\Cref{thm:pushforward}.\ref{pushforward:K}) combined with a generalized form \cite[Theorem 2.3]{Ish90} of Brion's formula \cite{Bri88}.}.

\begin{lem}\label{lem:yMM}
Let $[\mathcal E^{(i)}]$ be a class in $K_T^0(\Gr(r_i;E))$ for each $i = 1, \ldots, k$, and let $[\mathcal E^{(i)}_{M_i}] \in K_T^0(X_E)$ be as defined in \Cref{prop:welldefinedgeneral}.  For a flag matroid $\MM = (M_1, \ldots, M_k)$ on $E$ with ranks $\boldsymbol r$, we have
\[
\chi_{\Fl(\boldsymbol r;E)}\Big( y(\MM) \cdot \prod_{i=1}^k \pi_i^*[\mathcal E^{(i)}]\Big) = \chi_{X_E} \Big(\prod_{i=1}^k [\mathcal E^{(i)}_{M_i}]\Big).
\]
\end{lem}

\subsection{Flag-geometric Tutte polynomials of flag matroids}
Generalizing the $K$-theoretic interpretation of Tutte polynomials of matroids in \cite{FS12}, 
the authors of \cite{CDMS18} and \cite{DES21} defined the \textbf{flag-geometric Tutte polynomials} of flag matroids.
We give a Chow-theoretic formula for the flag-geometric Tutte polynomial of a matroid.
Recall the shorthand that $\deg_\alpha(\xi) = \int_{X_E} (1 + \alpha + \cdots +\alpha^n) \cdot \xi$.

\begin{thm}\label{thm:FGT}
Let $\MM = (M_1, \ldots, M_k)$ be a flag matroid on ground set $E$ with rank sequence $(r_1, \ldots, r_k)$.  Then, the flag-geometric Tutte polynomial $KT_{\MM}(x,y)$ of $ \MM$, as defined in \cite[Definition 8.23]{CDMS18}, satisfies
\[
KT_{\MM}(x,y) = \sum_{i=0}^{r_k}\sum_{j=0}^{|E|-r_1} \deg_\alpha\Big( c(\cS_{M_1}^\vee) \cdots c(\cS_{M_{k-1}}^\vee) c_i(\cS_{M_k}^\vee)c_j(\cQ_{M_1})\Big) x^{r_k-i}y^{|E|-r_1-j}(1-y)^j
\]
In particular, the \textbf{flag-geometric characteristic polynomial} of $\MM$, defined in \cite[\S9]{CDMS18} as
\[
K\chi_\MM(q) \coloneqq (-1)^{r_1+\cdots + r_k}KT_\MM(1-q,0),
\]
satisfies
\[
(-1)^{r_1+\cdots+r_k}K\chi_\MM(q) = \sum_{i=0}^{r_k} \deg_\alpha\Big( c(\cS_{M_1}^\vee) \cdots c(\cS_{M_{k-1}}^\vee) c_i(\cS_{M_k}^\vee)c_{|E|-r_1}(\cQ_{M_1})\Big) (1-q)^{r_k-i},
\]
and its coefficients have alternating signs.
\end{thm}

\begin{proof}
By \cite[\S6.1]{DES21}, the flag-geometric Tutte polynomial $KT_{\MM}(x,y)$ is given by
\[
KT_{\MM}(u+1,v+1) = \sum_{i,j} \chi_{Fl(\boldsymbol r;E)} \Big(y(\MM) [\det \mathcal S_1^\vee]\cdots [\det \mathcal S_k^\vee][ \textstyle{\bigwedge^i} \mathcal S_k][ \textstyle{\bigwedge^j}\mathcal Q_1^\vee]\Big) u^iv^j,
\]
where $\cS_\ell$ and $\cQ_\ell$ denotes the tautological bundles on $\Gr(r_\ell; E)$ pulled back to $Fl(\boldsymbol r;E)$ for $\ell = 1, \ldots, k$.  By \Cref{lem:yMM}, this equals
\[
\sum_{i,j} \chi_{X_E}\Big( [\det \cS_{M_1}^\vee] \cdots [\det\cS_{M_k}^\vee][\textstyle \bigwedge^i\cS_{M_k}][\bigwedge^j \cQ_{M_1}^\vee] \Big) u^iv^j.
\]
Noting that $[\det \cS_{M_k}^\vee][\bigwedge^i \cS_{M_k}] = [\bigwedge^{\rk(\cS_{M_k}^\vee) - i} \cS_{M_k}^\vee]$, combining \Cref{thm:fakeHRR} and \Cref{prop:simpleChern} yields the desired equalities for $KT_{\MM}$ and $K\chi_{\MM}$.  Because $KT_{\MM}(1+q,0)$ has all nonnegative coefficients by \Cref{thm:logconcMmany}, that the coefficients of $K\chi_{\MM}(q)$ have alternating signs follows.
\end{proof}

We now resolve \cite[Conjecture 9.4]{CDMS18}, which stated that the flag-geometric characteristic polynomial of $\MM$ form a log-concave sequence.

\begin{cor}
For a flag matroid $\MM = (M_1, \ldots, M_k)$, the (unsigned) coefficients of $K\chi_\MM(q)$ form a log-concave sequence with no internal zeros.
\end{cor}

\begin{proof}
As coefficients of $KT_\MM(1-q,0)$ have alternating signs, we show the stronger statement that the coefficients of $KT_\MM(q,0)$ are log-concave with no internal zeros (see \cite{Dawson} where this reduction is proved in the context of showing that $h$-vector log concavity implies $f$-vector log concavity).  By \Cref{thm:FGT}, the homogenization of $KT_\MM(q,0)$ by an additional variable $p$ is written as
\begin{equation}\label{eq:flagLorent}
\sum_{i=0}^{r_k} \Big(\int_{X_E} c(\cS_{U_{n,E}}^\vee)c( \cS_{M_1}^\vee) \cdots c(\cS_{M_{k-1}}^\vee) c_i(\cS_{M_k}^\vee)c_{|E|-r_1}(\cQ_{M_1})\Big) p^iq^{r_k-i},
\end{equation}
since $c(\cS_{U_{n,E}}^\vee) = 1 + \alpha + \cdots + \alpha^n$ from \Cref{eg:alphabeta}.
This homogeneous polynomial is obtained by setting $q_0 = \cdots = q_{k-1} = q$ in $(\frac{\partial}{\partial u}^{|E|-r_1}f)|_{u = 0}$ where $f$ is the polynomial
\[
\sum_{i=0}^{r_k} \Big(\int_{X_E} c(\cS_{U_{n,E}}^\vee,q_0)c( \cS_{M_1}^\vee,q_1) \cdots c(\cS_{M_{k-1}}^\vee,q_{k-1}) c(\cS_{M_k}^\vee,p)c(\cQ_{M_1},u)\Big).
\]
\Cref{thm:logconcMmany} implies that $f$ above is a denormalized Lorentzian polynomial. As taking partials and evaluating at zero preserves Lorentzian polynomials (\cite[Theorem 2.25]{BH20} and \cite[Theorem 2.10]{BH20}), and since setting variables equal to each other preserves denormalized Lorentzian polynomials \cite[Lemma 4.8]{BLP20}, we conclude that the polynomial in \eqref{eq:flagLorent} is a denormalized Lorentzian polynomial.  We thereby conclude that its coefficients form a log-concave nonnegative sequence with no internal zeros.
\end{proof}

We also resolve \cite[Conjecture 7.7]{DES21}.

\begin{cor}
Let $M$ be a loopless matroid of rank $r$ on $E = \{0,1,\ldots, n\}$, so that one has $U_{1,E} \twoheadleftarrow M$.  Then we have
\[
KT_{U_{1,E},M}(x,0) = x^r.
\]
\end{cor}

\begin{proof}
By \Cref{thm:FGT}, we have
\[
KT_{U_{1,E}, M}(x,0)  =
\sum_{i=0}^{r} \deg_\alpha \Big( c(\mathcal S_{U_{1,E}}^\vee) c_i(\mathcal S_{M}^\vee) c_{n} (\mathcal Q_{U_{1,E}}^\vee) \Big) x^{r - i}(-1)^n
\]
since all other terms from $j \neq |E| - 1 = n$ vanish as we set $y = 0$ in $KT_{U_{1,E},M}(x,y)$.  Moreover, since the dimension of $X_{E}$ is $n$, the only term that survives is from $i = 0$.  Noting that $c_n(\mathcal Q_{U_{1,E}}^\vee) = (-\alpha)^n$ from \Cref{eg:alphabeta}, we have $KT_{U_{1,E},M}(x,0) = x^r$ as desired.
\end{proof}

\subsection{Las Vergnas Tutte polynomials of morphisms of matroids}
We now turn to Las Vergnas' Tutte polynomials of morphisms of matroids, which is different from the flag-geometric Tutte polynomials.  See \cite{DES21} for a geometric origin of the difference between the two generalizations.  Las Vergnas introduced the following generalization of the Tutte polynomial to morphisms of matroids in \cite{LV75}.  See \cite{LV80} for a survey of its properties.

\begin{defn}
Let $M_1$ and $M_2$ be matroids of rank $r_1$ and $r_2$ (respectively) on ground set $E$ such that $M_1 \twoheadleftarrow M_2$.  The \textbf{Las Vergnas Tutte polynomial} of $(M_1, M_2)$ is a polynomial in three variables $x,y,z$ defined by
\[
LVT_{M_1, M_2}(x,y,z) \coloneqq \sum_{A \subseteq E} (x-1)^{r_1 - \operatorname{rk}_{M_1}(A)}(y-1)^{|A| - \operatorname{rk}_{M_2}(A)} z^{r_2 - r_1 - \operatorname{rk}_{M_2}(A) + \operatorname{rk}_{M_1}(A))}.
\]
\end{defn}

To express $LVT_{M_1,M_2}$ Chow-theoretically, it is convenient to have the following notation.
Let $\mathcal S_{M_1,M_2}^\vee$ be a $K$-class on $X_{E}$ whose equivariant $K$-class $[\mathcal S_{M_1,M_2}^\vee]^T$ is defined by
\[
[\mathcal S_{M_1,M_2}^\vee]^T_\sigma = \sum_{i \in B_\sigma(M_2) \setminus B_\sigma(M_1)} T_i = [\cS_{M_2}^\vee]_\sigma - [\cS_{M_1}^\vee]_\sigma.
\]
When $(M_1, M_2)$ has a realization $(L_1, L_2)$, it is equal to the pullback $\varphi_{(L_1, L_2)}^* (\mathcal S_2/\mathcal S_1)^\vee$, where $\mathcal S_2/\mathcal S_1$ is the quotient of the two tautological subbundles on $Fl(r_1,r_2;E)$.

\begin{thm}\label{thm:LVT}
Let $M_1$ and $M_2$ be matroids of rank $r_1$ and $r_2$ (respectively) on ground set $E$ such that $M_1 \twoheadleftarrow M_2$.
The Las Vergnas Tutte polynomial of the matroid morphism $M_1 \twoheadleftarrow M_2$ satisfies
\begin{multline*}
LVT_{M_1, M_2}(x,y,z)  =\\
\sum_{i=0}^{r_1}\sum_{j=0}^{|E|-r_2}\sum_{k=0}^{r_2 - r_1} \deg_\alpha \Big( c_i(\mathcal S_{M_1}^\vee) c_j(\mathcal Q_{M_2}^\vee) c_k(\mathcal S_{M_1,M_2}^\vee)\Big) x^{r_1 - i}y^{|E|-r_2-j}(y-1)^j (z+1)^{r_2 - r_1 - k}.
\end{multline*}
\end{thm}

\begin{proof}
The partial flag variety $Fl(r_1, r_2; E)$ has tautological subbundles $\mathcal S_1$ and $\mathcal S_2$ with ranks $r_1$ and $r_2$ (respectively), and corresponding quotient bundles $\mathcal Q_1$ and $\mathcal Q_2$.  We also have the short exact sequence $0 \to \mathcal S_1 \to \mathcal S_2 \to \mathcal S_2/\mathcal S_1 \to 0$.
It was shown in \cite[Theorem 5.3]{DES21} that
\[
LVT_{M_1,M_2}(u+1, v+1, w) = \sum_{i,j,k} \chi_{Fl(r_1,r_2;E)} \Big( y(\MM)  [\det \mathcal S_2^\vee][\textstyle{\bigwedge^i} \mathcal S_1][\textstyle{\bigwedge^j} \mathcal Q_2^\vee][\textstyle{\bigwedge^k} (\mathcal S_2/\mathcal S_1)]\Big) u^iv^jw^k,
\]
where $\MM = (M_1, M_2)$ denotes the two-step flag matroid.
Now, applying \Cref{lem:yMM} while noting that $\det \mathcal S_2^\vee \simeq \det \mathcal S_1^\vee \otimes \det(\mathcal S_2/\mathcal S_1)^\vee$, gives
\[
LVT_{M_1,M_2}(u+1, v+1, w) = \sum_{i,j,k} \chi_{X_{A_E}}\Big([\textstyle{\bigwedge^{r_1 - i}}\mathcal S_{M_1}^\vee] [\textstyle{\bigwedge^{j}}\mathcal Q_{M_2}^\vee][\textstyle{\bigwedge^{r_2 - r_1 - k}} \mathcal S_{M_1, M_2}^\vee] \Big)u^iv^jw^k.
\]
Applying \Cref{prop:simpleChern} then yields
\begin{multline*}
LVT_{M_1,M_2}(u+1, v+1, w)  = \\
(u+1)^{r_1} (v+1)^{|E|-r_2}(w+1)^{r_2 - r_1} \deg_\alpha\Big( c(\mathcal S_{M_1}^\vee, \textstyle{\frac{1}{u+1}}) c(\mathcal Q_{M_2}^\vee,  \textstyle{\frac{v}{v+1}}) c(\mathcal S_{M_1,M_2}^\vee,  \textstyle{\frac{1}{w+1}}) \Big).
\end{multline*}
Substituting $u = x-1$, $v = y-1$, and $w = z$ then yields the desired equality.
\end{proof}

We remark that, despite this Chow-theoretic formula, \cite[Conjecture 7.10]{DES21} concerning the log-concavity of a specialization of the Las Vergnas Tutte polynomial remains open, since $\cS_{M_1,M_2}^\vee$ is in general not the $K$-class of a nef vector bundle 
even when $(M_1, M_2)$ has a realization.

\appendix
\renewcommand{\thesection}{\Roman{section}}

\section{Alternate proof of \Cref{thm:4degintro} via convolution formulas}
\label{sec:convolution}
We give another proof for \Cref{thm:4degintro}, different from the proof in \S\ref{sec:unifyingTutte}, by using the base polytope properties of the tautological classes established in \S\ref{sec:basepolytopeproperties}.
Instead of establishing a deletion-contraction relation, we establish a recursive convolution formula for $\alpha^i\beta^jc_k(\mathcal{S}_M^{\vee})c_\ell(\mathcal{Q}_M)$, and show that it agrees with a new Tutte polynomial convolution formula whose proof was communicated to us by Alex Fink.
As before, let $E = \{0,1,\ldots, n\}$, and $X_E$ the $n$-dimensional permutohedral variety. Important for us will be the following well-known formula, called the corank-nullity formula, for the Tutte polynomial of a matroid $M$ of rank $r$
\[
T_M(x,y) = \sum_{S \subseteq E} (x-1)^{r-\operatorname{rk}_M(S)}(y-1)^{|S| - \operatorname{rk}_M(S)}.
\]

\begin{thm:4degintro}
For a matroid $M$ of rank $r$ with ground set $E$, denote
\[
t_M(x,y,z,w)=(x+y)^{-1}(y+z)^r(x+w)^{|E|-r}T_M(\frac{x+y}{y+z},\frac{x+y}{x+w}),
\]
where $T_M$ is the Tutte polynomial of $M$.
Then, we have
\[
\sum_{i+j+k+\ell=n} \Big( \int_{X_E} \alpha^i\beta^jc_k(\mathcal{S}_M^{\vee})c_\ell(\mathcal{Q}_M) \Big) x^iy^jz^kw^\ell=t_M(x,y,z,w).
\]
\end{thm:4degintro}

For a matroid $M$, note that $t_M(x,y,z,w)$ is a polynomial since the Tutte polynomial $T_M$ always has no constant term.  
Let us denote 
\[
\widetilde t_M(x,y,z,w) = \sum_{i+j+k+\ell = n}\Big( \int_{X_E} \alpha^i\beta^jc_k(\mathcal{S}_M^{\vee})c_\ell(\mathcal{Q}_M) \Big) x^iy^jz^kw^\ell.
\]
We prove $\widetilde t_M(x,y,z,w) = t_M(x,y,z,w)$ in two steps.  First, by using the matroid minor decomposition properties, we show that $\widetilde t_M(x,y,z,w)$ and $t_M(x,y,z,w)$ satisfy an identical recursive relation, which reduces the proof of \Cref{thm:4degintro} to the case where $x=y=0$. This case is precisely the content of \Cref{thm:betainvar}, which we will give an alternate proof for using a computation in \cite[Theorem 5.1]{Spe09}, together with the valuativity and duality properties of tautological Chern classes of matroids.

\medskip
We start with a recursive relation for $\widetilde t_M(x,y,z,w)$.

\begin{lem}\label{lem:geomcoalgebra}
Let $M$ be a matroid with ground set $E$, and fix any element $e\in E$. Then, one has
\begin{align*}
&\widetilde t_M(x,y,z,w) = \widetilde t_M(0,y,z,w) + x \sum_{e\in S\subsetneq E} \widetilde t_{M|S}(0,y,z,w) \ \widetilde t_{M/S}(x,0,z,w), \quad\text{and}\\
&\widetilde t_M(x,y,z,w) = \widetilde t_M(x,0,z,w) + y \sum_{\substack{S \not\ni e\\ \emptyset \subsetneq S \subsetneq E}} \widetilde t_{M|S}(0,y,z,w) \ \widetilde t_{M/S}(x,0,z,w).
\end{align*}
\end{lem}

\begin{proof}
Let us show the first statement (the second statement is proved similarly).
Recall from \Cref{rem:alphabeta} that $\alpha = \sum_{e\in S\subsetneq E} [Z_S]$, where $Z_S$ is the torus-invariant divisor of $X_E$ corresponding to the ray $\operatorname{Cone}(\overline\be_S)$ of the fan $\Sigma_E$, and recall the notation that $c(\mathcal E,u) = \sum_{i\geq 0} c_i(\mathcal E)u^i$ denotes the Chern polynomial of a $K$-class $[\mathcal E]$ with formal variable $u$.  For any integers $i\geq 1$ and $j\geq 0$, we first compute that
\[
\begin{split}
\int_{X_E} \alpha^i \beta^j c(\cS_M^\vee,z) c(\cQ_M,w) &= \int_{X_E} \sum_{e\in S\subsetneq E} [Z_S] \alpha^{i-1} \beta^j c(\cS_M^\vee,z) c(\cQ_M,w) \\
&= \sum_{e\in S\subsetneq E} \int_{Z_S} \big(\alpha^{i-1} \beta^j c(\cS_M^\vee,z) c(\cQ_M,w) \big) |_{Z_S}.
\end{split}
\]
Moreover, since $Z_S \simeq X_S \times X_{E\setminus S}$ and $A^\bullet(Z_S) \simeq A^\bullet(X_S) \otimes A^\bullet(X_{E\setminus S})$ by \Cref{prop:torusrest}, applying the matroid minors decomposition formula (\Cref{prop:tautrest} and \Cref{cor:alphabetarest}) yields that
\[
\begin{split} 
&\sum_{e\in S\subsetneq E} \int_{Z_S} \big(\alpha^{i-1} \beta^j c(\cS_M^\vee,z) c(\cQ_M,w) \big) |_{Z_S}  \\
&=\sum_{e\in S\subsetneq E} \int_{X_S\times X_{E\setminus S}}\Big( \big(1\otimes \alpha_{E\setminus S}^{i-1}\big) \big(\beta_S^j \otimes 1\big) \big(c(\cS_{M|S}^\vee,z) \otimes c(\cS_{M/S}^\vee,z)\big)\big(c(\cQ_{M|S},w) \otimes c(\cQ_{M/S},w)\big)\Big)\\
&= \sum_{e\in S\subsetneq E} \int_{X_S}\big( \beta_S^j c(\cS_{M|S}^\vee,z) c(\cQ_{M|S},w)\big)\cdot \int_{X_{E\setminus S}}\big( \alpha_{E\setminus S}^{i-1} c(\cS_{M/S}^\vee,z) c(\cQ_{M/S},w)\big).
\end{split}
\]
Thus, by rewriting $\widetilde t_M(x,y,z,w)$ as
\[
\widetilde t_M(x,y,z,w) = \int_{X_E} \Big( (1+\alpha x + \cdots + \alpha^nx^n) \cdot (1+\beta y + \cdots + \beta^ny^n) \cdot c(\cS_M^\vee,z) \cdot c(\cQ_M,w) \Big),
\]
we conclude that
\[
\widetilde t_M(x,y,z,w) = \widetilde t_M(0,y,z,w) + x \sum_{e\in S\subsetneq E} \widetilde t_{M|S}(0,y,z,w) \ \widetilde t_{M/S}(x,0,z,w),
\]
as desired.
\end{proof}

We now show that the polynomial $t_M(x,y,z,w)$ obeys the same recursive relation.

\begin{lem}
\label{lem:tuttecoalgebra}
Let $M$ be a matroid with ground set $E$, and fix an element $e\in E$. Then one has
\begin{align*}t_M(x,y,z,w)&= t_M(0,y,z,w)+x\sum_{e \in S\subsetneq E} t_{M|S}(0,y,z,w)t_{M/S}(x,0,z,w),\quad\text{and}\\
t_M(x,y,z,w)&=t_M(x,0,z,w)+y\sum_{\substack{e\not \in S\\\emptyset \subsetneq S \subsetneq E}} t_{M|S}(0,y,z,w)t_{M/S}(x,0,z,w).\end{align*}
\end{lem}

From here to the end of this subsection, we include $\emptyset$ and $E$ in summations unless otherwise stated, and allow a matroid $M$ to have an empty ground set, in which case we write $M = \emptyset$ for the unique matroid on the ground set $\emptyset$ whose set of bases is $\{\emptyset\}$.  By convention, we set $T_{\emptyset}(x,y) = 1$.

\medskip
To prove the lemma, we first borrow some notation from \cite{Coalgebra}. 
For two functions $f$ and $g$ from the set of matroids with ground sets contained in $E$ to a common ring, we define $f \ast g$ by
\[
(f\ast g)(M)=\sum_{\emptyset\subseteq A\subseteq E} f(M|A)g(M/A).
\] 
Then, one can verify that $\ast$ is associative by computing that 
$$(f_0\ast \ldots \ast f_k)(M)=\sum_{\emptyset\subseteq A_1\subseteq \ldots \subseteq A_{k} \subseteq E}f_1(M|A_1)f_2(M|A_2/A_1)\ldots f_k(M/A_{k}).$$
The function $\nu$ such that $\nu(\emptyset)=1$ and $\nu(M)=0$ for $M\ne \emptyset$ acts as the identity for $\ast$, as one easily checks
$$\nu \ast f=f\ast \nu=f \quad\text{for any $f$.}$$

We define $N_{(a,b)}(M)=a^{\rk_M}b^{\crk_M}$, where $\rk_M$ and $\crk_M$ denotes the rank and corank of $M$, respectively. This function satisfies
$$N_{(a,b)}(\emptyset)=1\text{ and } N_{(a,b)}(M)=N_{(a,b)}(M|{A})N_{(a,b)}(M/A)$$
for all $\emptyset \subseteq A\subseteq E$.
We note the following convolution formula.  (The first part appears in \cite[Section 5 Equation (3)]{Coalgebra} and the second part appears in \cite[Proposition 3.6, proof of Theorem 5.10]{Coalgebra}).

\begin{lem}We have
$$(N_{(a,b)}\ast N_{(c,d)})(M)=a^{\rk_M}d^{\crk_M}T_M(1+\frac{c}{a},1+\frac{b}{d}),$$
and in particular, denoting $\overline{N_{(a,b)}} = N_{(-a,-b)}$, we have
$$N_{(a,b)}\ast \overline{N_{(a,b)}}= \overline{N_{(a,b)}}\ast N_{(a,b)} =\nu.$$
\end{lem}
\begin{proof}
For the first part, both sides are simultaneously homogenous in $a,c$ of degree $\rk_M$ and in $b,d$ of degree $\crk_M$, so it suffices to show the equality when $a=d=1$. We have $N_{(1,b)}(M|A)=b^{|A|-\rk_M(A)}$ and $N_{(c,1)}(M/A)=c^{\rk_M-\rk_M(A)}$, so by the corank-nullity formula for the Tutte polynomial and then the definition of the convolution $\ast$, we have
$$T_M(1+c,1+b)=\sum_{\emptyset  \subseteq A \subseteq E}c^{\rk_M-\rk_M(A)}b^{|A|-\rk_M(A)}=(N_{(1,b)}\ast N_{(c,1)})(M)$$
as desired.
The second part follows since $T_M(0,0)=0$ if $M\ne \emptyset$ and $T_{\emptyset}(0,0)=1$.
\end{proof}

\begin{proof}[Proof of \Cref{lem:tuttecoalgebra}]
Write $$g_M(x,y,z,w)=(x+y)t_M(x,y,z,w)=(y+z)^r(x+w)^{|E|-r}T_M(\frac{x+y}{y+z},\frac{x+y}{x+w}),$$ so that we have to show
\begin{align}\label{eqn:yoverxy}\frac{y}{x+y}g_M(x,y,z,w)&=\sum_{e\in B}g_{M|B}(0,y,z,w)g_{M/B}(x,0,z,w),\quad\text{and}\\ \label{eqn:xoverxy}\frac{x}{x+y}g_M(x,y,z,w)&=\sum_{e\not\in B}g_{M|B}(0,y,z,w)g_{M/B}(x,0,z,w).\end{align}
Here, we used our convention for this subsection that summations include the $\emptyset$ and $E$ cases unless stated otherwise.  Now, define the functions
$$N_0 = N_{(-y-z, -y+w)}, \qquad N_1 = N_{(-z, w)}, \qquad N_2 = N_{(x-z, x+w)}.$$ Then we can directly check from the $N_{(a,b)}\ast N_{(c,d)}$ formula that $$g_M(x,y,z,w)=(\overline{N_0}\ast N_2)(M),\quad g_{M}(0,y,z,w)=(\overline{N_0}\ast N_1)(M), \quad g_M(x,0,z,w)=(\overline{N_1}\ast N_2)(M).$$
Therefore,
$$g_M(x,y,z,w)=(\overline{N_0}\ast N_2)(M)=((\overline{N_0}\ast N_1) \ast (\overline{N_1} \ast N_2))(M)=\sum_B g_{M|B}(0,y,z,w)g_{M/B}(x,0,z,w),$$
which is the sum of \eqref{eqn:yoverxy} and \eqref{eqn:xoverxy}. Hence to conclude, we only need to verify \eqref{eqn:yoverxy}. To simplify notation, for subsets $X\subseteq Y \subseteq E$ we will write $X/Y$ for $M|X/Y$, which also equals $(M/Y)|X$. We compute
\begin{align*}
&\sum_{e\in B}g_{B}(0,y,z,w)g_{M/B}(x,0,z,w)\\
&=\sum_{e\in B}\sum_{A\subseteq B \subseteq C} \overline{N_0}(A)
N_1(B/A) \overline{N_1}(C/B) N_2(M/C)\\
&=\sum_{A\subseteq A\cup e \subseteq B \subseteq C }\overline{N_0}(A)N_1((A\cup e)/A)N_1(B/(A\cup e))\overline{N_1}(C/B)N_2(M/C)\\
&=\sum_{A\subseteq A\cup e \subseteq C}
\overline{N_0}(A)N_1((A\cup e)/A)(N_1\ast \overline{N_1})(C/(A\cup e))N_2(M/C)\\
&=\sum_{A}\overline{N_0}(A)N_1((A\cup e)/A)N_2(M/(A\cup e)).
\end{align*}
When $i\in A$ we have $N_1((A\cup i)/A)=1$, and when $i\not \in A$ then $(A\cup i)/A$ is a one element rank $1$ matroid. For a $1$ element matroid $L$ we have $N_1(L) = -\frac{x}{x+y}\overline{N_0}(L) + \frac{y}{x+y} N_2(L)$ since we can check
\begin{align*}N_1(U_{0,1}) &=w=-\frac{x}{x+y}(y-w)+\frac{y}{x+y}(x+w)= -\frac{x}{x+y} \overline{N_0}(U_{0,1}) + \frac{y}{x+y} N_2(U_{0,1})\\
N_1(U_{1,1}) &=-z=-\frac{x}{x+y}(y+z)+\frac{y}{x+y}(x-z)= -\frac{x}{x+y} \overline{N_0}(U_{1,1}) + \frac{y}{x+y} N_2(U_{1,1}).\end{align*}
Therefore, we continue our computation as
\begin{align*}
&\sum_{A}\overline{N_0}(A)N_1((A\cup e)/A)N_2(M/(A\cup e))\\
&=\sum_{e\in A}\overline{N_0}(A)N_2(M/A)
-\frac{x}{x+y}\sum_{e\not \in A}\overline{N_0}(A\cup e)N_2(M/(A\cup i))+\frac{y}{x+y}\sum_{e\not\in A}\overline{N_0}(A)N_2(M/A)\\
&=\frac{y}{x+y}\sum_{e\in A}\overline{N_0}(A)N_2(M/A)+\frac{y}{x+y}\sum_{e\not \in A}\overline{N_0}(A)N_2(M/A)\\
&=\frac{y}{x+y}(\overline{N_0}\ast N_2)(M)=\frac{y}{x+y}g_M(x,y,z,w).
\end{align*}
We have thus verified \eqref{eqn:yoverxy}.
\end{proof}

\begin{proof}[Proof of \Cref{thm:4degintro}]
When the ground set $E$ has cardinality 1, the left-hand-side $\widetilde t_M(x,y,z,w)$ equals 1, and the right-hand-side $t_M(x,y,z,w)$ is also 1 because $T_{U_{1,\{0\}}}(u,v) = u$ and $T_{U_{0,\{0\}}}(u,v) = v$.
Let us now induct on the cardinality of $E$.  Let $M$ be a matroid on $E$, and assume that the desired equality holds for all matroids on ground sets with cardinality less than $|E|$.

Since $\widetilde t_M(x,y,z,w)$ and $t_M(x,y,z,w)$ satisfy the same recursive relation given in \Cref{lem:geomcoalgebra} and \Cref{lem:tuttecoalgebra}, the induction hypothesis implies that it suffices to show $\widetilde t_M(0,y,z,w) = t_M(0,y,z,w)$ and $\widetilde t_M(x,0,z,w) = t_M(x,0,z,w)$.
Applying the recursive relation and the induction hypothesis again, we find that it suffices to show $\widetilde t_M(0,0,z,w) = t_M(0,0,z,w)$.
Noting that Tutte polynomials have no constant terms, we compute that $t_M(0,0,z,w) = z^{r}w^{|E|-r}(\beta(M) \frac{1}{z}+{\beta}(M^\perp)\frac{1}{w})$.  We have thus reduced the proof to showing \Cref{thm:betainvar}, reproduced below, for which we give an alternate proof.
\end{proof}

\newtheorem*{thm:betainvar}{\Cref{thm:betainvar}}
\begin{thm:betainvar}
Let $M$ be a matroid of rank $r$ on ground set $E$.  Then,
\[
\int_{X_E} c_{r-1}(\cS_M^\vee)c_{|E|-r}(\cQ_M) = \beta(M) \quad\text{and}\quad \int_{X_E} c_{r}(\cS_M^\vee)c_{|E|-r-1}(\cQ_M) = \beta(M^\perp),
\]
where we set by convention $c_{-1}(\mathcal E) = 0$ for a $K$-class $[\mathcal E]$.
\end{thm:betainvar}

In \S\ref{sec:betavia}, we had derived \Cref{thm:betainvar} as an immediate consequence of \Cref{thm:4degintro}.  Here, we give another proof that does not rely on \Cref{thm:4degintro}, but uses a geometric computation in \cite[Theorem 5.1]{Spe09} and valuativity.

\begin{proof}[Alternate proof of \Cref{thm:betainvar} via geometry and valuativity]\label{proof:betainvarAlt}
Noting that Cremona involution is an isomorphism, one has from the matroid duality property (\Cref{prop:duality}) that
\[
\int_{X_E}c_{r}(\cS_M^\vee)c_{|E|-r-1}(\cQ_M) = \int_{X_E}\crem \Big( c_{r}(\cS_M^\vee)c_{|E|-r-1}(\cQ_M)  \Big) = \int_{X_E} c_{r}(\cQ_{M^\perp})c_{|E|-r-1}(\cS_{M^\perp}^\vee).
\]
Hence, the second equality in the theorem follows from the first, so we prove the first equality only.

When $M$ has rank 0, the Tutte polyomial $T_M(x,y)$ has no $x$ terms, so the claimed equality is satisfied.  Suppose now $r\geq 1$.  If $|E| = 1$, so that $M = U_{1,\{0\}}$, then $\int_{X_E} c_0(\cS_M^\vee) c_0(\cQ) = 1$, whereas $ \beta(M) = 1$ since $T_{U_{1,\{0\}}}(x,y) = x$.  Hence, we now suppose $|E|\geq 2$.

Because the assignment $M\mapsto c_{r-1}(\cS_M^\vee)c_{|E|-r}(\cQ_M)$ is valuative by \Cref{prop:valpolys}, and the assignment $M \mapsto \beta(M)$ is also valuative \cite[Corollary~5.7]{ardila_fink_rincon_2010},
\Cref{lem:looplessval} implies that it suffices to show the equality $\int_{X_E} c_{r-1}(\cS_M^\vee)c_{|E|-r}(\cQ_M) = \beta(M)$ for all matroids $M$ that are realizable over $\CC$.
So, let $L\subseteq \CC^E$ be a realization of a matroid $M$ of rank $r\geq 1$ with $|E|\geq 2$.  For $H\subset \CC^E$ a generic hyperplane and $\ell \subset H$ a generic line in $H$, denote by $\Omega(\ell,H)$ the Schubert variety in $\Gr(r;E)$ consisting of $L\in \Gr(r;E)$ such that $\ell\subseteq L \subseteq H$.
In \cite[Theorem 5.1]{Spe09} it is shown that
\[
\int_{\Gr(r;E)}[\overline{T\cdot L}] \cdot [\Omega(\ell,H)] = \beta(M).
\]
Note that the Chow class $[\Omega(\ell,H)]$ is equal to $c_{r-1}(\cS^\vee)c_{|E|-r}(\cQ)$, where $\cS$ and $\cQ$ are the tautological sub and quotient bundles of $\Gr(r;E)$, respectively (see for instance \cite[\S5.6.2]{EH16}).  Writing $\varphi_L\colon  X_E \to \overline{T\cdot L} \subset \Gr(r;E)$ for the map as defined in \S\ref{subsec:welldefined}, we have by the functoriality of Chern classes that
\[
\int_{X_E} c_{r-1}(\cS_L^\vee)c_{|E|-r}(\cQ_L)  = \int_{X_E} \crem\varphi_L^* [\Omega(\ell,H)].
\]
We now break into two cases.  First, suppose the matroid $M$ is disconnected, say $M = M_1 \oplus M_2$ for matroids $M_1$ and $M_2$ on nonempty ground sets.
Then, \Cref{prop:directsum1} states that $\dim P(M) < n$, so that $\dim \overline{T\cdot L} < n$.  Thus, we have $\varphi_L^*[\Omega(\ell,H)] = 0$, as the pullback of the $n$-codimensional class $[\Omega(\ell,H)]$ to $\overline{T\cdot L}$ is already 0 by dimensional reason.  We also have $\beta(M) = 0$ since $T_M(x,y) = T_{M_1}(x,y) T_{M_2}(x,y)$ and both $T_{M_1}$ and $T_{M_2}$ have no constant terms.  Now, suppose $M$ is connected, in which case \Cref{prop:directsum1} states that $\dim P(M) = \dim \overline{T\cdot L} = n$, so that $\varphi_L$ is birational onto its image.  Then, the push-pull formula implies that
\[
\int_{X_E} \crem\varphi_L^* [\Omega(\ell,H)] = \int_{\Gr(r;E)}({\varphi_L}_*[X_E]) \cdot [\Omega(\ell,H)] =  \int_{\Gr(r;E)}[\overline{T\cdot L}] \cdot [\Omega(\ell,H)] = \beta(M).
\]
Thus, we have the desired equality in both cases.
\end{proof}

\section{The tropical logarithmic Poincar\'e-Hopf theorem: representable case}
\label{sec:troplogPoinc}
A reformulation of the Poincar\'e-Hopf theorem states that the (topological) Euler characteristic $\chi(X)$ of a compact manifold is equal to the self-intersection number of its diagonal $\operatorname{diag}(X)$ in $X\times X$.  In an attempt to create a tropical analogue, Rau computed the self-intersection number of the diagonal of the Bergman class of a matroid \cite{Rau20}.

\begin{thm}[{\cite[Theorem 1.1]{Rau20}}]
\label{thm:poincarehopf}
Let $M$ be a loopless matroid of rank $r$, and let $\operatorname{diag}(\Delta_M)$ be the Minkowski weight of constant weight 1 on the diagonal copy of $\Sigma_M$ inside $\Sigma_M\times \Sigma_M$.  Then, as a tropical subcycle of $\Delta_M\times \Delta_M$, its self-intersection number is given by
\[
\deg (\operatorname{diag}(\Delta_M)^2) = (-1)^{r-1}\beta(M).
\]
\end{thm}

In \cite[Remark 1.7]{Rau20}, the author expresses a desire for a classical counterpart to \Cref{thm:poincarehopf}. The goal in this section is to provide such a classical counterpart. We give a geometric proof of \Cref{thm:poincarehopf} in the representable case, using the intuition gained from tautological bundles on matroids and reducing to a logarithmic version of the Poincar\'e Hopf theorem.

\begin{proof}[Proof of \Cref{thm:poincarehopf} when $M$ is representable]
Let $L \subseteq \CC^E$ be a realization of the matroid $M$.
The first step is to translate the tropical self-intersection $\deg (\operatorname{diag}(\Delta_M)^2) $ into a Chow-theoretic intersection. To do this, we recall that the tropical intersection $\operatorname{diag}(\Delta_M)^2$ is computed by expressing the diagonal Minkowski weight $[\operatorname{diag}(\Sigma_M)]$ as the intersection of the Minkowski weight $[\Sigma_M\times \Sigma_M]$ with $r-1$ piecewise linear functions. This is summarized in \cite[Section 2]{Rau20} and uses \cite[Proposition 3.10]{FR13}. 

Next, the tropical intersection of a weighted fan with a piecewise linear function \cite[Definition 3.4]{AR10} mirrors the intersection of the corresponding Minkowski weight with a divisor on a toric variety (\cite[Lemma 2.5]{K12} or \cite[Theorem 27]{huh2014rota}). 
Thus, to compute the intersection $\operatorname{diag}(\Delta_M)^2$, we start with $\operatorname{diag}(W_L)\subset W_L\times W_L\subset X_{E}\times X_{E}$ and perform three steps:
\begin{enumerate}
    \item Refine the fan $\Sigma_{E}^2$ of $X_{E}\times X_{E}$ to $\widetilde{\Sigma}$ so that the piecewise linear functions used in \cite[Proposition 2.6]{Rau20} are linear on each cone of the fan $\widetilde{\Sigma}$. 
    \item Take the proper transform of $\operatorname{diag}(W_M)$ and $W_L\times W_L$ in $X_{\widetilde{\Sigma}}$ to get $\widetilde{\operatorname{diag}(W_L)}$ and $\widetilde{W_L\times W_L}$.
    \item Evaluate $\int_{\widetilde{W_L\times W_L}}{[\widetilde{\operatorname{diag}(W_L)}]^2}$ in Chow theory.
\end{enumerate}
We know the final answer is independent of the choice of sufficiently refined $\widetilde{\Sigma}$ by the equivalence with the tropical intersection number, and this will also be implied by the proof below.

At this point, we will translate our question into the self-intersection of a section within the projectivization of a tautological bundle. Let $\phi$ be the map
\begin{align*}
    \phi\colon  X_{\widetilde{\Sigma}}\dashrightarrow X_{E}\times \mathbb{P}^n
\end{align*}
given on the open torus $T\times T$ by $(x,y)\mapsto (x,x^{-1}y)$. Similarly, let \begin{align*}
    \phi_{\operatorname{trop}}\colon  (\mathbb{Z}^{n+1}/\mathbb{Z}\textbf 1)\times (\mathbb{Z}^{n+1}/\mathbb{Z}\textbf 1)&\to (\mathbb{Z}^{n+1}/\mathbb{Z}\textbf 1)\times (\mathbb{Z}^{n+1}/\mathbb{Z}\textbf 1)\\
    (u,v)&\mapsto (u,-u+v).
\end{align*}
We can and will choose $\widetilde{\Sigma}$ so that it contains the fan obtained by $\phi_{\operatorname{trop}}^{-1}$ applied to the fan of $X_{E}\times \mathbb{P}^n$. This means $\phi$ is now a regular map $X_{\widetilde{\Sigma}}\xrightarrow{\phi} X_{E}\times \mathbb{P}^n$. We now claim to have the following diagram where the vertical arrows are all birational morphisms
\begin{center}
    \begin{tikzcd}
    \widetilde{\operatorname{diag}(W_L)} \ar[r, hook] \ar[d] & \widetilde{W_L\times W_L} \ar[r, hook] \ar[d] & X_{\widetilde{\Sigma}} \ar[d,"\phi"] \\
    W_L\times \{\mathbf{1}\} \ar[r, hook] & \mathbb{P}(S_L)|_{W_L}  \ar[r, hook]& X_{E}\times \mathbb{P}^n
    \end{tikzcd}
\end{center}
The two things to check are that $\widetilde{\operatorname{diag}(W_L)}$ and $\widetilde{W_L\times W_L}$ map birationally onto $W_L\times \{\mathbf{1}\}$ and $\mathbb{P}(S_L)|_{W_L}$ respectively. This is possible because it suffices to check that this is true when restricted to the open torus $\phi|_{T\times T}$ as 
$\widetilde{\operatorname{diag}(W_L)}$ and $\widetilde{W_L\times W_L}$ are irreducible. To see that $(\widetilde{W_L\times W_L})\cap (T\times T)$ maps into $\mathbb{P}(S_L)|_{W_L}$, we need our convention that the fiber of $\mathbb{P}(S_L)\to X_{E}$ over $t\in T$ is $t^{-1}\mathbb{P}(L)\subset \mathbb{P}^n$. 

To proceed, we need to know that the pullback of the Chow class $[W_L\times \{\mathbf{1}\}]$ agrees with the Chow class of the proper transform, or equivalently that
\begin{align}
    [\widetilde{\operatorname{diag}(W_L)} ] &= \phi^{*}[W_L\times \{\mathbf{1}\}].\label{eq:closurepullback}
\end{align}
To prove \eqref{eq:closurepullback}, one first notes that the wonderful compactification $W_L$ intersects the torus orbits of the permutohedral toric variety $X_{E}$ properly \cite[Theorem 6.3]{H18}. This implies $W_L\times \{\bf 1\}$ intersects the torus orbits of $X_{E}\times \mathbb{P}^n$ properly. Finally, applying the dimension count in \Cref{lem:intersectproperly} below yields \eqref{eq:closurepullback}. Alternatively, it also is possible to deduce \eqref{eq:closurepullback} from \Cref{lem:tropicalfacts}. 

Applying \eqref{eq:closurepullback} to the problem at hand, we obtain
\begin{align*}
 \int_{\widetilde{W_L\times W_L}}[\widetilde{\operatorname{diag}(W_L)} ]^2 &=\int_{\widetilde{W_L\times W_L}}\phi^{*}[W_L\times \{\mathbf{1}\}]^2 =\\
 \int_{\mathbb{P}(S_L)|_{W_L}}\phi_{*}\phi^{*}([W_L\times \{\mathbf{1}\}]^2) &= \int_{\mathbb{P}(S_L)|_{W_L}}[W_L\times \{\mathbf{1}\}]^2.
\end{align*}
At this point, one can use the formula for the class of $[W_L\times \{\mathbf{1}\}]$ as the projectivization of a subbundle \cite[Proposition 9.13]{EH16} and finish by a computation. 

Instead of doing the computation, we will present a geometric proof, connecting the the self intersection with the log-tangent sheaf and finally reducing to a logarithmic version of the Poincar\'e-Hopf theorem. To make this connection, we will need to show that
\begin{align}
\label{eq:normalRau}
    N_{(W_L\times \{\mathbf{1}\})/(\mathbb{P}(S_L)|_{W_L})}=T_{W_L}(-\log D_{W_L}).
\end{align}
To compute the normal bundle of $W_L\times \{\mathbf{1}\}$ in $\mathbb{P}(S_L)|_{W_L}$, we will express $W_L\times \{\mathbf{1}\}$ as the zero locus of a section of a vector bundle on $\mathbb{P}(S_L)|_{W_L}$. The locus $W_L\times \{\mathbf{1}\}\subset \mathbb{P}(S_L)|_{W_L}$ can be described as the locus in $\mathbb{P}(S_L)|_{W_L}$, where the universal line is parallel to $\mathbf{1}$. This is equivalently the zero locus of the map
\begin{align*}
      \mathcal{O}_{\mathbb{P}(S_L)|_{W_L}}(-1)\to \pi^{*}S_L|_{W_L}/(\mathcal{O}_{\mathbb{P}(S_L)|_{W_L}}\cdot \mathbf{1}).
\end{align*}
The target $\pi^{*}S_L|_{W_L}/(\mathcal{O}_{\mathbb{P}(S_L)|_{W_L}}\cdot \mathbf{1})$ is given taking the quotient of the inclusion of the constant section $\mathcal{O}|_{W_L}\cdot \mathbf{1}=\mathcal{O}|_{W_L}\cdot (1,\ldots,1)$ in $S_L|_{W_L}\subset \mathcal{O}_{W_L}^{n+1}$, and pulling back by the projection $\pi\colon  \mathbb{P}(S_L)|_{W_L}\to W_L$. We have already taken the quotient $S_L|_{W_L}/(\mathcal{O}_{W_L}\cdot \mathbf{1})$ in \Cref{thm:CSMGeom} and identified it as $T_{W_L}(-\log D_{W_L})$. 

Thus, $W_L\times \{\mathbf{1}\}\subset \mathbb{P}(S_L)|_{W_L}$ is the zero locus of the map
\begin{align*}
    \mathcal{O}_{\mathbb{P}(S_L)|_{W_L}}(-1)\to \pi^{*}T_{W_L}(-\log D_{W_L}),
\end{align*}
or equivalently the zero locus of a section of $\pi^{*}T_{W_L}(-\log D_{W_L})\otimes \mathcal{O}_{\mathbb{P}(S_L)|_{W_L}}(1)$.

The restriction of a vector bundle to the zero locus of a section vanishing in proper codimension is the normal bundle of that section \cite[Proposition-Definition 6.15]{EH16}. Thus, the restriction of the vector bundle $\pi^{*}T_{W_L}(-\log D_{W_L})\otimes \mathcal{O}_{\mathbb{P}(S_L)|_{W_L}}(1)$ to $W_L\times \{\mathbf{1}\}$ is the normal bundle $ N_{(W_L\times \{\mathbf{1}\})/(\mathbb{P}(S_L)|_{W_L})}$. To perform the restriction, $\mathcal{O}_{\mathbb{P}(S_L)|_{W_L}}(1)$ restricts to the trivial bundle as the universal line is constant along $W_L\times \{\mathbf{1}\}$ and $\pi^{*}T_{W_L}(-\log D_{W_L})$ restricts to $T_{W_L}(-\log D_{W_L})$ as $W_L\times \{\mathbf{1}\}$ maps isomorphically to $W_L$ under $\pi$. Therefore, $\pi^{*}T_{W_L}(-\log D_{W_L})\otimes \mathcal{O}_{\mathbb{P}(S_L)|_{W_L}}(1)$ restricted to $W_L\times \{\mathbf{1}\}$ is $T_{W_L}(-\log D_{W_L})$, concluding our proof of \eqref{eq:normalRau}. 

Finally, 
\begin{align*}
    \int_{\mathbb{P}(S_L)|_{W_L}}[W_L\times \{\mathbf{1}\}]^2 = c_{\operatorname{top}}(N_{(W_L\times \{\mathbf{1}\})/\mathbb{P}(S_L)})
\end{align*}
by \cite{LMS75}, and by \eqref{eq:normalRau}, 
\begin{align*}
    c_{\operatorname{top}}(N_{(W_L\times \{\mathbf{1}\})/\mathbb{P}(S_L)})=c_{\operatorname{top}}(T_{W_L}(-\log D_{W_L})).
\end{align*}
The top Chern class $c_{\operatorname{top}}(T_{W_L}(-\log D_{W_L}))$ is $c_{r-1}(\mathcal{S_L}|_{W_L})$ by \Cref{thm:CSMGeom} and $\int_{X_E} c_{r-1}(\mathcal{S_L}|_{W_L}) = \int_{X_E} c_{r-1}(\cS_L)c_{|E|-r}(\cQ_M)$ is equal to $(-1)^{r-1}\beta(M)$ by \Cref{thm:betainvar}. 
\end{proof}

We chose to conclude $c_{\operatorname{top}}(T_{W_L}(-\log D_{W_L}))=(-1)^{r-1}\beta(M)$ using the framework given in this paper to be self-contained. However, there is a more classical approach to get the same result given in \Cref{rem:betainvarByLog}, which uses a logarithmic version of the Poincar\'e Hopf theorem to relate the Chern class to the topological Euler characteristic of the hyperplane arrangement complement $W_L\backslash D_{W_L}$. 

\medskip
\Cref{lem:intersectproperly} below was used in the proof of \Cref{thm:poincarehopf} in the representable case as a substitute for \Cref{lem:tropicalfacts}, giving a geometric proof independent of tropical methods. 

\begin{lem}
\label{lem:intersectproperly}
Let $Y\subset T$ be an irreducible subvariety of a torus. Let $X_{\Sigma}$ be a smooth toric variety compactifying $T$ and $\overline{Y}$ be the closure of $Y$ in $X_{\Sigma}$. Suppose $\overline{Y}$ intersects each torus orbit in $X_{\Sigma}$ properly. Then, the following statement holds:

Let $\widetilde{\Sigma}$ be a unimodular fan refining $\Sigma$, and $\pi\colon  \widetilde{X}\to X$ be the corresponding map of toric varieties. Then, $\pi^{-1}(\overline{Y})$ is equal to the closure $\overline{\pi^{-1}(Y)}$ in $X_{\widetilde{\Sigma}}$. In particular, $\pi^{*}[\overline{Y}]=[\overline{\pi^{-1}(Y)}]$ in $A^{\bullet}(\pi^{-1}(\overline{Y}))=A^{\bullet}(\overline{\pi^{-1}(Y)})$, which implies equality in $A^{\bullet}(X_{\widetilde{\Sigma}})$. 
\end{lem}

\begin{proof}
We clearly have $\pi^{-1}(\overline{Y})\supset \overline{\pi^{-1}(Y)}$. To show the reverse inclusion, we first show $\dim(\pi^{-1}(\overline{Y}\backslash Y))<\dim(Y)$. To do this, we will show for all positive-dimensional cones $\sigma$ in $\Sigma$ and the corresponding torus orbit $O_\sigma$, we must have $\dim(\pi^{-1}(O_\sigma\cap \overline{Y}))<\dim(Y)$. 

By the assumption that $\overline{Y}$ intersects each torus orbit of $X_{\Sigma}$ properly, $\dim(O_\sigma\cap \overline{Y})\leq \dim(Y)-\dim(\sigma)$, where either equality holds or the intersection $O_\sigma\cap \overline{Y}$ is empty, in which case the dimension is understood to be $-\infty$. By \cite[Proposition 2.14]{HLY02}, the fibers over $O_\sigma$ under $\pi\colon  \widetilde{X}\to X$ have dimension at most $\dim(\sigma)-1$. Therefore, 
\begin{align*}
    \dim(\pi^{-1}(O_\sigma\cap \overline{Y}))\leq \dim(O_\sigma\cap \overline{Y})+\dim(\sigma)-1 < \dim(Y).
\end{align*}

To finish, we first note that every irreducible component of $\pi^{-1}(\overline{Y})$ has dimension at least $\dim(Y)$ \cite[Theorem 0.2]{EH16}, as $\pi^{-1}(\overline{Y})$ can be expressed as the intersection between the graph of $\pi$ and $X_{\widetilde{\Sigma}}\times \overline{Y}$ inside the smooth variety $X_{\widetilde{\Sigma}}\times X_{\Sigma}$. Next, $\pi^{-1}(\overline{Y})=\pi^{-1}(Y)\sqcup \bigcup_{\sigma}\pi^{-1}(O_\sigma\cap \overline{Y})$, where the union is over all positive-dimensional cones $\sigma$ in $\Sigma$. Since we have just shown that $\pi^{-1}(O_\sigma\cap \overline{Y})<\dim(Y)$, $\pi^{-1}(\overline{Y})$ must be irreducible. Since $\pi^{-1}(\overline{Y})$ is an irreducible variety containing $\overline{\pi^{-1}(Y)}$ and their dimensions agree, $\pi^{-1}(\overline{Y})=\overline{\pi^{-1}(Y)}$. 

To deduce $\pi^{*}[\overline{Y}]=[\overline{\pi^{-1}(Y)}]$, we note that $\pi^{*}[\overline{Y}]$ is a well-defined class in $A^{\bullet}(\pi^{-1}(\overline{Y}))=A^{\bullet}(\overline{\pi^{-1}(Y)})$ by construction of the cap product using \cite[Definition 8.1.2]{Ful98}, so it must be the fundamental class $[\overline{\pi^{-1}(Y)}]$. 
\end{proof}

\section{Global Chern roots}
In this section we show that one can decompose tautological $K$-classes of matroids as sums of classes of line bundles.
The construction of these decompositions are analogous to considering successive quotients in filtrations of tautological bundles of Grassmannians, and likewise are not canonical.
Moreover, they are not directly applicable for proving positivity statements because the line bundle summands are generally not nef.  However, they relate the tautological $K$-classes of matroids to certain constructions in previous works \cite{FR13,huh2014rota,EurThesis}.  Also, they are useful in computer computations, for instance in Macaulay2 \cite{M2}, which has been helpful for the development of this paper.

\medskip
Let $M$ be a matroid of rank $r$ on $E$.
Fix a sequence $\MM = (M_0, \ldots, M_{n+1})$ consisting of matroids $M_i$ of rank $i$ on $E$ such that $M_r = M$ and $B_\sigma(M_i) \subset B_\sigma(M_{i+1})$ for all permutations $\sigma\in \mathfrak S_E$ and $i = 0, \ldots, n$.
Such a sequence $\MM$ is known as a ``full flag matroid that lifts $M$'' \cite[Ch.~1]{BGW03}.  
One such $\MM$ is the ``full Higgs lift'' of $M$ which is obtained by defining
\[
\text{the set of bases of } M_i = \left\{S \in  \textstyle\binom{E}{i} \ \middle| \ \text{$S$ contains or is contained in a basis of $M$}\right\}
\]
for all $0 \leq i \leq  n+1$.  For each $0 \leq i \leq n$, we express the differences $[\cS_{M_{i+1}}] - [\cS_{M_i}]$ and $[\cQ_{M_{i}}] -[\cQ_{M_{i+1}}]$ as $K$-classes of line bundles as follows.
As denoted in \S\ref{subsec:Tdiv}, 
let $\mathcal O(D_P)$ be the $T$-equivariant line bundle of $X_E$ corresponding to a lattice polytope $P\subset \RR^E$ whose normal fan coarsens $\widetilde\Sigma_E$.
For a matroid $N$ with ground set $E$, we then have by the discussion in \S\ref{subsec:Tdiv} that
\[
[\mathcal O(D_{-P(N)})]_\sigma = \prod_{i \in B_{\sigma}(N)} T_i  \quad\text{and}\quad [\mathcal O(D_{P(N^\perp)})]_\sigma = \prod_{i \not\in B_{\sigma}(N)} T_i^{-1}.
\]
Thus, since $B_\sigma(M_i) \subset B_\sigma(M_{i+1})$ for all $0\leq i\leq n$ and permutations $\sigma$, we have that
\[
\begin{split}
[\cS_{M_{i+1}}] - [\cS_{M_i}] & = [\mathcal O(D_{-P(M_{i+1})})^\vee \otimes \mathcal O(D_{-P(M_i)})] \qquad\text{and}\\
 [\cQ_{M_{i}}] -[\cQ_{M_{i+1}}] & = [\mathcal O(D_{P({M_i}^\perp)}) \otimes \mathcal O(D_{P({M_{i+1}}^\perp)})^\vee].
\end{split}
\]
In particular, since $M_0 = U_{0,E}$ and $M_{n+1} = U_{n+1, E}$ so that $[\cS_{M_0}] = [\cQ_{M_{n+1}}] = 0$, we have that
\begin{equation}\label{eq:ChernRoots}
\begin{split}
&[\cS_M] = \sum_{i=0}^{r-1} [\mathcal O(D_{-P(M_{i+1})})^\vee \otimes \mathcal O(D_{-P(M_{i})})] \qquad\text{and}\\
&[\cQ_M] = \sum_{j=r}^{|E|-1} [\mathcal O(D_{P({{M_{j}}^\perp})}) \otimes \mathcal O(D_{P({{M_{j+1}}^\perp})})^\vee] \qquad\text{as elements in $K_T^0(X_E)$.}
\end{split}
\end{equation}

One might hope that this decomposition implied positivity properties of $[\mathcal{S}_M^{\vee}]$ and $[\mathcal{Q}_M]$. However, for example for $[\mathcal{S}_M^{\vee}]$, the line bundles $\mathcal O(D_{-P(M_{i+1})}) \otimes \mathcal O(D_{-P(M_{i})})^\vee$ is nef if and only if $P(M_i)$ is a weak Minkowski summand of $P(M_{i+1})$ --- see \cite[\S2.2 \& \S2.4]{ACEP20} for a proof.  This however seldom occurs:  When a matroid $M$ is connected after removing its loops and coloops, the only weak Minkowski summand of $P(M)$ is itself \cite{Ngu78}.
Hence, although the bundles $\cS_L^\vee$ and $\cQ_L$ are globally generated and hence nef if $L$ is a realization of $M$, it is unclear how to establish from the Chern roots listed here that the positivity property of $\cS_L^\vee$ and $\cQ_L$ as nef vector bundles persist for arbitrary (not necessarily realizable) matroids.

\begin{rem}
\label{rem:modular}
Let $z_S \in A^1(X_E)$ denote the divisor class of the torus-invariant divisor $Z_S\subset X_E$ corresponding to a nonempty proper subset $S$ of $E$, and denote $z_E = -\alpha \in A^1(X_E)$.  Combining \Cref{rem:alphabeta} with a well-known description for the Chow ring of a smooth complete toric variety (see for instance \cite[Ch.~5]{Ful93}), one has that the Chow ring of the permutohedral variety has a presentation
\[
A^\bullet(X_E) =\frac{\ZZ[z_S \mid \emptyset \subsetneq S \subseteq E]}{\langle z_Sz_{S'} \mid S\not\subseteq S' \text{ and } S\not\supseteq S'\rangle + \langle \sum_{i\in S \subseteq E} z_S \mid i \in E\rangle}.
\]
Note that in this presentation, one has $\sum_{\emptyset\subsetneq S\subseteq E} z_S = \beta$ because it follows from the end of \Cref{rem:alphabeta} that $\alpha+ \beta = \sum_{\emptyset\subsetneq S \subsetneq E} z_S$.
For a matroid $N$ of rank $r$ on $E$, the translate $P' = -P(N) + r\be_0$ of its base polytope lies in the lattice $\mathbf 1^\perp$.  The support function $h_{P(N)}(x) = \max_{\bm \in P(M)}\langle \bm, x \rangle$ of the base polytope satisfies $h_{P(N)}(\be_S) = \rk_M(S)$ for any subset $S\subseteq E$, and hence the support function $h_{P'}$ of the translate $P'$ satisfies $h_{P'}(\be_S) = \rk_N(S) - r$ if $0\in S$ and $h_{P'}(\be_S) = \rk_N(S)$ otherwise.
Thus, by the discussion in \S\ref{subsec:Tdiv} and the fact that $\alpha = \sum_{0\in S\subsetneq E} z_S$ (\Cref{rem:alphabeta}), one has
\[
\sum_{\emptyset\subsetneq S \subseteq E} \operatorname{rk}_N(S) z_S  =  [D_{-P(N)}]  = [D_{P(N^\perp)}] \quad\text{as elements in }A^1(X_E),
\]
where the last equality follows from the fact that $P(N^\perp)$ and $-P(N)$ are translates $P(N^\perp) = -P(N) + \mathbf 1$ of each other.
In particular, one can restate the decomposition of $[\cS_M]$ and $[\cQ_M]$ into sums of line bundles in \Cref{eq:ChernRoots} by stating that a possible collection of Chern roots for $[\cS_M]$ and $[\cQ_M]$ is given by
\[
\begin{split}
\text{Chern roots of } [\cS_M] &= \Big\{ \sum_{\emptyset\subsetneq S\subseteq E} \big(-\operatorname{rk}_{M_{i+1}}(S) + \operatorname{rk}_{M_{i}}(S) \big)z_S \Big\}_{i = 0, \ldots, r-1} \quad\text{and}\\
\text{Chern roots of } [\cQ_M] &= \Big\{ \sum_{\emptyset\subsetneq S\subseteq E} \big(-\operatorname{rk}_{M_{i+1}}(S) + \operatorname{rk}_{M_{i}}(S) \big)z_S \Big\}_{i = r, \ldots, n}.
\end{split}
\]
The ``modular filter'' of two consecutive matroids $M_i$ and $M_{i+1}$ in the sequence $\MM$ is defined as the collection $\mathscr F_i = \{S\subseteq E \mid \operatorname{rk}_{M_{i+1}}(S) - \operatorname{rk}_{M_i}(S) = 1\}$.  Writing $\alpha_{\mathscr F_i} = \sum_{\substack{S\in \mathscr F_i\\ \emptyset\subsetneq S \subsetneq E}} z_S$, we have that the elements $\alpha - \alpha_{\mathscr F_i}$ for various $i$ give a collection of Chern roots for $[\cS_M]$ and $[\cQ_M]$.  These elements $\alpha - \alpha_{\mathscr F_i}$ appeared previously in \cite[Remark 31]{huh2014rota} and \cite{FR13}, and the interpretation of them as Chern roots of a $K$-class first appeared in \cite[Remark 7.2.6]{EurThesis}.  The elements $\alpha - \alpha_{\mathscr F_i}$ when $\mathscr F_i$ are principal filters were studied in \cite{BES} to give a generalization of a volume formula for generalized polyhedra \cite[Corollary 9.4]{P09} and a simplified proof for the portion of the Hodge theory of matroids in \cite{AdiHuhKatz} relevant to log-concavity.
\end{rem}

\bibliography{BEST_TautClMat_arXiv}
\bibliographystyle{amsalpha}

\end{document}